\documentclass[10pt]{scrarticle} 
\usepackage[utf8]{inputenc}

\usepackage{tikz}
\usetikzlibrary{matrix,arrows}
\usepackage{tikz-cd}
\usetikzlibrary{cd}
\usepackage{bbm}
\usepackage{bm} 
\usepackage{amssymb, url}
\usepackage{amsthm}
\usepackage{amsmath}
\usepackage{amsfonts}
\usepackage{stmaryrd} 
\usepackage{enumitem}
\usepackage{comment}
\usepackage{multicol}
\usepackage{hyperref} 
\usepackage[top=3cm, bottom=3cm, left=2.4cm, right=2.4cm]{geometry} 

\hypersetup{urlcolor=blue,linkcolor=cyan,citecolor=cyan,colorlinks=true}
\usepackage[symbol=$\uparrow~$,numberlinked=true]{footnotebackref}

\usepackage{faktor}

\theoremstyle{plain} 
\newtheorem{Theorem}{Theorem}[section]
\newtheorem*{Theorem*}{Theorem}
\newtheorem{Lemma}[Theorem]{Lemma}

\newtheorem{Corollary}[Theorem]{Corollary}
\newtheorem*{Corollary*}{Corollary}
\newtheorem{Proposition}[Theorem]{Proposition}
\newtheorem*{Proposition*}{Proposition}

\newtheorem{MainTheorem}{Theorem}

\theoremstyle{definition}
\newtheorem{Definition}[Theorem]{Definition}
\newtheorem{Proposition-Definition}[Theorem]{Proposition-Definition}
\newtheorem*{Definition*}{Definition}
\newtheorem{Example}[Theorem]{Example}
\newtheorem{Conjecture}[Theorem]{Conjecture}
\newtheorem*{Conjecture*}{Conjecture}

\newtheorem*{Question*}{Question}

\theoremstyle{remark}
\newtheorem{Remark}[Theorem]{Remark}
\newtheorem*{Remark*}{Remark}

\newcommand{\coker}{\operatorname{coker}}

\newcommand{\im}{\operatorname{im}}

\newcommand{\Spec}{\operatorname{Spec}}
\newcommand{\Spm}{\operatorname{Spm}}

\newcommand{\Gr}{\operatorname{Gr}}

\newcommand{\Res}{\operatorname{Res}}

\newcommand{\rank}{\operatorname{rank}}

\newcommand{\Hom}{\operatorname{Hom}}

\newcommand{\RHom}{\operatorname{RHom}}
\newcommand{\Ext}{\operatorname{Ext}}
\newcommand{\Tor}{\operatorname{Tor}}

\newcommand{\id}{\operatorname{id}}

\newcommand{\Gal}{\operatorname{Gal}}
\newcommand{\Frob}{\operatorname{Frob}}

\usepackage[mathscr]{eucal} 
\newcommand{\eC}{\mathscr{C}}
\newcommand{\eA}{\mathscr{A}}

\newcommand{\eD}{\mathscr{D}}

\newcommand{\eH}{\mathscr{H}}
\newcommand{\eM}{\mathscr{M}}

\newcommand{\cHom}{\eH\!\operatorname{om}} 

\newcommand{\bC}{\mathbb{C}}

\newcommand{\bF}{\mathbb{F}}
\newcommand{\bG}{\mathbb{G}}

\newcommand{\bN}{\mathbb{N}}

\newcommand{\bP}{\mathbb{P}}
\newcommand{\bQ}{\mathbb{Q}}
\newcommand{\bR}{\mathbb{R}}

\newcommand{\bZ}{\mathbb{Z}}

\usepackage{calrsfs}
\usepackage{calligra}

\DeclareMathAlphabet{\pazocal}{OMS}{zplm}{m}{n} 

\newcommand{\cA}{\pazocal{A}}
\newcommand{\cB}{\pazocal{B}}

\newcommand{\cF}{\pazocal{F}}

\newcommand{\cM}{\eM}

\newcommand{\cO}{\pazocal{O}}

\newcommand{\cV}{\pazocal{V}}

\newcommand{\fa}{\mathfrak{a}}

\newcommand{\fd}{\mathfrak{d}}

\newcommand{\fj}{\mathfrak{j}}

\newcommand{\fm}{\mathfrak{m}}

\newcommand{\fp}{\mathfrak{p}}

\newcommand{\fq}{\mathfrak{q}}

\newcommand{\fv}{\mathfrak{v}}

\numberwithin{equation}{section}

\title{On the Integral Part of $A$-Motivic Cohomology}
\author{Quentin Gazda\thanks{\textbf{ Affiliation:} CMLS, \'Ecole Polytechnique, Palaiseau, France \\ \textbf{ Mail:} quentin@gazda.fr.}}
\date{}

\def\temp{&} \catcode`&=\active \let&=\temp 

\begin{document}

\maketitle

\begin{center}
\textbf{Abstract}
\end{center}
The deepest arithmetic invariants attached to an algebraic variety defined over a number field $F$ are conjecturally captured by the \emph{integral part} of its \emph{motivic cohomology}. There are essentially two ways of defining it when $X$ is a smooth projective variety: one is via the $K$-theory of a regular integral model, the other is through its $\ell$-adic realization. Both approaches are conjectured to coincide. 

This paper initiates the study of motivic cohomology for global fields of positive characteristic, hereafter named \emph{$A$-motivic cohomology}, where classical mixed motives are replaced by mixed Anderson $A$-motives. Our main objective is to set the definitions of the \emph{integral part} and the \emph{good $\ell$-adic part} of the $A$-motivic cohomology using Gardeyn's notion of \emph{maximal models} as the analogue of regular integral models of varieties. Our main result states that the \emph{integral part} is contained in the \emph{ good $\ell$-adic part}. As opposed to what is expected in the number field setting, we show that the two approaches do not match in general.

We conclude this work by introducing the submodule of \emph{regulated extensions} of mixed Anderson $A$-motives, for which we expect the two approaches to match, and solve some particular cases of this expectation. 

{\footnotesize
\tableofcontents
}

\section{Introduction}\label{chap:introduction}
\subsection{The number field picture}\label{section:numberfieldspicture}
The idea of \emph{mixed motives} and \emph{motivic cohomology} has been gradually formulated by Deligne, Beilinson and Lichtenbaum and aims to extend Grothendieck's philosophy of pure motives. Before discussing the function fields side, subject of the paper, let us first present the classical setting to derive some motivations. \\

The theory, mostly conjectural, starts with a number field $F$. The hypothetical landscape portrays a $\bQ$-linear Tannakian category $\cM\cM_F$ of \emph{mixed motives} over $F$, equipped with several \emph{realization functors} having $\cM\cM_F$ as source (see \cite[\S 1]{deligne-droite}). Among them, the \emph{$\ell$-adic realization functor} $V_{\ell}$, for a prime number $\ell$, takes values in the category of continuous $\ell$-adic representations of the absolute Galois group $G_F=\Gal(\bar{F}|F)$, given a fixed algebraic closure $\bar{F}$ of $F$.

It is expected that \emph{reasonable} cohomology theories factor through the category $\cM\cM_F$. For instance, the $\ell$-adic realization should recover the \'etale cohomology of algebraic varieties with coefficients in $\bQ_{\ell}$ in the following way:  for all integer $i$, one foresees the existence of a functor $h^i$, from the category of algebraic varieties over $F$ to $\cM\cM_F$, making the following diagram of categories commute: 
\begin{equation}
\begin{tikzcd}
\{\text{Varieties}/F\} \arrow[r,"h^i"] \arrow[dr,"X\mapsto \operatorname{H}^i_{\text{\'et}}(X\times_F \bar{F}{,}\bQ_{\ell})"'] & \cM\cM_F \arrow[d,"V_{\ell}"]\\
& \operatorname{Rep}_{\bQ_{\ell}}(G_F) 
\end{tikzcd} \nonumber
\end{equation}
According to Deligne \cite[\S 1.3]{deligne-droite}, the category $\cM\cM_F$ should admit a \emph{weight filtration} in the sense of Jannsen \cite[Def 6.3]{jannsen}, which would coincide with the classical weight filtration of varieties. The \emph{weights} of a mixed motive $M$ would then be defined as the breaks of its weight filtration.\\

From the Tannakian formalism, $\cM\cM_F$ admits a tensor operation extending the fiber product on varieties. We fix $\mathbbm{1}$ a neutral object. Let $M$ be a mixed motive over $F$. According to Beilinson \cite[\S 0.3]{beilinson} (see also \cite[Def 17.2.11]{andre}), the \emph{motivic cohomology of $M$} is defined as the complex 
\[
\RHom_{\cM\cM_F}(\mathbbm{1},M)
\]
in the derived category of $\bQ$-vector spaces. Its $i$th cohomology is the $\bQ$-vector space $\Ext^i_{\cM\cM_F}(\mathbbm{1},M)$, the space of $i$-fold extensions of $\mathbbm{1}$ by $M$ in $\cM\cM_F$. We quote from~\cite[\S2]{scholl} and \cite[\S1.3]{deligne-droite} respectively:
\begin{Conjecture*}
We expect that:
\begin{enumerate}[label=(C$\arabic*$)]
\item\label{item:conjectureC1} for $i\not\in \{0,1\}$, $\Ext^i_{\cM\cM_F}(\mathbbm{1},M)=0$,
\item\label{item:conjectureC2} if $w$ denotes the smallest weight of $M$ and $w\geq 0$, then  $\Ext^1_{\cM\cM_F}(\mathbbm{1},M)=0$.
\end{enumerate}
\end{Conjecture*}

Let us focus on the $\bQ$-vector space of $1$-fold extensions $\Ext^1_{\cM\cM_F}(\mathbbm{1},M)$. A subspace thereof of fundamental importance is the space of \emph{extensions having everywhere good reduction}. In the literature, we encounter two definitions which are expected to give the same result. Let us first describe its local constructions.

\paragraph{Via the $\ell$-adic realization:}
Let $F_{\fp}$ be the local field of $F$ at a finite place $\fp$, and let $M_{\fp}$ be a mixed motive over $F_{\fp}$. Let $G_{\fp}$ be the absolute Galois group of $F_{\fp}$, and let $I_{\fp}$ be its inertia subgroup. Given a prime number $\ell$, one predicts that the $\ell$-adic realization $V_{\ell}$ is an exact functor. This allows one to construct a $\bQ$-linear morphism, called \emph{the $\ell$-adic realization map of $M_{\fp}$},
\begin{equation}
r_{M,\ell,\fp}:\Ext^1_{\cM\cM_{F_{\fp}}}(\mathbbm{1}_{\fp},M_{\fp})\longrightarrow \operatorname{H}^1(G_{\fp},V_{\ell}M_{\fp}) \nonumber
\end{equation}
which maps the class of an exact sequence $[E_{\fp}]:0\to M_{\fp}\to E_{\fp}\to \mathbbm{1}_{\fp}\to 0$ in $\cM\cM_{F_{\fp}}$ to the class of the continuous cocycle $c:G_{\fp}\to V_{\ell}M_{\fp}$ associated to the class of the exact sequence $[V_\ell E_{\fp}]:0\to V_{\ell}M_{\fp} \to V_{\ell}E_{\fp} \to V_{\ell}\mathbbm{1}_{\fp}\to 0$ in $\operatorname{Rep}_{\bQ_{\ell}}(G_{\fp})$.

Suppose $\ell$ does not divide $\fp$. Following Scholl \cite{scholl}, we say that $[E_{\fp}]\in \Ext^1_{\cM\cM_{F_{\fp}}}(\mathbbm{1},M)$ has \emph{good reduction} if $r_{M,\ell,\fp}([E_{\fp}])$ splits as a representation of $I_{\fp}$ (that is, $[V_\ell E_{\fp}]$ is zero in $\operatorname{H}^1(I_{\fp},V_{\ell}M_{\fp})$). In \cite[\S 2  Rmk]{scholl}, Scholl conjectures:
\begin{Conjecture*}
We expect that:
\begin{enumerate}[label=(C$3$)]
\item\label{item:conjectureC3} The property that $[E_{\fp}]$ has good reduction is independent of the prime $\ell$.
\end{enumerate}
\end{Conjecture*}

We then define $\Ext^1_{\text{good}}(\mathbbm{1},M_{\fp})_{\ell}$ as the subspace of $\Ext^1_{\cM\cM_F}(\mathbbm{1},M_{\fp})$ consisting of extensions having good reduction. By \ref{item:conjectureC3}, it should not depend on $\ell$: this is the \emph{$\fp$-integral part} of the motivic cohomology of $\underline{M}_\fp$. 

\paragraph{Via the $K$-theory of regular models:} 
An other conjectural way of defining the \emph{$\fp$-integral part} of motivic cohomology uses its expected link with $K$-theory. Following Beilinson \cite{beilinson} in the case where $M_{\fp}$ is of the form $h^{i-1}(X)(n)$ for a smooth projective variety $X$ over $F_{\fp}$ and two integers $n$, $i\geq 1$, there should be a natural isomorphism of $\bQ$-vector spaces (see \emph{loc.~cit.} for details):
\begin{equation}\label{eq:mc-via-K-theory}
\Ext^1_{\cM\cM_{F_{\fp}}}(\mathbbm{1},M_{\fp})\stackrel{\sim}{\longrightarrow} (K_{2n-i}(X)\otimes_{\bZ}\bQ)^{(n)}. 
\end{equation}
Assume that $X$ has a regular model $\mathcal{X}$ over $\cO_{\fp}$ (\emph{i.e.} $\mathcal{X}$ is regular over $\Spec \cO_{\fp}$ and $\mathcal{X}\times_{\Spec \cO_{\fp}}\Spec F_{\fp}=X$). Then, we define $\Ext^1_{\cO_{\fp}}(\mathbbm{1},M_{\fp})$ to be the inverse image of 
\[
\operatorname{image}\left((K_{2n-i}(\mathcal{X})\otimes_{\bZ}\bQ)^{(n)}\longrightarrow (K_{2n-i}(X)\otimes_{\bZ}\bQ)^{(n)}\right)
\]
through \eqref{eq:mc-via-K-theory}. By  \cite[Lem. 8.3.1]{beilinson}, this does not depend on the choice of the model $\mathcal{X}$. The next conjecture supersedes \ref{item:conjectureC3}:
\begin{Conjecture*}
We expect that
\begin{enumerate}[label=(C$4$)]
\item\label{item:conjectureC4} For any prime $\ell$ not under $\fp$, we have $\Ext^1_{\cO_{\fp}}(\mathbbm{1},M_{\fp})=\Ext^1_{\text{good}}(\mathbbm{1},M_{\fp})_{\ell}$.
\end{enumerate}
\end{Conjecture*}

The global version of the integral part of the motivic cohomology of a mixed motive $M$ over $F$ is recovered as follows. The motive $M$ induces a motive $M_{\fp}$ over $F_{\fp}$ by localization. Assuming conjecture \ref{item:conjectureC3}, we say that an extension $[E]$ (of $\mathbbm{1}$ by $\underline{M}$) has \emph{everywhere good reduction} if, for all $\fp$, the extension $[E_\fp]$ belongs to $\Ext^1_{\cM\cM_{F_{\fp}}}(\mathbbm{1}_{\fp},M_{\fp})_{\ell}$ for some prime $\ell$ not dividing $\fp$. We denote by $\Ext^1_{\text{good}}(\mathbbm{1},M)$ the subspace of $\Ext^1_{\cM\cM_F}(\mathbbm{1},M)$ consisting of extensions having everywhere good reduction.

Similarly, in the case where $M=h^{i-1}(X)(n)$ for a smooth projective variety $X$ over $F$, we let $\Ext^1_{\cO_F}(\mathbbm{1},M)$ be the subspace of extensions $[E]$ such that $[E_{\fp}]$ belongs to $\Ext^1_{\cO_\fp}(\mathbbm{1}_\fp,M_\fp)$ for all finite places $\fp$ of $F$. In virtue of the previous conjectures, we should have:
\[
\Ext^1_{\text{good}}(\mathbbm{1},M)=\Ext^1_{\cO_F}(\mathbbm{1},M)\cong (K_{2n-i}(\mathcal{X})\otimes_{\bZ}\bQ)^{(n)},
\]
where $\mathcal{X}$ is a regular model of $X$ over $\cO_F$.

The space $\Ext^1_{\cO_F}(\mathbbm{1},M)$ is at the heart of Beilinson's conjectures, the next expectation being the starting point thereof:
\begin{Conjecture*}
We expect that
\begin{enumerate}[label=(C$5$)]
\item\label{item:conjectureC5} The space $\Ext^1_{\cO_F}(\mathbbm{1},M)$ has finite dimension over $\mathbb{Q}$.
\end{enumerate}
\end{Conjecture*}

\subsection{The function field picture}
Despite its intrinsic obscurities, motivic cohomology remains a difficult subject also because its definition sits on a completely conjectural framework. The present paper grew out as an attempt to understand the analogous picture in function field arithmetic. There, the theory looks more promising using \emph{Anderson $A$-motives}, instead of classical motives, whose definition is well-established. This parallel has been drawn by many authors and led to celebrated achievements. The analogue of the Tate conjecture \cite{taguchi} \cite{tamagawa}, of Grothendieck's periods conjecture \cite{papanikolas} and of the Hodge conjecture \cite{hartl-juschka} are now theorems on the function fields side. The recent volume \cite{motif} records some of these feats. Counterparts of Motivic cohomology in function field arithmetic have not been studied yet, although recent works of Taelman \cite{taelman-dirichlet} \cite{taelman-t-motif} and Mornev \cite{mornev-shtuka} strongly suggest the pertinence of such a project.

\subsubsection*{The setting}
Let $\bF$ be a finite field, $q$ its number of elements, and let $(C,\cO_C)$ be a geometrically irreducible smooth projective curve over $\bF$. We let $K$ be the function field of $C$ and we fix a closed point $\infty$ on $C$. The $\bF$-algebra:
\begin{equation}
A:=\cO_C(C\setminus\{\infty\}) \nonumber
\end{equation}
has $K$ as its field of fractions. We let $K_{\infty}$ be the completion of $K$ with respect to the valuation $v_{\infty}$ associated to $\infty$, and let $K_{\infty}^s$ be the separable closure of $K_{\infty}$. The analogy with number fields that guide us in this text is:
\begin{center}
\begin{tabular}{lccccccccccc}
Number fields: & $\bZ$ & $\subset$ & $\bQ$ & $\subset$ & $\bR$ & $\subset$ & $\bC$ \\
& $\wr$ & & $\wr$ & & $\wr$ & & $\wr$ \\
Function fields: & $A$ & $\subset$ & $K$ & $\subset$ & $K_{\infty}$ & $\subset$ & $K_{\infty}^s$ 
\end{tabular}
\end{center}

Let $R$ be an $A$-algebra. The analogy with number fields disappears when one considers the tensor product $A\otimes R$, which is at the heart of the definition of Anderson $A$-motives (unlabeled fiber and tensor products are over $\bF$). We consider the ring endomorphism $\tau$ of $A\otimes R$ which acts as the identity on $A$ and as the $q$-Frobenius on $R$. We let $\fj$ be the ideal of $A\otimes R$ generated by the set $\{a\otimes 1-1\otimes a|a\in A\}$. \\

Following \cite{anderson}, an \emph{Anderson $A$-motive $\underline{M}$ over $R$} is a pair $(M,\tau_M)$ where $M$ denotes a finite projective $A\otimes R$-module of constant rank, and where $\tau_M:(\tau^*M)[\fj^{-1}]\to M[\fj^{-1}]$ is an $(A\otimes R)[\fj^{-1}]$-linear isomorphism (Definition \ref{def:Amotives}). We let $\cM_R$ denote the category of Anderson $A$-motives with obvious morphisms. $\cM_R$ is known to be $A$-linear, rigid monoidal, and is exact in the sense of Quillen (Proposition \ref{prop:MF-is-exact}) but not abelian (\cite[\S 2.3]{hartl-juschka} or Subsection \ref{section:definitions}). Let $\mathbbm{1}$ in $\cM_R$ be a neutral object for the tensor operation.

\subsubsection*{Extensions of $A$-motives}
The category $\cM_R$, or rather full subcategories of it, will play the role of the category of Grothendieck's motives. Guided by this, the next theorem already describes the analogue of motivic cohomology in an explicit manner and is the starting point of our research (see Theorem \ref{thm:cohomology-in-MR}). Let $\underline{M}$ be an $A$-motive over $R$.
\begin{MainTheorem}\label{mthm:extension}
The cohomology of the complex $\left[M\xrightarrow{\id-\tau_M} M[\fj^{-1}]\right]$ of $A$-modules, sitting in degrees zero and one, computes the extension modules $\Ext^i_{\cM_R}(\mathbbm{1},\underline{M})$ for all $i$.
\end{MainTheorem}
We immediately deduce that $\Ext^i_{\cM_R}(\mathbbm{1},\underline{M})$ vanishes for $i>1$. For $i=1$, the $A$-module of degree one extensions admit the following explicit description. There is a natural surjective morphism
\begin{equation}\label{eq:iota}
\iota:M[\fj^{-1}]\longrightarrow\Ext^1_{\cM_F}(\mathbbm{1},\underline{M})
\end{equation}
which maps $m\in M[\fj^{-1}]$ to the class of the extension of $\mathbbm{1}$ by $\underline{M}$ whose middle term is the $A$-motive given by $\left[M\oplus (A\otimes R),\left(\begin{smallmatrix}\tau_M & m \\ 0 & 1\end{smallmatrix}\right)\right]$ (Subsection \ref{sec:extensions-in-MF}). The kernel of $\iota$ being $(\id-\tau_M)(M)$, we recover the isomorphism provided by Theorem \ref{mthm:extension}.

\begin{Remark*}
Extension groups in the full subcategory of $\cM_R$ consisting of \emph{effective $A$-motives} (see Definition \ref{def:Amotives-effective}) were already determined in the existing literature (see \emph{e.g.} \cite{taelman-woods}, \cite{taelman-t-motif}, \cite{papanikolas-ramachandran}). The novelty of Theorem \ref{mthm:extension} is to consider the whole category $\cM_R$. 
\end{Remark*}

To pursue the analogy with number fields, we now present the notion of \emph{weights} and \emph{mixedness} for Anderson $A$-motives over fields. In the case $A=\bF[t]$ or $\deg(\infty)=1$ and over a complete algebraically closed base field, the corresponding definitions were carried out respectively by Taelman \cite{taelman} and Hartl--Juschka \cite{hartl-juschka}. We completed this picture in the most general way (over any $A$-field and without any restriction on $\deg(\infty)$). \\

Let $R=F$ be a field. To an Anderson $A$-motive $\underline{M}$ over $F$, we attach an \emph{isocrystal} $\operatorname{I}_{\infty}(\underline{M})$ at $\infty$ (in the sense of \cite{mornev-isocrystal}). The naming \emph{isocrystal} is borrowed from $p$-adic Hodge theory, where the function field setting allows to apply the non-archimedean theory at the infinite point $\infty$ of $C$ as well. Following \cite[1.9]{anderson}, we call $\underline{M}$ \emph{pure} of weight $\mu$ if its associated isocrystal is pure of slope $-\mu$ (Definition \ref{def:pure-A-motives}). More generally, we call $\underline{M}$ \emph{mixed} if there exists rational numbers $\mu_1<\ldots <\mu_s$ together with a finite ascending filtration of $\underline{M}$ by saturated sub-$A$-motives:
\[
0=W_{\mu_0}\underline{M}\subsetneq W_{\mu_1}\underline{M}\subsetneq \cdots \subsetneq W_{\mu_s}\underline{M}=\underline{M}
\]
for which the successive quotients $W_{\mu_i}\underline{M}/W_{\mu_{i-1}}\underline{M}$ are pure of weight $\mu_i$ (Definition \ref{def:mixed-A-motives}). We show in Proposition \ref{def:weight-filtration} that such a filtration--when it exists--is unique, as well as the numbers $\mu_i$ which are then called the \emph{weights} of $\underline{M}$. As it was observed by Hartl--Juschka (\cite[Ex. 2.3.13]{hartl-juschka}), there exist non mixed $A$-motives. Nonetheless, it is always possible to define the weights of a (not necessarily mixed) $A$-motive via the Dieudonn\'e-Manin decomposition of isocrystals (see Definition \ref{def:weights-A-motives}). We denote by $\cM\cM_F$ the full subcategory of $\cM_F$ whose objects are mixed Anderson $A$-motives over $F$. The main results of Section \ref{sec:mixed A motives} are gathered in the next theorem.

\begin{MainTheorem}\label{mthm:Extensions-mixed}
Let $\underline{M}$ be an object of $\cM\cM_F$. If all the weights of $\underline{M}$ are non-positive, then every extension of $\mathbbm{1}$ by $\underline{M}$ is mixed, that is: 
\[
\Ext^1_{\cM\cM_F}(\mathbbm{1},\underline{M})=\Ext^1_{\cM_F}(\mathbbm{1},\underline{M}).\]
If all the weights of $\underline{M}$ are positive, then an extension of $\mathbbm{1}$ by $\underline{M}$ is mixed if and only if its class is torsion, that is: 
\[
\Ext^1_{\cM\cM_F}(\mathbbm{1},\underline{M})=\Ext^1_{\cM_F}(\mathbbm{1},\underline{M})^{\operatorname{tors}}.\]
Furthermore, for $i>1$, $\Ext^i_{\cM\cM_F}(\mathbbm{1},\underline{M})$ is a torsion module for all $\underline{M}$.
\end{MainTheorem}
\begin{Remark*}
For our analogy to be complete, one would rather seek for a $K$-linear category~: where $\cM\cM_F$ is $A$-linear, the classical category $\cM\cM_\bQ$ is $\bQ$-linear. To obtain a $K$-linear category out of $\cM\cM_F$, one introduces $\cM\cM^{\text{iso}}_F$--the category of \emph{mixed $A$-motives up to isogenies}--whose objects are the ones of $\cM\cM_F$ and whose Hom-spaces are given by ${\Hom_{\cM\cM_F}(-,-)\otimes_A K}$ (\emph{e.g.} \cite{hartl-isogeny}, \cite{hartl-juschka}). Theorem \ref{mthm:Extensions-mixed} implies that $\Ext^i_{\cM\cM^{\text{iso}}_F}(\mathbbm{1},\underline{M})=0$ for $i>1$ and $\Ext^1_{\cM\cM^{\text{iso}}_F}(\mathbbm{1},\underline{M})=0$ if the weights of $\underline{M}$ are positive. This reveals that the analogue of the number fields conjecture \ref{item:conjectureC1} and \ref{item:conjectureC2} are true for function fields. Note, however, that contrary to what is expected for number fields, the full subcategory of pure $A$-motives is not semi-simple. Hence, we cannot expect any $1$-fold Yoneda extension of two pure $A$-motives to split, even if they have the same weight.
\end{Remark*}

\subsubsection*{The good $\ell$-adic part}
To present the theory of good and integral extensions, we now assume that $F$ is a finite field extension of $K$ (namely, a \emph{global function field}). Let $\fp$ be a finite place of $F$ (\emph{i.e.} not above $\infty$), $F_{\fp}$ the associated local function field, $F_{\fp}^s$ a separable closure of $F_\fp$ and $G_{\fp}=\Gal(F_{\fp}^s|F_\fp)$ the absolute Galois group of $F_\fp$ equipped with the profinite topology. Given a maximal ideal $\ell$ in $A$ which does not lie under $\fp$, there is an \emph{$\ell$-adic realization functor} from $\cM_{F_\fp}$ to the category of continuous $\cO_{\ell}$-linear representations of $G_\fp$. Given an object $\underline{M}_{\fp}=(M_{\fp},\tau_M)$ of $\cM_{F_\fp}$, it is defined as the $\cO_{\ell}$-module
\begin{equation}
\operatorname{T}_{\ell}\underline{M}_{\fp}:=\varprojlim_n ~\{m\in (M_{\fp}\otimes_{F_\fp} F_{\fp}^{s})/\ell^n (M_{\fp}\otimes_{F_\fp} F_{\fp}^{s})~|~m=\tau_M(\tau^*m)\}  \nonumber
\end{equation}
where $G_\fp$ acts compatibly on the right of the tensor $M_{\fp}\otimes_{F_{\fp}} F_{\fp}^{s}$ (Definition \ref{def:m-adic realization functor}).  \\

We prove in Corollary \ref{adic-realization-exact} that $\operatorname{T}_{\ell}$ is exact. This paves the way for introducing \emph{extensions with good reduction}, as Scholl did in the number fields setting. Let $I_{\fp}\subset G_{\fp}$ be the inertia subgroup. We consider the \emph{$\ell$-adic realization map} restricted to $I_{\fp}$:
\begin{equation}\label{intro:good-reduction-sholl}
r_{\underline{M},\ell,\fp}:\Ext^1_{\cM_{F_\fp}}(\mathbbm{1}_{\fp},\underline{M}_{\fp})\longrightarrow \operatorname{H}^1(I_{\fp},\operatorname{T}_{\ell}\underline{M}_{\fp})
\end{equation} 
(we refer to Subsection \ref{sec:extensions-having-good-reduction}). Mimicking Scholl's approach, we say that an extension $[\underline{E}_{\fp}]$ of $\mathbbm{1}_{\fp}$ by $\underline{M}_{\fp}$ has \emph{good reduction} if $[\underline{E}_{\fp}]$ lies in the kernel of \eqref{intro:good-reduction-sholl}. As in the number field setting, we expect this definition to be independent of $\ell$, although this is presumably out of our reach. We let $\Ext^1_{\text{good}}(\mathbbm{1}_{\fp},\underline{M}_{\fp})_{\ell}$ denote the kernel of $r_{\underline{M},\ell,\fp}$ (Definition \ref{def:ext-good-red-with-resp-to-ell}). 

\subsubsection*{The integral part}
Gardeyn in \cite{gardeyn2} has introduced a notion of \emph{maximal models} for $\tau$-sheaves. Inspired by Gardeyn's work, we developed the notion of \emph{maximal integral models of $A$-motives} (Section \ref{chapter:integral models}). They form the function field analogue of N\'eron models of abelian varieties, or more generally, of regular models of varieties. \\

Let $\underline{M}$ be an $A$-motive over $F$, and denote by $\underline{M}_\fp$ the $A$-motive of $F_\fp$ obtained from $\underline{M}$ by base-change along $F\subset F_\fp$. Let $\cO_{\fp}$ be the valuation ring of $F_{\fp}$ and let also $\cO_F$ denote the integral closure of $A$ in $F$.

\begin{Definition*}[Definition \ref{def:integral-models-of-motives}]
Let $L$ be a finitely generated sub-$A\otimes \cO_\fp$-module of $M_\fp$ (resp. $A\otimes \cO_F$-module of $M$).
\begin{enumerate}[label=$(\arabic*)$]
\item We say that $L$ is an $\cO_\fp$-\emph{model for $\underline{M}_{\fp}$} if it generates $M_\fp$ over $F_\fp$ and $\tau_{M_\fp}(\tau^*L)\subset L[\fj^{-1}]$. 
\item We say that $L$ is an $\cO_F$-\emph{model for $\underline{M}$} if it generates $M$ over $F$ and $\tau_M(\tau^*L)\subset L[\fj^{-1}]$. 
\end{enumerate}
We say that $L$ is \emph{maximal} if $L$ is not strictly contained in any other model. 
\end{Definition*}

As opposed to \cite[Def 2.1 \& 2.3]{gardeyn2}, we do not ask for an $\cO_\fp$-model (resp. $\cO_F$-model) to be locally free. We show that this is automatic for maximal ones (Theorem \ref{thm:integral-models-are-locally-free-local}). Compared to Gardeyn, our exposition is therefore simplified and avoids the use of a technical lemma due to L. Lafforgue \cite[\S2.2]{gardeyn2}. Our next result should be compared with \cite[Prop 2.13]{gardeyn2} (see Proposition \ref{prop-existence-of-maximal-integral-model-global} and Theorems \ref{thm:integral-models-are-locally-free} and \ref{thm:from-global-to-local} in the text).
\begin{MainTheorem}\label{mthm:Maximal-Model}
A maximal $\cO_\fp$-model $M_{\cO_{\fp}}$ for $\underline{M}_\fp$ (resp. $\cO_F$-model $M_{\cO_F}$ for $\underline{M}$) exists and is unique. It is locally free over $A\otimes \cO_\fp$ (resp. $A\otimes\cO_F$). In addition, the maximal local and global models are related by canonical isomorphisms:
\[
M_{\cO_F}\cong \bigcap_{\fp}{(M\cap M_{\cO_\fp})}, \quad \text{and} \quad M_{\cO_{\fp}}\cong M_{\cO_F}\otimes_{\cO_F}\cO_\fp,
\]
where the intersection is taken over finite places $\fp$ of $F$.
\end{MainTheorem}

Along the way, we also prove a good reduction criterion for $A$-motives, in the style of N\'eron-Ogg-Shafarevi\v{c} (Proposition \ref{prop-good-reduction-local-case}).\\

Then, we call \emph{$\fp$-integral} any extension of $\mathbbm{1}$ by $\underline{M}$ which arises as an element of $\iota(M_{\cO_{\fp}}[\fj^{-1}])$ \eqref{eq:iota}. We let $\Ext^1_{\cO_\fp}(\mathbbm{1}_\fp,\underline{M}_\fp)$ be the module of $\fp$-integral extensions (Definition \ref{def:integral-part-local}). Our main result (repeated from Theorem \ref{thm:main1}) is the next:
\begin{MainTheorem}\label{mthm:Integral-inside-Good}
Let $\ell$ be a maximal ideal of $A$ which does not lie under $\fp$. Then $\Ext^1_{\cO_{\fp}}(\mathbbm{1}_{\fp},\underline{M}_{\fp})$ is a sub-$A$-module of $\Ext^1_{\operatorname{good}}(\mathbbm{1}_{\fp},\underline{M}_{\fp})_{\ell}$. 
\end{MainTheorem} 

Surprisingly enough, we cannot claim equality in general. In Subsection \ref{subsec:counter-example}, in the simplest case of the neutral $A$-motive, we construct for some $\ell$ and $\fp$ an explicit extension in $\Ext^1_{\text{good}}(\mathbbm{1}_{\fp},\mathbbm{1}_{\fp})_{\ell}$ which does not belong to $\Ext^1_{\cO_\fp}(\mathbbm{1}_{\fp},\mathbbm{1}_{\fp})$. \\

In Subsection \ref{subsec:integral-part}, we define the \emph{global} version of the above. Namely, the $A$-motive $\underline{M}$ defines an $A$-motive $\underline{M}_{F_\fp}$ over $F_{\fp}$ by extending the base field. We let $\Ext^1_{\cO_F}(\mathbbm{1},\underline{M})$ be the module of \emph{integral extensions}; \emph{i.e.} of extensions that are $\fp$-integral for all $\fp$ after base-change along $F\subset F_\fp$ (Definition \ref{def:integral-part-local}). Our second main result (repeated from Theorem \ref{thm:main2}) is the following:
\begin{MainTheorem}\label{mthm:Global-Integral}
The $A$-module $\Ext^1_{\cO_F}(\mathbbm{1},\underline{M})$ equals the image of $M_{\cO_F}[\fj^{-1}]$ through $\iota$. In addition, $\iota$ induces a natural isomorphism of $A$-modules:
\[
\frac{M_{\cO_F}[\fj^{-1}]}{(\id-\tau_M)(M_{\cO_F})}\stackrel{\sim}{\longrightarrow} \Ext^1_{\cO_F}(\mathbbm{1},\underline{M}).
\]
\end{MainTheorem}

\subsubsection*{Regulated extensions of $A$-motives}
We are facing two main issues to pursue our analogy: the counterpart of Conjecture \ref{item:conjectureC4} does not hold true and, more seriously, neither is the counterpart of \ref{item:conjectureC5}: the $A$-module $ \Ext^1_{\cO_F}(\mathbbm{1},\underline{M})$ is typically not finitely generated. Those facts suggest that the category $\cM_F$--and also $\cM\cM_F$--is \emph{too huge} to held a convincing motivic cohomology theory. We end this text by presenting a conjectural picture aiming to answer the analogue of \ref{item:conjectureC4} and \ref{item:conjectureC5}. \\

Pink, in the context of function fields Hodge structures \cite[\S 6]{pink}, is facing a similar issue. There, he introduced the notion of \emph{Hodge additivity}, whose counter-part for $A$-motives is as follow.
\begin{Definition*}[\emph{cf}. \ref{subsec:Hodge-polygon} for details]
Let $0\to \underline{M}\to \underline{E} \to \underline{N} \to 0$ be an exact sequence of $A$-motives over $F$. We say that $[\underline{E}]$ is \emph{regulated} if the Hodge polygon of the Hodge-Pink structure attached to $\underline{M}\oplus \underline{N}$ matches the one of $\underline{E}$ (see \cite[\S 6]{pink}). We denote by $\Ext^{1,\text{reg}}$ the submodule of regulated extensions.
\end{Definition*}

Pink's intuition for introducing \emph{Hodge additivity} is driven by the inexact feature of the operation assigning to a Hodge-Pink structure its "classical" Hodge structure (\emph{i.e.} the data of its Hodge filtration). Although Pink is concerned with computation of Hodge groups, for our purpose this prevents the well-definedness of a \emph{regulator map} from the module of extensions of $A$-motives to that of extensions of "classical" Hodge structures (see Remark \ref{rmk:regulation}). \emph{Regulated extensions} are designed to resolve this point, hence the naming. The notion of \emph{regulation} has applications in sequels to this text \cite{gazda2}, \cite{gazda-maurischat-ext}. \\

In Corollary \ref{cor:regulated-extensions-1-M}, we prove that $\iota$ induces an isomorphism of $A$-modules:
\[
\frac{M+\tau_{M}(\tau^*M)}{(\id-\tau_M)(M)}\stackrel{\sim}{\longrightarrow} \Ext^{1,\text{reg}}_{\cM_{F}}(\mathbbm{1},\underline{M}).
\]
While $\Ext^{1,\text{reg}}_{\cO_F}(\mathbbm{1},\underline{M})$ is still not finitely generated over $A$ in general, a version of Conjecture \ref{item:conjectureC5} involving the infinite places holds (we refer the reader to \cite[Thm. 4.1]{gazda2}). Concerning the analogue of Conjecture \ref{item:conjectureC4}, we strongly suspect regulated extensions to also correct the integral$\neq$good extensions phenomenon; namely we expect the following to hold for $A$-motives $\underline{M}_\fp$ over $F_\fp$ (see Conjecture \ref{conjecture-reg}):
\begin{Conjecture*}
Let $\ell$ be a maximal ideal of $A$ not under $\fp$. Then, 
\[
\Ext^{1,\text{reg}}_{\cO_{\fp}}(\mathbbm{1}_{\fp},\underline{M}_{\fp})=\Ext^{1,\text{reg}}_{\text{good}}(\mathbbm{1}_{\fp},\underline{M}_{\fp})_{\ell}.
\] 
In particular, the module $\Ext^{1,\text{reg}}_{\text{good}}(\mathbbm{1}_{\fp},\underline{M}_{\fp})_{\ell}$ does not depend on $\ell$.
\end{Conjecture*}

We conclude this text by solving particular instances of the above conjecture (Subsection \ref{subsec:regulated-good-reduction}). The general case, however, remains open.
 
\paragraph{Acknowledgments:} The main results presented in this text were obtained during the author's PhD thesis. I would like to express my profound gratitude to Professors Gebhard Böckle and Urs Hartl, as well as my advisor, Professor Federico Pellarin. I am greatly indebted to the anonymous referee for their advice and suggestions, which significantly contributed to improving the quality of this document. Among the suggestions made by the referee, I am especially grateful for their assistance in completing the original proof of Theorem \ref{thm:true-for-carlitz} and resolving the above conjecture for positive Carlitz twists $\underline{A}(n)$. Our initial argument only applied to cases where $n$ was a power of the characteristic of $\bF$. Parts of this text were revised during my time at the Max Planck Institute for Mathematics in Bonn, and I am thankful for their hospitality and financial support.

\subsection{Plan of the paper}
The paper is organized as follows.\\

In the beginning of Section \ref{chapter:definitions}, Subsection \ref{section:definitions}, we review the usual set up (notations, definitions, basic properties) of $A$-motives over an arbitrary commutative and unital $A$-algebra. We follow \cite{hartl-juschka} and \cite{hartl-isogeny} as a guideline, though the former reference is concerned with the particular choice of a closed point $\infty$ of degree one and over a complete algebraically closed field. Most of the results on $A$-motives extend without changes to our more general setting. In Subsection \ref{sec:extensions-in-MF}, $A$-Motivic cohomology is introduced. We describe the extension modules in categories of $A$-motives and obtain Theorem \ref{mthm:extension} as Theorem \ref{thm:cohomology-in-MR} in the text. In Subsection \ref{sec:extensions-having-good-reduction}, we recall the definition and main properties of the $\ell$-adic realization functor for $A$-motives, and introduce \emph{extensions having good reduction with respect to $\ell$} in Definition \ref{def:ext-good-red-with-resp-to-ell}. \\

Section \ref{sec:mixed A motives} is concerned with mixed $A$-motives. In the beginning of Subsection \ref{subsec:mixed-A-motives} we recall, and add some new material, to the theory of function fields isocrystals in the steps of \cite{mornev-isocrystal}. The main ingredient, used later on in Subsection \ref{subsec:mixed-A-motives} to define the category of mixed $A$-motives over $A$-fields, is the existence and uniqueness of the slope filtration (extending \cite[Prop. 1.5.10]{hartl} to general coefficient rings $A$). We focus on extension modules in the category $\cM\cM_F$ in Subsection \ref{subsec:extension-modules-of-mixed} where we deduce Theorem \ref{mthm:Extensions-mixed} from Propositions \ref{prop:separated-weights-ext}, \ref{prop-torsion-in-Ext-positive-weights} and Theorem \ref{thm:main0}.\\

In Section \ref{chapter:integral models}, we develop the notion of \emph{maximal integral models} of $A$-motives over local and global function fields. It splits into four subsections. In Subsection \ref{subsec:integral-models-of-frobenius-modules}, we present integral models of \emph{Frobenius spaces} over local function fields. The theory is much easier than the one for $A$-motives, introduced over a local function field in Subsection \ref{sec:integral-models-A-motives-local} and over a global function field in Subsection \ref{sec:integral-models-A-motives-global}. Although our definition of integral model is inspired by Gardeyn's work in the context of $\tau$-sheaves \cite{gardeyn2}, our presentation is simpler as we removed the \emph{locally free} assumption. That maximal integral models are locally free is automatic, as we show in Theorems \ref{thm:integral-models-are-locally-free-local} and \ref{thm:integral-models-are-locally-free}. The chief aim of this section, however, is Subsection \ref{subsec:integral-part} where we use the results of the earlier subsections to prove Theorems \ref{mthm:Integral-inside-Good} and \ref{mthm:Global-Integral} (respectively Theorems \ref{thm:main1} and \ref{thm:main2} in the text). \\

In our last Section \ref{sec:regulated-extensions}, we introduce the notion of \emph{regulated extensions of $A$-motives} with an eye toward understanding the lack of equality in Theorem \ref{mthm:Integral-inside-Good}, highlighted in Subsection \ref{subsec:counter-example}. We recall the definition of Hodge polygons, as introduced in \cite{pink}, in Subsection \ref{subsec:Hodge-polygon}. Those are used to define regulated extensions in Definition \ref{def:regulated-extensions}. We conclude this text by Subsection \ref{subsec:regulated-good-reduction}, where we present a general hope that $\Ext^{1,\text{reg}}_{\cO_\fp}(\mathbbm{1}_\fp,\underline{M}_\fp)$ and $\Ext^{1,\text{reg}}_{\text{good}}(\mathbbm{1}_\fp,\underline{M}_\fp)_{\ell}$ match. We then prove some particular instances of this expectation, namely when $\underline{M}$ arises from a Frobenius space with good reduction (Theorem \ref{thm:true-for-pure0}), or when $\underline{M}$ is a positive Carlitz's twist $\underline{A}(n)$ for $n\geq 0$ (Theorem \ref{thm:true-for-carlitz}). \\

All along this text, we make constant use of extension modules in exact but non necessarily abelian categories. Although everything works as expected, we decided to add an appendix on exact categories as Section \ref{sec:exact-categories} because we did not find all of the relevant references in the literature. 

\section{Anderson $A$-motives and their extension modules}\label{chapter:definitions}
Let $\bF$ be a finite field with $q$ elements. By convention, unlabeled tensor products and fiber products throughout this text are over $\bF$; all algebras are associative, commutative and with unit. Let $(C,\cO_C)$ be a geometrically irreducible smooth projective curve over $\bF$, and fix a closed point $\infty$ on $C$. Let $A:=\cO_C(C\setminus\{\infty\})$ be the ring of regular functions on $C\setminus \{\infty\}$ and let $K$ be the function field of $C$. 
 
\subsection{Definition of $A$-motives}\label{section:definitions}
This subsection is devoted to define and recall the main properties of Anderson $A$-motives. We begin with a paragraph of notations.

\subsubsection*{Preliminaries on $\fj$ and $\tau$}
Let $R$ be an $\bF$-algebra and let $\kappa:A\to R$ be an $\bF$-algebra morphism. $R$ will be referred to as \emph{the base algebra} and $\kappa$ as the \emph{characteristic morphism}. The kernel of $\kappa$ is called \emph{the characteristic of $(R,\kappa)$}. We consider the ideal $\fj=\fj_{\kappa}$ of $A\otimes R$ generated by the set $\{a\otimes 1-1\otimes \kappa(a)|a\in A\}$; $\fj$ is equivalently defined as the kernel of $A\otimes R\to R$, $a\otimes f\mapsto \kappa(a)f$. The following observation appears in \cite{hartl-isogeny}.
\begin{Lemma}\label{lem:j-Cartier}
The ideal $\fj$ is an invertible $A\otimes R$-module. 
\end{Lemma}

According to Lemma \ref{lem:j-Cartier}, the $A\otimes R$-module $\fj^n$ makes sense for any integer $n$ and we have canonical inclusions $\fj^{-n}\hookrightarrow \fj^{-(n+1)}$. Given an $A\otimes R$-module $M$, we denote by $M[\fj^{-1}]$ the colimit
\[
M[\fj^{-1}]:=\varinjlim~ M\otimes_{A\otimes R}\fj^{-n}
\]
taken over non-negative integers $n$ in the category of $A\otimes R$-modules.
\begin{Remark}
For consistency with \cite{hartl-isogeny}, observe that the map of schemes: 
\[
D(\fj)=(\Spec A\otimes R)\setminus V(\fj) \longrightarrow \Spec(A\otimes R)[\fj^{-1}], \quad \fp \longmapsto \fp[\fj^{-1}],
\]
composed with $\Spec(A\otimes R)[\fj^{-1}]\to \Spec A\otimes R$, $\fq\mapsto \fq\cap (A\otimes R)$, coincides with the canonical inclusion $D(\fj)\hookrightarrow \Spec A\otimes R$ of an open subscheme. Hence, $D(\fj)$ and $\Spec(A\otimes R)[\fj^{-1}]$ are canonically isomorphic. As tensor products commute with direct limits, we further have:
\[
M[\fj^{-1}]\cong M\otimes_{A\otimes R}(A\otimes R)[\fj^{-1}]
\]
which implies that $M[\fj^{-1}]$ is obtained by pulling back $M$ along the open immersion
\[
(\Spec A\otimes R)\setminus V(\fj) \hookrightarrow \Spec A\otimes R.
\]
\end{Remark}

Given another module $N$, we shall use freely the identifications\footnote{For any integer $n$, the module $\fj^{n}$ is finite projective and hence $\Hom(N,M)\otimes \fj^n\cong \Hom(N,M \otimes \fj^n)$; the first isomorphism follows as direct limits commute with $\Hom(N,-)$. The second one follows from $\Hom(N\otimes \fj^{-n},M[\fj^{-1}])\cong \Hom(N,M[\fj^{-1}]\otimes \fj^{n})$ together with $M[\fj^{-1}]\otimes \fj^n\cong M[\fj^{-1}]$ (as tensor products commute with colimits).}
\[
\Hom_{A\otimes R}(N,M)[\fj^{-1}]\stackrel{\sim}{\rightarrow} \Hom_{A\otimes R}(N,M[\fj^{-1}])\stackrel{\sim}{\leftarrow} \Hom_{(A\otimes R)[\fj^{-1}]}(N[\fj^{-1}],M[\fj^{-1}]).
\]
We shall also use, without further notice, that the map $A\otimes R\to (A\otimes R)[\fj^{-1}]$ is flat\footnote{This follows from the flatness of the $A\otimes R$-module $\fj^{n}$ together with the fact that colimits preserve exact sequences.}.\\

We denote by $\tau:A\otimes R\to A\otimes R$ be the $A$-linear morphism given by $a\otimes r\mapsto a\otimes r^q$ on elementary tensors. We denote by $\tau^*M$ be the pull-back of $M$ by $\tau$ (\cite[A.II.\S 5]{bourbaki}); \emph{i.e.} $\tau^*M$ is the $A\otimes R$-module
\begin{equation}
\tau^*M:=(A\otimes R)\otimes_{\tau, A\otimes R}M \nonumber
\end{equation}
where the subscript $\tau$ signifies that the relation $(a\otimes_{\tau} bm)=(a\tau(b)\otimes_{\tau} m)$ holds for all elements $a,b\in A\otimes R$, $m\in M$, and where the module structure corresponds to $b\cdot (a\otimes_{\tau} m):=(ba\otimes_{\tau} m)$. We let 
\[
\mathbf{1}:\tau^*(A\otimes R)\longrightarrow A\otimes R, 
\]
be the $A\otimes R$-linear morphism which maps $(a\otimes r)\otimes_{\tau} (b\otimes s)\in \tau^*(A\otimes R)$ to $ab\otimes rs^q\in A\otimes R$.

\subsubsection*{Anderson $A$-motives}
The next definition takes its roots in the work of Anderson \cite{anderson}, though this version is borrowed from \cite[Def.~2.1]{hartl-isogeny}:
\begin{Definition}\label{def:Amotives}
An \emph{Anderson $A$-motive $\underline{M}$ (over $R$) } is a pair $(M,\tau_M)$ where $M$ is a projective $A\otimes R$-module of finite constant rank and where $\tau_M:(\tau^*M)[\fj^{-1}]\to M[\fj^{-1}]$ is an isomorphism of $(A\otimes R)[\fj^{-1}]$-modules. \\
In all the following, we shall simply write \emph{$A$-motive} instead of \emph{Anderson $A$-motive}. The \emph{rank of $\underline{M}$} is the (constant) rank of $M$ over $A\otimes R$. \\
A morphism $(M,\tau_M)\to (N,\tau_N)$ of $A$-motives (over $R$) is an $A\otimes R$-linear morphism ${f:M\to N}$ such that $f\circ \tau_M=\tau_N\circ \tau^*f$. We let $\cM_R$ be the $A$-linear category of $A$-motives over $R$.
\end{Definition}

\begin{Remark}
$A$-motives as in Definition \ref{def:Amotives} are called \emph{abelian $A$-motives} by several authors (see \emph{e.g.} \cite{brownawell-papanikolas}). The word \emph{abelian} refers to the assumption that the underlying ${A\otimes R}$\nobreakdash-module is finite projective. Dropping this assumption is not a good strategy in our work, as too many analogies with number fields motives would fail to hold. 
\end{Remark}

\begin{Definition}\label{def:Amotives-effective}
An $A$-motive $\underline{M}=(M,\tau_M)$ over $R$ is called \emph{effective} if $\tau_M(\tau^*M)\subset M$. We let $\cM_R^{\text{eff}}$ be the full subcategory of $\cM_R$ whose objects are effective $A$-motives.
\end{Definition}

Given an $A$-motive $\underline{M}$, we will write $M$ without an underline to refer to its underlying module and write $\tau_N$ to refer to the underlying morphism. We list below some important classical constructions:
\begin{enumerate}[label=$-$]
\item Let $\mathbbm{1}$ be the \emph{unit $A$-motive over $R$} defined as $(A\otimes R,\textbf{1})$.
\item The \emph{direct sum} of two $A$-motives $\underline{M}$ and $\underline{N}$, denoted $\underline{M}\oplus \underline{N}$, is representable in $\cM_R$ and corresponds to the $A$\nobreakdash-motive whose underlying $A\otimes R$-module is $M\oplus N$ and whose $\tau$-linear morphism is $\tau_M\oplus \tau_N$.
\item Their \emph{tensor product}, denoted $\underline{M}\otimes \underline{N}$, is defined to be $(M\otimes_{A\otimes R}N,\tau_M\otimes \tau_N)$. The tensor operation admits $\mathbbm{1}$ as a neutral object.
\item The \emph{internal hom} is the object $\underline{\eH}=\underline{\cHom}(\underline{M},\underline{N})$ of $\cM_R$ representing the functor $\underline{X}\mapsto \Hom_{\cM_R}(\underline{X}\otimes\underline{M},\underline{N})$. One verifies that its underlying module is $\eH=\Hom_{A\otimes R}(M,N)$ with $\tau_{\eH}$ given by
\[
\tau_{\eH}:(\tau^*\eH)[\fj^{-1}]=\Hom_{A\otimes R}((\tau^*M)[\fj^{-1}],(\tau^*N)[\fj^{-1}])\stackrel{\sim}{\to} \Hom_{A\otimes R}(M[\fj^{-1}],N[\fj^{-1}])=\eH[\fj^{-1}], 
\]
mapping $h\mapsto \tau_N\circ h \circ \tau_M^{-1}$.
\item The \emph{dual} of $\underline{M}$ is defined to be the $A$-motive $\underline{\cHom}(\underline{M},\mathbbm{1})$.
\item Given $S$ an $R$-algebra, there is a \emph{base-change} functor $\cM_R\to \cM_S$ mapping the $A$\nobreakdash-motive ${\underline{M}=(M,\tau_M)}$ over $R$ to $\underline{M}_S:=(M\otimes_R S,\tau_M\otimes_R \id_S)$.
\item In the situation where $S$ is a finitely presented flat $R$-module, there is also a \emph{restriction functor} $\Res_{S/R}:\cM_S\to \cM_R$ mapping an $A$-motive $\underline{M}$ over $S$ to $\underline{M}$ seen as an $A$-motive over $R$. Given two $A$-motives $\underline{M}$ and $\underline{N}$ over $R$ and $S$ respectively, we have
\begin{equation}
\Hom_{\cM_R}(\underline{M},\Res_{S/R}\underline{N}) =\Hom_{\cM_S}(\underline{M}_S,\underline{N}). \nonumber
\end{equation}
In other words, the base-change functor is left-adjoint to the restriction functor.
\end{enumerate}

\begin{Example}[Carlitz's motive]\label{ex:carlitz-motive}
Let $C=\bP_{\bF}^1$ be the projective line over $\bF$ and let $\infty$ be the closed point of coordinates $[0:1]$. If $t$ is any element in $\cO(\bP^1_{\bF}\setminus\{\infty\})$ whose order of vanishing at $\infty$ is $1$, we have an identification $A=\bF[t]$. For an $\bF$-algebra $R$, the tensor product $A\otimes R$ is identified with $R[t]$. The morphism $\tau$ acts on $p(t)\in R[t]$ by raising its coefficients to the $q$th-power. It is rather common to denote by $p(t)^{(1)}$ the polynomial $\tau(p(t))$. Let $\kappa:A\to R$ be an $\bF$-algebra morphism and let $\theta=\kappa(t)$. The ideal $\fj\subset R[t]$ is principal, generated by $(t-\theta)$. \\
The Carlitz $\bF[t]$-motive $\underline{C}$ over $R$ is defined by the couple $(R[t],\tau_{C})$ where $\tau_C$ maps $\tau^*p(t)$ to $(t-\theta)p(t)^{(1)}$. Its $n$th tensor power $\underline{C}^n:=\underline{C}^{\otimes n}$ is isomorphic to the $\bF[t]$-motive whose underlying module is $R[t]$ and where $\tau_{C^{n}}$ maps $\tau^*p(t)$ to $(t-\theta)^n p(t)^{(1)}$. \\
To match with number fields analogy, we let $\underline{A}(n):=\underline{C}^{-n}=(\underline{C}^{n})^{\vee}$. For $A=\bF[t]$, this notation stresses that $\underline{A}(1)$ plays the role of the number fields Tate motive $\bZ(1)$ and, more generally, that $\underline{A}(n)$ plays the role of $\bZ(n)$.
\end{Example}

In particular, the category $\cM_R$ is additive and $A$-linear and it is a mere verification that that data of $(\cM_R,\otimes,\mathbbm{1})$ forms a symmetric monoidal additive category which is rigid (\emph{e.g.} \cite[\S 2.3]{hartl-juschka}). \\

In general, morphisms in $\cM_R$ might not admit kernel nor cokernel because the underlying module of the naturally defined kernel or cokernel might not be projective over $A\otimes R$. Nonetheless, there are some favorable situations where they are: 
\begin{Lemma}\label{lem:kernel-cokernel}
Let $f:\underline{M}\to \underline{N}$ be a morphism in $\cM_R$.
\begin{enumerate}
\item If its underlying map of modules is surjective, then $f$ admits a kernel in $\cM_R$ and its underlying module is $\ker f$.
\item If its underlying map of modules is a split injection, then $f$ admits a cokernel in $\cM_R$ and its underlying module is $\coker f$.
\end{enumerate}
\end{Lemma}
\begin{proof}
We only prove the first part, as the argument for the second one is similar. By assumption, $f:M\to N$ is surjective. As $N$ is projective, $f$ admits a splitting, amounting to $M\cong \ker f\oplus N$. We deduce that $\ker f$ is finite projective (as $M$ is) and that the sequence $0\to (\tau^* \ker f)[\fj^{-1}]\to (\tau^* M)[\fj^{-1}]\to (\tau^* N)[\fj^{-1}]\to 0$ is exact. By the universal property of $\ker f$, there exists a unique dashed arrow making the following diagram commute:
\begin{equation}
\begin{tikzcd}
0 \arrow[r] & {(\tau^* \ker f)[\fj^{-1}]}\arrow[r]\arrow[d,dashed,"\tau_{\ker f}","\exists!"'] & {(\tau^* M)[\fj^{-1}]} \arrow[r]\arrow[d,"\tau_M","\wr"'] & {(\tau^* N)[\fj^{-1}]} \arrow[d,"\tau_N","\wr"']\arrow[r] & 0 \\
0 \arrow[r] & {(\ker f)[\fj^{-1}]} \arrow[r] & {M[\fj^{-1}]} \arrow[r] & {N[\fj^{-1}]} \arrow[r] & 0
\end{tikzcd}
\nonumber
\end{equation}
and this dashed arrow is an isomorphism by the Snake Lemma. That is, $\underline{\ker}(f):=(\ker f,\tau_{\ker f})$ defines an $A$-motive over $R$ sitting in an sequence $\underline{S}:0\to \underline{\ker}(f)\to \underline{M}\to \underline{N}\to 0$.

It remains to check its universal property. Let $g:\underline{M}'\to \underline{M}$ be a morphism in $\cM_R$ whose composition with $f$ is zero. From the universal property of $\ker f$ in $\mathbf{Mod}_{A\otimes R}$, there exists a unique $h:M'\to \ker f$ through which $g$ factors. Since $\underline{S}$ splits in $\mathbf{Mod}_{A\otimes R}$, $(\ker f)[\fj^{-1}]$ (resp. $(\tau^*\ker f)[\fj^{-1}])$) is also a kernel of $f[\fj^{-1}]$ (resp. $(\tau^* f)[\fj^{-1}]$) and $g$ factors uniquely through it. By uniqueness, we obtain $h[\fj^{-1}]=\tau_{\ker f}\circ (\tau^* h)\circ \tau_{M'}^{-1}[\fj^{-1}]$ from which we deduce that $h$ extends (uniquely) to a morphism $h:\underline{M}'\to \underline{\ker}(f)$.
\end{proof}

We end this paragraph by showing a useful corollary of Lemma \ref{lem:kernel-cokernel}. Recall that a \emph{retraction} in an additive category is a morphism $r:X\to Y$ for which there exists a section $s:Y\to X$ such that $r\circ s=\id_Y$; an additive category is called \emph{weakly idempotent complete} if every retraction admits a kernel \cite[\S 7]{buhler}. The following consequence will be important for the result in the appendix to apply, especially Proposition~\ref{prop:pull-push-are-exact}.
\begin{Corollary}\label{cor:wic}
The additive category $\cM_R$ is weakly idempotent complete. 
\end{Corollary}
\begin{proof}
A retraction in $\cM_R$ is necessarily surjective at the level of modules, hence admits a kernel by \ref{lem:kernel-cokernel}. 
\end{proof}

\subsubsection*{Isogenies of $A$-motives}
The category $\cM_R$ of $A$-motives over $R$ is generally \emph{not} abelian, even if $R=F$ is a field. To remedy to the non-abelian feature, we shall introduce next the category $\cM^{\text{iso}}_R$ of $A$-motives \emph{up to isogenies over $R$} (see Definition \ref{def:tilde-motives}), which is abelian whenever $R=F$ is a field. Following \cite[Def.~5.5, Prop.~5.8, Thm.~5.12, Cor.~5.15]{hartl-isogeny}, we recall:
\begin{Proposition-Definition}\label{def:isogeny}
A morphism $f:\underline{M}\to \underline{N}$ in $\cM_R$ is an \emph{isogeny} if one of the following equivalent conditions is satisfied.
\begin{enumerate}[label=$(\alph*)$]
\item $f$ is injective and $\coker(f:M\to N)$ is a finite projective $R$-module,
\item\label{item:isogeny-rank} $M$ and $N$ have the same rank and $\coker f$ is finite over $R$,
\item\label{item:isogeny-a} there exists $0\neq a\in A$ such that $f$ induces an isomorphism of $(A\otimes R)[a^{-1}]$-modules $M[a^{-1}]\stackrel{\sim}{\to} N[a^{-1}]$,
\item there exists $0\neq a\in A$ and $g:\underline{N}\to \underline{M}$ in $\cM_R$ such that $f\circ g=a\id_N$ and $g\circ f=a\id_M$.
\end{enumerate}
If an isogeny between $\underline{M}$ and $\underline{N}$ exists, $\underline{M}$ and $\underline{N}$ are said to be \emph{isogenous}.
\end{Proposition-Definition}

We then introduce the notion of \emph{saturation}. 
\begin{Definition}\label{def:saturation}
Let $f:\underline{N}\to \underline{M}$ be a morphism of $A$-motives.
\begin{enumerate}
\item We call $f$ \emph{saturated}, or say that \emph{$\underline{N}$ is saturated in $\underline{M}$}, if $f$ admits a cokernel in $\cM_R$.
\item Let $\mathbf{Sat}(f)$ denote the category of diagrams in $\cM_R$ of the form
\begin{equation}
\begin{tikzcd}
 & \underline{K}\arrow[d,"g"] \\
\underline{N}\arrow[ur]\arrow[r,"f"] & \underline{M}
\end{tikzcd}\nonumber
\end{equation}
where $g$ is saturated, with morphisms in $\mathbf{Sat}(f)$ being morphisms of diagrams which are the identity on both $\underline{N}$ and $\underline{M}$. When it exists, we write an initial object in $\mathbf{Sat}(f)$--unique up to unique isomorphism--in the form 
\begin{equation}
\begin{tikzcd}
 & \underline{N}^{\operatorname{sat}}\arrow[d,"f^{\operatorname{sat}}"] \\
\underline{N}\arrow[ur,"i"]\arrow[r,"f"] & \underline{M}
\end{tikzcd}\nonumber
\end{equation}
and call $f^{\operatorname{sat}}:\underline{N}^{\operatorname{sat}}\to \underline{M}$ the \emph{saturation of $f$}.
\end{enumerate}
\end{Definition}
\begin{Lemma}\label{lem:satiso}
Assume that $R=F$ is a field and let $f:\underline{N}\to \underline{M}$ be a morphism in $\cM_F$. Then $f$ admits a kernel  in $\cM_F$, and $\mathbf{Sat}(f)$ admits an initial object. If further $f$ is a monomorphism, the canonical map $i:\underline{N}\to \underline{N}^{\operatorname{sat}}$ is an isogeny. 
\end{Lemma}
\begin{proof}
The choice of $R=F$ a field makes $A\otimes F$ into a Dedekind domain. Therefore, any torsion-free subquotient of a finite projective $A\otimes F$-module is itself finite projective; hence so is the module $\ker(f)$. That $\tau_N$ induces an isomorphism ${\tau_{\ker f}:\tau^*\ker(f)[\fj^{-1}]\stackrel{\sim}{\to} \ker(f)[\fj^{-1}]}$ and that $\underline{\ker}(f):=(\ker(f),\tau_{\ker f})$ represents a kernel of $f$ in $\cM_F$ is a mere verification similar to the proof of Lemma \ref{lem:kernel-cokernel}. 

For what remains, we assume--up to replacing $\underline{N}$ by the quotient $\underline{N}/\underline{\ker}(f)$--that $f$ is an inclusion of subobjects $\underline{N}\hookrightarrow \underline{M}$; \emph{i.e.} on the underlying modules, $f$ is the inclusion $N\subset M$. We consider the submodules
\begin{align*}
N^{\operatorname{sat}} &:=\left\{m\in M~|~\text{there~exists~}a\in A\otimes F:~a\cdot n\in N\right\}, \\
(\tau^*N)^{\operatorname{sat}} &:=\left\{m\in \tau^*M~|~\text{there~exists~}a\in A\otimes F:~a\cdot n\in \tau^*N\right\}
\end{align*}
of $M$ and $\tau^*M$ respectively. Both are finite projective $A\otimes F$-modules and so is $M/N^{\operatorname{sat}}$. One verifies that $\tau_M$ induces an isomorphism
\[
(\tau_N)^{\operatorname{sat}}:(\tau^*N)^{\operatorname{sat}}[\fj^{-1}]\stackrel{\sim}{\longrightarrow} N^{\operatorname{sat}}[\fj^{-1}].
\]
It remains to show that the couple $(N^{\operatorname{sat}},(\tau_N)^{\operatorname{sat}})$ forms an $A$-motive $\underline{N}^{\operatorname{sat}}$; namely, it suffices to prove the equality $\tau^*N^{\operatorname{sat}}=(\tau^*N)^{\operatorname{sat}}$ as submodules of $\tau^*M$. That $\underline{N}^{\operatorname{sat}}\hookrightarrow \underline{M}$ is a saturation of $f$ will then be clear. \\
We consider the following ideal of $A\otimes F$:
\[
\fv:=\operatorname{Ann}(N^{\operatorname{sat}}/N)=\{a\in A\otimes F~|~\text{for~all~}m\in N^{\operatorname{sat}}:~am\in N\}.
\]
As the modules $(\tau^*N)^{\operatorname{sat}}/\tau^*N$ and $N^{\operatorname{sat}}/N$ are isomorphic away from $\fj$ (via the induced map from $\tau_M$), we obtain 
\[
\fv=\operatorname{Ann}(N^{\operatorname{sat}}/N)=\operatorname{Ann}((\tau^*N)^{\operatorname{sat}}/\tau^*N))\subset \operatorname{Ann}(\tau^*N^{\operatorname{sat}}/\tau^*N)=\tau(\fv)
\]
as ideals of $(A\otimes F)[\fj^{-1}]$, where we used $\tau^*N^{\operatorname{sat}}\subset (\tau^*N^{\operatorname{sat}})$ for the above inclusion and \cite[Tag 07T8]{stack} for the last equality. Using prime ideal decomposition in the Dedekind domain $A\otimes F$ yields $\tau(\fv)=\fv$. Consequently, $\fv$ descends to an ideal of $A$ from which we obtain $(\tau^*N)^{\operatorname{sat}}=\tau^*N^{\operatorname{sat}}$, as desired. In addition, this shows that the cokernel of $N\hookrightarrow N^{\operatorname{sat}}$ is $A$-torsion, hence that $\underline{N}\to \underline{N}^{\operatorname{sat}}$ is an isogeny (use definition \ref{item:isogeny-a} for $a\in A$ a generator of a principal ideal in the form $\fv^h$).
\end{proof}

The previous lemma motivates the definition of the category of \emph{$A$-motives up to isogenies} (see \cite[Def. 2.3.1]{hartl-juschka}).
\begin{Definition}\label{def:tilde-motives}
Let $\cM^{\text{iso}}_R$ be the $K$-linear category whose objects are those of $\cM_R$ and where the hom-sets of two objects $\underline{M}$ and $\underline{N}$ is given by the $K$-vector space 
\begin{equation}
\Hom_{\cM^{\text{iso}}_R}(\underline{M},\underline{N}):=\Hom_{\cM_R}(\underline{M},\underline{N})\otimes_A K. \nonumber
\end{equation}
We call the objects of $\cM^{\text{iso}}_R$ the \emph{$A$-motives over $R$ up to isogenies}.
\end{Definition}
An isogeny in $\cM_R$ then becomes an isomorphism in $\cM^{\text{iso}}_R$. In particular, over $R=F$ a field, Lemma \ref{lem:satiso} implies that any morphism admits a kernel and a cokernel in $\cM^{\text{iso}}_F$. The following is detailed in \cite[Prop. 2.3.4]{hartl-juschka}:
\begin{Proposition}
The category $\cM^{\text{iso}}_F$ is abelian.
\end{Proposition}

\subsection{Extension modules in $\cM_R$}\label{sec:extensions-in-MF}
In this subsection, we are concerned with the computation of extension modules in categories of $A$-motives. Theorem \ref{mthm:extension} of the introduction is proved below (Theorem \ref{thm:cohomology-in-MR}). \\

Let $R$ be an $A$-algebra. As mentioned above, the category of $A$-motives over $R$ is not abelian. However, there is a notion of exact sequences in the category $\cM_R$ which we borrow from \cite[Rmk. 2.3.5(b)]{hartl-juschka}:
\begin{Definition}\label{def:exact-sequence}
We say that a sequence $0\to \underline{M}'\to \underline{M} \to \underline{M}''\to 0$ in $\cM_R$ is \emph{exact} if its underlying sequence of $A\otimes R$-modules is exact.
\end{Definition} 
The next proposition appears and is discussed in \cite[Rmk. 2.3.5(b)]{hartl-juschka} and will allow us to consider extension modules in Subsection \ref{subsec:extension-modules-of-mixed}. Although stated in the case where $R$ is a particular $A$-algebra and $\deg(\infty)=1$, its statement extends to our setting:
\begin{Proposition}\label{prop:MF-is-exact}
The category $\cM_R$ together with the notion of exact sequences as in Definition \ref{def:exact-sequence} forms an exact category (Definition \ref{def:exact-category}).
\end{Proposition}
\begin{proof}
We consider the forgetful functor $F:\cM_R\to \mathbf{Mod}_{A\otimes R}$. In particular, a sequence $\underline{S}$ in $\cM_R$ is exact if and only if $F(\underline{S})$ is exact. By Lemma \ref{lem:kernel-cokernel}, Proposition \ref{prop:Functor-to-exactCat} applies to the functor $F$, thus completing the proof.
\end{proof}

By the above proposition, we can make sense of $\Ext^i_{\cM_R}$ for $i\in \{0,1\}$ but, as $\cM_R$ is not abelian, there is an ambiguity about what one means for $i>1$. This ambiguity is cleared up in the appendix where we defined higher extension modules in exact categories. By the fact that $\cM_R$ is weakly idempotent complete (Corollary \ref{cor:wic}) and by the embedding Theorem \ref{thm:embedding-thm}\ref{item:abelian-embedding-reflects-epi}, higher extension modules in $\cM_R$ behave as well as in abelian categories.\\

Let $\underline{M}$ and $\underline{N}$ be two $A$-motives over $R$. The morphisms from $\underline{N}$ to $\underline{M}$ in $\cM_R$ are precisely the $A\otimes R$-linear maps of the underlying modules $f:N\to M$ such that $\tau_M\circ \tau^*f=f\circ \tau_N$. The module of degree zero extensions is set to be the module of homomorphisms in $\cM_R$, and hence:
\begin{equation}
\Ext^0_{\cM_R}(\underline{N},\underline{M}):=\Hom_{\cM_R}(\underline{N},\underline{M})=\{f\in \Hom_{A\otimes R}(N,M)~|~\tau_M\circ \tau^*f=f\circ \tau_N\}. \nonumber
\end{equation}
The next proposition computes the module of degree one extensions.
\begin{Proposition}\label{prop-cohomology-in-tilde}
There is a surjective morphism of $A$-modules, functorial in both $\underline{N}$ and $\underline{M}$,
\begin{equation}
\iota:\Hom_{A\otimes R}(\tau^*N,M)[\fj^{-1}]\twoheadrightarrow \Ext^1_{\cM_R}(\underline{N},\underline{M}) \nonumber
\end{equation}
whose kernel is $\left\{f\circ \tau_N-\tau_M\circ \tau^*f~|~f\in \Hom_{A\otimes R}(N,M)\right\}$. $\iota$ maps a morphism $u\in \Hom_{A\otimes R}(\tau^*N,M)[\fj^{-1}]$ to the class of the extension $[M\oplus N,\left(\begin{smallmatrix} \tau_M & u \\ 0 & \tau_N \end{smallmatrix}\right)]$ in $\Ext^1_{\cM_R}(\underline{N},\underline{M})$.
\end{Proposition}
\begin{proof}
Consider an exact sequence $[\underline{E}]:0\to\underline{M}\stackrel{i}{\to}  \underline{E} \stackrel{\pi}{\to}  \underline{N}\to 0$ in $\cM_R$, namely, an exact sequence of the underlying $A\otimes R$-modules with commuting $\tau$-action. Because $N$ is a projective module, there exists $s:N\to E$ a section of the underlying short exact sequence of $A\otimes R$-modules. We let $\xi:=i\oplus s:M\oplus N\to E$. This produces an equivalence of extensions:
\begin{equation}
\begin{tikzcd}
0\arrow[r] & \underline{M}  \arrow[r,"i"] &   \underline{E}  \arrow[r,"\pi"] & \underline{N}  \arrow[r] & 0 \\
0\arrow[r] & \underline{M} \arrow[u,"\id"] \arrow[r] & (M\oplus N, \xi^{-1}\circ \tau_E\circ \xi) \arrow[u,"\xi"] \arrow[r] & \underline{N} \arrow[u,"\id"] \arrow[r] & 0
\end{tikzcd}
\nonumber
\end{equation}
Because $\xi^{-1}\circ \tau_M \circ \xi$ is an isomorphism from $\tau^*M[\fj^{-1}]\oplus \tau^*N[\fj^{-1}]$ to $M[\fj^{-1}]\oplus N[\fj^{-1}]$ which restricts to $\tau_M$ on the left and to $\tau_N$ on the right, there exists $u\in \Hom_{(A\otimes R)[\fj^{-1}]}(\tau^*N[\fj^{-1}],M[\fj^{-1}])=\Hom_{A\otimes R}(\tau^*N,M)[\fj^{-1}]$ such that $\xi^{-1}\circ \tau_E\circ \xi=\left(\begin{smallmatrix} \tau_M & u \\ 0 & \tau_N \end{smallmatrix}\right)$. We have just shown that the map
\begin{equation}
\iota:\Hom_{A\otimes R}(\tau^*N,M)[\fj^{-1}]\longrightarrow \Ext^1_{\cM_R}(\underline{N},\underline{M}), \quad u\longmapsto [M\oplus N,\left(\begin{smallmatrix} \tau_M & u \\ 0 & \tau_N \end{smallmatrix}\right)]\nonumber
\end{equation}
is onto. Note that $\iota(0)$ corresponds to the class of the split extension. Further, $\iota(u+v)$ corresponds to the Baer sum of $\iota(u)$ and $\iota(v)$. In addition, given the exact sequence $[\underline{E}]$ and $a\in A$, the pullback of multiplication by $a$ on $N$ by $\pi$ gives another extension which defines $a\cdot [\underline{E}]$. If $[\underline{E}]=\iota(u)$, it is formal to check that $a\cdot [\underline{E}]=\iota(au)$. As such, $\iota$ is a surjective $A$-module morphism. To find its kernel, it suffices to determine whenever $\iota(u)$ is equivalent to the split extension. This happens if and only if there is a commutative diagram in $\cM_R$ of the form
\begin{equation}
\begin{tikzcd}
0\arrow[r] & \underline{M}  \arrow[r]\arrow[d,"\id_M"] &   \underline{M}\oplus \underline{N}  \arrow[d,"h"]\arrow[r] & \underline{N}  \arrow[d,"\id_N"]\arrow[r] & 0 \\
0\arrow[r] & \underline{M} \arrow[r] & \left[M\oplus N, \left(\begin{smallmatrix} \tau_M & u \\ 0 & \tau_N \end{smallmatrix}\right)\right] \arrow[r] & \underline{N} \arrow[r] & 0
\end{tikzcd}
\nonumber
\end{equation}
where $h$ is an isomorphism in $\cM_R$. Since the diagram commutes in the category of $A\otimes R$\nobreakdash-modules, it follows that $h$ is of the form $\left(\begin{smallmatrix} \id_M & f  \\ 0 & \id_N \end{smallmatrix}\right)$ for an $A\otimes R$-linear map $f:N\to M$. Because it is a diagram in $\cM_R$, it further requires commuting $\tau$-action, that is: 
\begin{equation}
\begin{pmatrix}\tau_M & u \\ 0 & \tau_N \end{pmatrix} \tau^*\begin{pmatrix}\id_M & f \\ 0 & \id_N \end{pmatrix}=\begin{pmatrix}\id_M & f \\ 0 & \id_N \end{pmatrix}\begin{pmatrix}\tau_M & 0 \\ 0 & \tau_N \end{pmatrix}. \nonumber
\end{equation}
The above equation amounts to $u=f\circ \tau_N-\tau_M\circ \tau^*f$, and hence 
\begin{equation}
\ker(\iota)=\left\{f\circ \tau_N-\tau_M\circ \tau^*f~|~f\in \Hom_{A\otimes R}(N,M)\right\} .\nonumber
\end{equation}
This concludes the proof.
\end{proof}

\begin{Corollary}\label{cor:ontoonextension}
Let $f:\underline{M}\to \underline{M}''$ be an epimorphism in $\cM_R$. Then, the induced map $\Ext^1_{\cM_R}(\underline{N},\underline{M})\to \Ext^1_{\cM_R}(\underline{N},\underline{M}'')$ is onto. 
\end{Corollary}
\begin{proof}
Because $N$ is finite projective over $A\otimes R$, so is  $\tau^*N$ and hence the functor $\Hom_{A\otimes R}(\tau^*N,-)$ is exact. That $f$ is an epimorphism means that $f$ is a surjective morphism of the underlying modules. The induced morphism 
\[
\Hom_{A\otimes R}(\tau^*N,M)[\fj^{-1}]\longrightarrow \Hom_{A\otimes R}(\tau^* N,M'')[\fj^{-1}]
\]
is therefore surjective, and we conclude by Proposition \ref{prop-cohomology-in-tilde}.
\end{proof}  

As an consequence of the above corollary and Proposition \ref{prop:vanishing-higher-ext} of the appendix, we record:
\begin{Proposition}\label{prop:higher-vanishing}
The modules $\Ext^i_{\cM_R}(\underline{N},\underline{M})$ vanish for $i>1$. 
\end{Proposition}

Propositions \ref{prop-cohomology-in-tilde} and \ref{prop:higher-vanishing} combine to show that the cohomology of the complex
\begin{equation}\label{eq:complex-extension}
\left[\Hom_{A\otimes R}(N,M)\xrightarrow{\tau_N^{\vee}-\tau_M}\Hom_{A\otimes R}(\tau^*N,M)[\fj^{-1}]\right],
\end{equation}  
sitting in degree zero and one, computes the extension modules in the corresponding degree. In that respect, we make a slight abuse of notations in making the next definition:
\begin{Definition}
We denote by $\RHom_{\cM_R}(\underline{N},\underline{M})$ the complex \eqref{eq:complex-extension}.
\end{Definition}
\begin{Remark}
The notation "$\RHom$" is generally reserved for the \emph{derived $\Hom$}, namely, the universal extension of the $\Hom$ functor in the derived category. Here we do not consider a derived category, since $\cM_R$ is not abelian, and there is no confusion possible.
\end{Remark}

The category $(\cM_R,\otimes, \mathbbm{1})$ is a monoidal exact category in the sense of Definition \ref{def:monoidal-exact}. From Corollary \ref{cor:dualizable-tensor}, we have functorial isomorphisms for all $i\geq 0$:
\begin{equation}\label{eq:explici-ext1-map-dual}
\Ext^i_{\cM_R}(\underline{N},\underline{M})\stackrel{\sim}{\longrightarrow} \Ext^i_{\cM_R}(\mathbbm{1},\underline{M}\otimes \underline{N}^{\vee}). 
\end{equation} 
In particular, there is no loss of generality in considering extension modules of the form $\Ext^i_{\cM_R}(\mathbbm{1},\underline{M})$. From now on, we will be mainly interested in extension modules of the latter form. For readability, we introduce here and elsewhere the notation $(\id-\tau_M)(M)$ to designate the sub-$A$-module
\[
(\id-\tau_M)(M):=\left\{m-\tau_M(\tau^*m)~|~m\in M\right\}
\]
of $M[\fj^{-1}]$. We shall restate the main results of this section in this case (repeated from Theorem \ref{mthm:extension} of the introduction).
\begin{Theorem}\label{thm:cohomology-in-MR}
The cohomology of $\RHom_{\cM_R}(\mathbbm{1},\underline{M})$ is computed by the cohomology of the complex of $A$-modules
\begin{equation}
\left[M\xrightarrow{\id-\tau_M} M[\fj^{-1}]\right] \nonumber
\end{equation}
sitting in degree zero and one. Further, the natural $A$-linear surjection
\begin{equation}
\iota:M[\fj^{-1}]\twoheadrightarrow \Ext^1_{\cM_R}(\mathbbm{1},\underline{M}), \nonumber
\end{equation}
whose kernel is $(\id-\tau_M)(M)$, is given by mapping $m\in M[\fj^{-1}]$ to the class of the extension 
\[0\to \underline{M}\to \left[M\oplus (A\otimes R),\left(\begin{smallmatrix} \tau_M & m\cdot \mathbf{1} \\ 0 & \mathbf{1} \end{smallmatrix}\right)\right]\to \mathbbm{1}\to 0.
\] 
\end{Theorem}

\begin{Remark}
From $\Hom_{\cM^{\text{iso}}_R}(-,-)=\Hom_{\cM_R}(-,-)\otimes_A K$ (\emph{cf}. Definition \ref{def:tilde-motives}), the extension spaces of $\mathbbm{1}$ by $\underline{M}$ in the $K$-linear category $\cM^{\text{iso}}_R$ are computed by the complex 
\begin{equation}
\left[M\otimes_A K\xrightarrow{\id-\tau_M} M[\fj^{-1}]\otimes_A K\right]. \nonumber
\end{equation}
\end{Remark}

\subsection{Extensions having good reduction}\label{sec:extensions-having-good-reduction}
We introduce here the function field analogue of the $\ell$-adic realization functor (Definition \ref{def:m-adic realization functor}), and show that it is exact (Proposition \ref{prop-exact-tate}). It will allow us to define \emph{extensions with good reduction} next (Definition \ref{def:ext-good-red-with-resp-to-ell}).\\

Let $\ell$ be a maximal ideal of $A$ and denote by $\cO_{\ell}$ the completed local ring of $A$ at $\ell$. We let $F$ be a field containing $K$ and let $\kappa:A\to F$ be the inclusion. Let $F^s$ be a separable closure of $F$ and denote by $G_F=\Gal(F^s|F)$ the absolute Galois group of $F$ equipped with the profinite topology. The group $G_F$ acts $A$-linearly on the left-hand side of the tensor $A\otimes F^s$, and this action extends by continuity to an $\cO_{\ell}$-linear action of $G_F$ on the algebra
\[
\cA_{\ell}(F^s):=(A\otimes F^s)^{\wedge}_{\ell}=\varprojlim_n~(A\otimes F^s)/\ell^n(A\otimes F^s)
\]
leaving $\cA_{\ell}(F):=(A\otimes F)^{\wedge}_{\ell}$ invariant. If $\bF_{\ell}$ denotes the residue field of $\cO_{\ell}$ and $\pi$ a uniformizer, we have an identification $\cA_{\ell}(F)=(\bF_{\ell}\otimes F)[\![\pi]\!]$.
\begin{Remark}\label{rmk:Witt-vectors}
In the function field/number field dictionary, the assignment $R\mapsto \cA(R)$ is akin to the $p$-typical Witt vectors construction $R\mapsto W(R)$ (\emph{e.g.} \cite[\S 1.1]{hartl-dico}). 
\end{Remark}
Let $\underline{M}=(M,\tau_M)$ be an $A$-motive over $F$ of rank $r$. Let $\underline{M}_{F^s}=(M_{F^s},\tau_M)$ be the $A$-motive over $F^s$ obtained from $\underline{M}$ by base-change. $G_F$ acts $\cO_{\ell}$-linearly on:
\[
(M_{F^s})^{\wedge}_{\ell}:=\varprojlim_{n}~(M\otimes_F F^s)/\ell^n (M\otimes_F F^s)=M\otimes_{A\otimes F}\cA_{\ell}(F^s)
\]
and leaves the submodule $M^{\wedge}_{\ell}=M\otimes_{A\otimes F}\cA_{\ell}(F)$ invariant. Following \cite[\S 1.8]{anderson}, we define:
\begin{Definition}\label{def:m-adic realization functor}
The \emph{$\ell$-adic realization} $\operatorname{T}_{\ell}\underline{M}$ of $\underline{M}$ consists of the $\cO_{\ell}$-module
\begin{equation}
\operatorname{T}_{\ell}\underline{M}:=\left\{m\in (M_{F^s})^{\wedge}_{\ell}~|~m=\tau_M(\tau^* m)\right\} \nonumber
\end{equation}
together with the compatible action of $G_F$ which it inherits as a submodule of $(M_{F^s})^{\wedge}_{\ell}$. 
\end{Definition}
\begin{Remark}
In \cite{mornev-isocrystal}, Mornev extended this construction to the situation where $\ell$ is the closed point~$\infty$.
\end{Remark}

The next lemma is well-known in the case of $\tau$-sheaves (\emph{e.g.} \cite[Prop.6.1]{taguchi-wan}). 
\begin{Lemma}\label{lem-rankr-tate-module}
The map $\operatorname{T}_{\ell}\underline{M}\otimes_{\cO_{\ell}}\cA_{\ell}(F^s)\to (M_{F^s})^{\wedge}_{\ell}$, $\omega\otimes f\mapsto \omega\cdot f$ is an isomorphism of $\cA_{\ell}(F^s)$-modules. In particular, the $\cO_{\ell}$-module $\operatorname{T}_{\ell}\underline{M}$ is free of rank $r$ and the action of $G_F$ on $\operatorname{T}_{\ell}\underline{M}$ is continuous.
\end{Lemma}
\begin{proof}
Let $n\geq 1$. In the ring $A\otimes F$, the ideals $\ell^n$ and $\fj$ are coprime. Hence, the following composition of $A\otimes F$-linear maps is a well-defined isomorphism:
\[
\varphi_n : \tau^*(M/\ell^n M)\cong \frac{(\tau^*M)}{\ell^n (\tau^*M)}\cong \frac{(\tau^*M)[\fj^{-1}]}{\ell^n (\tau^*M)[\fj^{-1}]}\stackrel{\tau_M}{\longrightarrow} \frac{M[\fj^{-1}]}{\ell^n M[\fj^{-1}]}\cong M/\ell^nM.
\]   
The data of $\varphi_n$ induces a semi-simple $q$-linear map\footnote{For $k$ a field containing $\bF$ and $V$ a $k$-vector space, an $\bF$-linear endomorphism $f$ of $V$ is \emph{$q$-linear} if $f(rv)=r^qf(v)$ for all $r\in k$ and $v\in V$.} (in the sense of \cite[\S 1]{katz}) on the finite dimensional $F^s$-vector space:
\begin{equation}
(M\otimes_F F^s)/\ell^n(M\otimes_F F^s)=(M_{F^s})^{\wedge}_{\ell}/\ell^n (M_{F^s})^{\wedge}_{\ell}. \nonumber
\end{equation}
By Lang's isogeny theorem (\emph{e.g.} \cite[Prop. 1.1]{katz}), the multiplication map
\begin{equation}\label{eq-lang-mod-n}
\{m\in (M_{F^s})^{\wedge}_{\ell}/\ell^n (M_{F^s})^{\wedge}_{\ell}~|~\tau_M(\tau^*m)=m\}\otimes_{\bF} F^s\to (M_{F^s})^{\wedge}_{\ell}/\ell^n (M_{F^s})^{\wedge}_{\ell}
\end{equation}
is an isomorphism. Taking the inverse limit of \eqref{eq-lang-mod-n} over all $n$ yields the desired isomorphism.

As $(M_{F^s})^{\wedge}_{\ell}$ is free of rank $r$ over $\cA_{\ell}(F^s)$, the same is true for the module $(M_{F^s})^{\wedge}_{\ell}/\ell^n (M_{F^s})^{\wedge}_{\ell}$ over $(A/\ell^{n})\otimes F^s$.  The isomorphism \eqref{eq-lang-mod-n} implies that the $A/\ell^{n}$-module
\begin{equation}
[(M_{F^s})^{\wedge}_{\ell}/\ell^n(M_{F^s})^{\wedge}_{\ell}]^{\tau_M=1}:=\{m\in (M_{F^s})^{\wedge}_{\ell}~|~\tau_M(\tau^*m)=m\}\nonumber
\end{equation}
is free of rank $r$ over $A/\ell^n$. Their projective limit $\operatorname{T}_{\ell}\underline{M}$ is thus a free $\cO_{\ell}$-module of rank $r$.

By definition, the action of $G_F$ on $\operatorname{T}_{\ell}\underline{M}$ is continuous if, and only if, for all $n$, the induced action of $G_F$ on $\operatorname{T}_{\ell}\underline{M}/\ell^n \operatorname{T}_{\ell}\underline{M}$ factors through $\Gal(E_n|F)$ for some finite Galois extension $E_n$ of $F$. Let $\textbf{t}=\{t_{1},\ldots,t_{s}\}$ be a basis of the finite dimensional $F$-vector space $M^{\wedge}_{\ell}/\ell^nM^{\wedge}_{\ell}$. Let $F_M$ be the matrix of $\tau_M$ written in the basis $\tau^*\textbf{t}$ and $\textbf{t}$. Let $\bm{\omega}=\{\omega_1,\ldots,\omega_s\}$ be a basis of $\operatorname{T}_{\ell}\underline{M}/\ell^n\operatorname{T}_{\ell}\underline{M}$ over $\bF$. By \eqref{eq-lang-mod-n}, $\bm{\omega}$ is a basis of $(M_{F^s})^{\wedge}_{\ell}/\ell^n(M_{F^s})^{\wedge}_{\ell}$ over $F^s$, and we let $w_{ij}\in F^s$ be the coefficients of $\bm{\omega}$ expressed in $\mathbf{t}$, that is, for $i\in \{1,\ldots,s\}$, $\omega_i=\sum{w_{ij}t_j}$. We let $E_n$ denote the Galois closure of the finite separable extension $F(w_{ij}|(i,j)\in\{1,\ldots,s\}^2)$ of $F$ in $F^s$. Then,
\begin{equation}
\operatorname{T}_{\ell}\underline{M}/\ell^n\operatorname{T}_{\ell}\underline{M}=\{m\in (M\otimes_{F}E_n)/\ell^n(M\otimes_{F}E_n)~|~\tau_M(\tau^*m)=m\}. \nonumber
\end{equation}
That is, the action of $G_F$ factors through $\Gal(E_n|F)$, as desired.
\end{proof}

\begin{Proposition}\label{prop-exact-tate}
The following sequence of $\cO_{\ell}[G_F]$-modules is exact
\begin{equation}
0\to \operatorname{T}_{\ell}\underline{M}\longrightarrow (M_{F^s})^{\wedge}_{\ell}\xrightarrow{\id-\tau_M} (M_{F^s})^{\wedge}_{\ell} \longrightarrow 0. \nonumber
\end{equation}
\end{Proposition}
\begin{proof}
Everything is clear but the surjectivity of $\id-\tau_M$. From Lemma \ref{lem-rankr-tate-module}, one reduces to the situation where $\underline{M}=\mathbbm{1}$ is the unit $A$-motive, in which case $(M_{F^s})^{\wedge}_{\ell}=\cA_\ell(F^s)=(\bF_{\ell}\otimes F^s)[\![\pi]\!]$ and $\tau_M=\id\otimes \Frob_q$. Let $f=\sum_{n\geq 0}{a_n\pi^n}$ be a series in $(\bF_{\ell}\otimes F^s)[\![\pi]\!]$ and let $b_n\in \bF_{\ell}\otimes F^s$ be such that $[\id_{\bF_{\ell}}\otimes (\id-\Frob_q)](b_n)=a_n$ (which exists as $F^s$ is separably closed). Then, $g:=\sum_{n\geq 0}{b_n\pi^n}$ in $\cA_{\ell}(F^s)$ satisfies 
\[
f=(\id-\tau_M)(g),
\]
and hence $\id-\tau_M$ is surjective.
\end{proof}

We obtain the following:
\begin{Corollary}\label{adic-realization-exact}
The functor $\underline{M}\mapsto \operatorname{T}_{\ell}\underline{M}$, from $\cM_F$ to the category of continuous $\cO_{\ell}$-linear $G_F$-representations, is exact.
\end{Corollary}
\begin{proof}
Let $S:0\to \underline{M}'\to \underline{M}\to \underline{M}''\to 0$ be an exact sequence in $\cM_F$. The underlying sequence of $A\otimes F$-modules is exact, and because $\cA_{\ell}(F^s)$ is flat over $A\otimes F$ (\emph{e.g.} \cite[AC.III \S 4, Thm. 3(iii)]{bourbaki}), the sequence of $\cA_{\ell}(F^s)$-modules $(S_{F^s})^{\wedge}_{\ell}$ is exact. In particular, the next commutative diagram of $\cO_{\ell}$\nobreakdash-modules has exact rows:
\begin{equation}
\begin{tikzcd}
0 \arrow[r] & (M'_{F^s})^{\wedge}_{\ell} \arrow[r]\arrow[d,"\id-\tau_{M'}"] & (M_{F^s})^{\wedge}_{\ell}\arrow[r]\arrow[d,"\id-\tau_{M}"] & (M''_{F^s})^{\wedge}_{\ell} \arrow[r]\arrow[d,"\id-\tau_{M''}"] & 0 \\
0 \arrow[r] & (M'_{F^s})^{\wedge}_{\ell} \arrow[r] & (M_{F^s})^{\wedge}_{\ell} \arrow[r] & (M''_{F^s})^{\wedge}_{\ell} \arrow[r] & 0
\end{tikzcd}\nonumber
\end{equation}
and the Snake Lemma together with Proposition \ref{prop-exact-tate} yields that $\operatorname{T}_{\ell}S$ is exact.
\end{proof}

Let $\underline{M}$ be an $A$-motive over $F$. From Corollary \ref{adic-realization-exact}, the functor $\operatorname{T}_{\ell}$ induces an $A$-linear morphism:
\begin{equation}\label{eq:Tell-exact-on-Ext}
\Ext^1_{\cM_F}(\mathbbm{1},\underline{M})\longrightarrow \operatorname{H}^1(G_F,\operatorname{T}_{\ell}\underline{M})
\end{equation}
into the first continuous cohomology group of $G_F$ with values in $\operatorname{T}_{\ell}\underline{M}$. The next proposition is devoted to the explicit determination of \eqref{eq:Tell-exact-on-Ext}. 

\begin{Proposition}\label{prop:diagram-explicit-cocycle}
There is a commutative diagram of $A$-modules:
\begin{equation}
\begin{tikzcd}
\Ext^1_{\cM_F}(\mathbbm{1},\underline{M}) \arrow[r,"\eqref{eq:Tell-exact-on-Ext}"]& \operatorname{H}^1(G_F,\operatorname{T}_{\ell}\underline{M}) \\
\displaystyle\frac{M[\fj^{-1}]}{(\id-\tau_M)(M)}\arrow[u,"\iota","\wr"']\arrow[r] & \displaystyle\frac{M^{\wedge}_{\ell}}{(\id-\tau_M)(M^{\wedge}_{\ell})} \arrow[u,"\wr"']
\end{tikzcd}\nonumber
\end{equation}
where the right vertical morphism maps the class of $f\in M^{\wedge}_{\ell}$ to the class of the cocycle $\sigma\mapsto \xi-\,^{\sigma}\xi$, $\xi$ being any solution in $(M_{F^s})^{\wedge}_{\ell}$ of $f=\xi-\tau_M(\tau^*\xi)$.
\end{Proposition}\label{prop:explicit-cocycle}

\begin{proof}
We prove the commutativity. Let $[\underline{E}]:0\to \underline{M}\to \underline{E}\to \mathbbm{1}\to 0$ be a class in $\Ext^1_{\cM_F}(\mathbbm{1},\underline{M})$ of the form $\iota(m)$ for some $m\in M[\fj^{-1}]$ (Theorem \ref{thm:cohomology-in-MR}). The $\ell$-adic realization $\operatorname{T}_{\ell}\underline{E}$ of $\underline{E}$ is the $\cO_\ell[G_F]$-module consisting of solutions $\xi\oplus a\in (M_{F^s})^{\wedge}_{\ell}\oplus \cA_{\ell}(F^s)$ of the equation
\begin{equation}
\begin{pmatrix} \tau_M & m\cdot \mathbf{1} \\ 0 & \mathbf{1} \end{pmatrix} \begin{pmatrix}\tau^*\xi \\ \tau^*a \end{pmatrix}=\begin{pmatrix}\xi \\ a \end{pmatrix} \nonumber
\end{equation}
(see Definition \ref{def:m-adic realization functor}). The above equality amounts to $a\in \cO_{\ell}$ and $\xi-\tau_M(\tau^*\xi)=am$. A splitting of $[\operatorname{T}_{\ell}\underline{E}]$ as a sequence of $\cO_{\ell}$-modules corresponds to the choice of a particular solution $\xi_m\in (M_{F^s})^{\wedge}_{\ell}$ of $\xi-\tau_M(\tau^*\xi)=m$ (whose existence is ensured by Proposition \ref{prop-exact-tate}). We then have 
\begin{equation}
\operatorname{T}_{\ell}\underline{M}\oplus \cO_{\ell}\stackrel{\sim}{\longrightarrow} \operatorname{T}_{\ell}\underline{E}, \quad (\omega,a)\longmapsto (\omega+a\xi_m,a). \nonumber
\end{equation}
It follows that the morphism \eqref{eq:Tell-exact-on-Ext} in the diagram in Proposition 2.29 maps $[\underline{E}]$ to the class of the cocycle $(\sigma\mapsto \xi_m-{}^\sigma\xi_m)$, where $\xi_m$ is any solution in $(M_{F^s})^{\wedge}_{\ell}$ of the equation $\xi-\tau_M(\tau^*\xi)=m$. In other words, the square commutes.

It remains to check that the right vertical morphism is an isomorphism. Applying the functor of $G_F$-invariants to the short exact sequence of Proposition \ref{prop-exact-tate}, we obtain a long exact sequence of cohomology:
\[
M^{\wedge}_{\ell}\xrightarrow{\id-\tau_M}M^{\wedge}_{\ell} \longrightarrow \operatorname{H}^1(G_F,\operatorname{T}_{\ell}\underline{M})\longrightarrow \operatorname{H}^1(G_F,(M_{F^s})^{\wedge}_{\ell}).
\]
Therefore, it is sufficient to prove that $\operatorname{H}^1(G_F,(M_{F^s})^{\wedge}_{\ell})$ vanishes. We have ${(M_{F^s})^{\wedge}_{\ell}=M\otimes_{A\otimes F}\cA_{\ell}(F^s)}$, $G_F$ acting on the right-hand side of the tensor. Hence, it is enough to prove that the module $\operatorname{H}^1(G_F,\cA_{\ell}(F^s))$ vanishes. This follows from: 
\[
\operatorname{H}^1(G_F,\cA_{\ell}(F^s)) \stackrel{(1)}{=} \varprojlim_n \operatorname{H}^1(G_F,(A/\ell^n)\otimes F^s)= \varprojlim_n (A/\ell^n)\otimes \operatorname{H}^1(G_F,F^s) \stackrel{(2)}{=}(0)
\]
where we used $(1)$ the explicit form $\cA_\ell(F^s)=(\bF_\ell\otimes F^s)[\![\pi]\!]$ and $(2)$ that $\operatorname{H}^1(G_F,F^s)$ vanishes by the additive version of Hilbert's 90 Theorem \cite[x.\S1, Prop. 1]{serre}.
\end{proof}

For the remainder of this section, let us assume that $F=F_{\fp}$ is a local function field with valuation ring $\cO_\fp$ and maximal ideal $\fp$. Let $F_{\fp}^{\text{ur}}$ be the maximal unramified extension of $F_{\fp}$ in $F_{\fp}^s$. Let $I_\fp$ be the inertia subgroup of $G_\fp=G_{F_{\fp}}$.
\begin{Definition}\label{def:ext-good-red-with-resp-to-ell}
Let $\underline{M}$ be an $A$-motive over $F_{\fp}$, and let $\ell$ be a maximal ideal in $A$. We say that an extension $[\underline{E}]$ of $\mathbbm{1}$ by $\underline{M}$ has \emph{good reduction with respect to $\ell$} if $[\underline{E}]$ lies in the kernel of $\Ext^1_{\cM_{F_\fp}}(\mathbbm{1},\underline{M}) \to \operatorname{H}^1(I_{\fp},\operatorname{T}_{\ell}\underline{M})$. We denote by $\Ext^1_{\operatorname{good}}(\mathbbm{1},\underline{M})_{\ell}$ the kernel of \eqref{eq:Tell-exact-on-Ext}, namely the module of extensions having good reduction with respect to $\ell$.
\end{Definition}

The proof of Proposition \ref{prop:diagram-explicit-cocycle} with $I_\fp=\Gal(F^s_{\fp}|F_{\fp}^{\text{ur}})$ in place of $G_F$ applies almost \emph{verbatim} to show:
\begin{Proposition}\label{prop-characterization-of-extension-good}
There is a commutative diagram of $A$-modules:
\begin{equation}
\begin{tikzcd}
\Ext^1_{\cM_{F_\fp}}(\mathbbm{1},\underline{M}) \arrow[r,"\eqref{eq:Tell-exact-on-Ext}"]& \operatorname{H}^1(I_\fp,\operatorname{T}_{\ell}\underline{M}) \\
\displaystyle\frac{M[\fj^{-1}]}{(\id-\tau_M)(M)}\arrow[u,"\iota","\wr"']\arrow[r] & \displaystyle\frac{(M_{F_{\fp}^{\text{ur}}})^{\wedge}_{\ell}}{(\id-\tau_M)((M_{F_{\fp}^{\operatorname{ur}}})^{\wedge}_{\ell})} \arrow[u,"\wr"']
\end{tikzcd}\nonumber
\end{equation}
In particular, given $m\in M[\fj^{-1}]$, the following are equivalent:
\begin{enumerate}[label=$(\alph*)$]
\item The extension $\iota(m)$ has good reduction with respect to $\ell$,
\item The equation $\xi-\tau_M(\tau^*\xi)=m$ admits a solution $\xi$ in $(M_{F^{\emph{ur}}_{\fp}})^{\wedge}_{\ell}$.
\end{enumerate}
\end{Proposition}

Definition \ref{def:ext-good-red-with-resp-to-ell} should presumably be independent of $\ell$. The analogous statement is the counterpart of Conjecture \ref{item:conjectureC3} which we state next:
\begin{Conjecture}[Independence of $\ell$]
Assume $\kappa(A)\subset \cO_{\fp}$. Let $[\underline{E}]$ be an extension of $\mathbbm{1}$ by $\underline{M}$ in the category $\cM_{F_\fp}$. If $[\underline{E}]$ has good reduction with respect to an ideal $\ell$ of $A$ for which $\kappa(\ell)\cO_{\fp}=\cO_{\fp}$, then it has good reduction with respect to any ideal $\ell'$ with the same property. 
\end{Conjecture}

\section{Mixed $A$-motives and their extension modules}\label{sec:mixed A motives}
We discuss here the notion of mixedness for Anderson $A$-motives. In the case where $A$ is a univariate polynomial algebra, the definition of \emph{pure} $t$-motives is traced back to the work of Anderson \cite[1.9]{anderson}. \emph{Mixed} $A$-motives were first mentioned in the talk of Pink at the Arbeitstagung in Bonn from 1997\footnote{Available at the address \href{https://people.math.ethz.ch/~pink/ftp/HS-AT97.pdf}{https://people.math.ethz.ch/~pink/ftp/HS-AT97.pdf}.}. A systematic study, however, appeared only recently in the work of Hartl--Juschka \cite[\S 3]{hartl-juschka} under the condition that the place $\infty$ has degree one and the base field $R=F$ is algebraically closed. Our presentation deals with the most general case: arbitrary curve $C$ and place $\infty$, and over an arbitrary base field $F$. Compared to \cite{taelman} or \cite{hartl-juschka}, the new difficulty is to deal with non perfect fields $F$, since then, the slope filtration for isocrystals does not necessarily split. 

\subsection{Mixed Anderson $A$-motives}\label{subsec:mixed-A-motives}
\subsubsection*{Isocrystals over a field}
In this subsection, we present some materials on \emph{function fields isocrystal} following \cite{mornev-isocrystal}. Our objective is to prove existence and uniqueness of the slope filtration with pure subquotients having increasing slopes. The general theory of slope filtrations has been developed in \cite{andre}, and the results of interest for us on isocrystals appear in \cite{hartl} in the case $A$ is a univariate polynomial algebra. The new account of this subsection is the adaptation of \cite[Prop 1.5.10]{hartl} to allow more general ring $A$ (see Theorem \ref{thm:HN-filtration}). All of this are key steps towards the definition of weights and mixedness (Definition \ref{def:weights-A-motives}). \\

We begin with some general notations. Let $R$ be an $\bF$-algebra and let $E$ be a \emph{local function fields}. By that, we mean that $E$ is the field of Laurent series over a finite field extension $k$ of $\bF$ in the formal variable $\pi$, $\cO$ the subring of $E$ consisting of power series over $k$ and $\fm$ the maximal ideal of $\cO$. Explicitly $E=k (\!(\pi)\!)$, $\cO=k[\![\pi]\!]$ and $\fm=\pi\cO$. In the sequel, $E$ will correspond to the local field of $(C,\cO_C)$ at a closed point of $C$. \\

Extending the notation introduced in Subsection \ref{sec:extensions-having-good-reduction} in the context of the $\ell$-adic realization functor, we denote by $\cA(R)$ the completion of the ring $\cO\otimes R$ at its ideal $\fm\otimes R$:
\begin{equation}\label{eq:A}
\cA(R):=\varprojlim_{n}~ (\cO\otimes R)/(\fm^n\otimes R)
\end{equation}
and we let $\cB(R)$ be the ring $E\otimes_{\cO}\cA(R)$. Through the previous identifications, we readily check that $\cA(R)=(k\otimes R)[\![\pi]\!]$ and $\cB(R)=(k\otimes R)(\!(\pi)\!)$. The invertible elements of $\cB(R)$ are described as:
\[
\cB(R)^{\times}=\bigcup_{n\in \bZ} \pi^n\cdot\cA(R)^{\times}, \quad \cA(R)^{\times}=(\bF_\ell\otimes R)^{\times}+\pi\cA(R).
\]
Given $x\in \cB(R)^{\times}$, we define $v_{\pi}(x)$ as the unique $n\in \bZ$ for which $x\in \pi^n\cdot \cA(R)^{\times}$. \\

Let $\tau:\cO\otimes R\to \cO\otimes R$, be the $\cO$-linear map induced by $a\otimes r\mapsto a\otimes r^q$. We shall also denote by $\tau$ its continuous extension to $\cA(R)$ or $\cB(R)$. Observe that $v_{\pi}$ is preserved by $\tau$. Again, we denote by $\mathbf{1}$ the canonical $A\otimes R$-linear morphisms $\tau^*\cA(R)\to \cA(R)$ and $\tau^*\cB(R)\to \cB(R)$. \\

For the remaining of this subsection, we assume that $R=F$ is a field. 
\begin{Definition}\label{def:isocrystal}
An \emph{isocrystal $\underline{D}$ over $F$} is a pair $(D,\varphi_D)$ where $D$ is a free $\cB(F)$-module of finite rank and $\varphi_D:\tau^*D\to D$ is a $\cB(F)$-linear isomorphism. \\
A morphism $(D,\varphi_D)\to (C,\varphi_C)$ of isocrystals is a $\cB(F)$-linear morphism of the underlying modules $f:D\to C$ such that $f\circ \varphi_D=\varphi_C\circ \tau^*f$. We let $\operatorname{IC}_{F}$ be the category of isocrystals over $F$.
\end{Definition}

\begin{Remark}
Pursuing the analogy of Remark \ref{rmk:Witt-vectors}, isocrystals are the analogue of the eponymous object in $p$-adic Hodge theory (we refer to \cite[\S 3.5]{hartl-dico}). In both settings, such objects carry a slope filtration (see Theorem \ref{thm:HN-filtration} for the function fields one). For number fields, isocrystals are only defined at finite places, whereas for function fields, isocrystals are defined regardless of the finiteness of the place. In the next subsection, we use the slope filtration at $\infty$ in order to define weights. 
\end{Remark}

We define the \emph{rank $\operatorname{rk}\underline{D}$ of $\underline{D}$} to be the rank of $D$ over $\cB(F)$. If $\underline{D}$ is nonzero, let $\textbf{b}$ be a basis of $D$ and let $U$ denote the matrix of $\varphi$ expressed in $\tau^*\textbf{b}$ and $\textbf{b}$. A different choice of basis $\textbf{b}'$ leads to a matrix $U'$ such that $U=\tau(P)U'P^{-1}$ for a certain invertible matrix $P$ with coefficients in $\cB(F)$. As such, $v_\pi(\det U)$--well-defined as $\det U$ is invertible in $\cB(R)$--is independent of $\textbf{b}$. We denote it by $\deg \underline{D}$ and we name it the \emph{degree of $\underline{D}$}. We define the \emph{slope of $\underline{D}$} to be the rational number $\lambda(\underline{D})=\delta \deg \underline{D}/\operatorname{rk}\underline{D}$, where $\delta$ is the degree of $k$ over $\bF$.

\begin{Remark}
The normalizing factor $\delta$ appearing in the slope was introduced by Mornev \cite{mornev-isocrystal} in order to manage consistently finite fields extensions. 
\end{Remark}

From \cite[Prop. 4.1.1]{mornev-isocrystal}, the category $\operatorname{IC}_{F}$ is abelian. We can therefore consider exact sequences in $\operatorname{IC}_{F}$. The degree and rank are additive in short exact sequences, and the association $\underline{D}\mapsto -\lambda(\underline{D})$ defines a slope function for $\operatorname{IC}_{F}$ in the sense of \cite[Def. 1.3.1]{andre}. The second point of the next definition should be compared with \cite[Def. 1.3.6]{andre}:
\begin{Definition}
Let $\underline{D}=(D,\varphi)$ be an isocrystal over $F$.
\begin{enumerate}
\item A \emph{subisocrystal} of $\underline{D}$ is an isocrystal $\underline{G}=(G,\varphi_G)$ for which $G\subset D$, $\varphi_G=\varphi_D|_{\tau^*G}$. The \emph{quotient} of $\underline{D}$ by $\underline{G}$ is the pair $(D/G,\varphi_D)$ (this is indeed an isocrystal by \cite[Prop. 4.1.1]{mornev-isocrystal}).
\item\label{item:semistable-isoclinic} The isocrystal $\underline{D}$ is \emph{semistable} (resp. \emph{isoclinic}) if, for any nonzero subisocrystal $\underline{D}'$ of $\underline{D}$, $\lambda(\underline{D}')\geq \lambda(\underline{D})$ (resp. $\lambda(\underline{D}')=\lambda(\underline{D})$).
\end{enumerate}
\end{Definition}
Semistability and isoclinicity are related to the notion of \emph{purity}, borrowed from \cite[Def. 3.4.6]{mornev-isocrystal}, that we next recall. We first require the definition of \emph{$\cA(F)$-lattices}:
\begin{Definition}
Let $D$ be a free $\cB(F)$-module of finite rank. An \emph{$\cA(F)$-lattice in $D$} is a sub-$\cA(F)$-module of finite type of $D$ which generates $D$ over $\cB(F)$.
\end{Definition}
Observe that any $\cA(F)$-lattice $L$ in $D$ is free, and that its rank is the rank of $D$ over $\cB(F)$. We denote by $\langle \varphi_D L\rangle$ the sub-$\cA(F)$-module $\varphi_D(\tau^*L)$ in $D$: it is again an $\cA(F)$-lattice in $D$ since $\varphi_D$ is an isomorphism. We define $\langle \varphi_D^n L\rangle$ inductively to be the $\cA(F)$-lattice $\langle \varphi_D \langle \varphi_D^{n-1}L\rangle \rangle$. To include the $n=0$-case, we agree on $\langle \varphi_D^0 L\rangle=L$.
\begin{Definition}\label{def-pure-iso}
A nonzero isocrystal $(D,\varphi_D)$ over $F$ is said to be \emph{pure of slope $\lambda$} if there exist an $\cA(F)$-lattice $L$ in $D$ and integers $s$ and $r>0$ such that $\langle \varphi_D^{r\delta}L \rangle=\fm^s L$ and $\lambda=s/r$. By convention, the zero isocrystal is pure with \emph{no slope}.
\end{Definition}

\begin{Example}
Let $D$ be the free $\cB(F)$-module of rank $r\geq 1$ with basis $\{e_0,\ldots,e_{r-1}\}$ and let $\varphi_D:\tau^*D\to D$ be the unique linear map such that $\varphi_D(\tau^*e_{i-1})=e_i$ for ${1\leq i<r}$ and $\varphi_D(\tau^*e_{r-1})=\pi^se_0$. Then $(D,\varphi_D)$ is a pure isocrystal of slope $\delta s/r$ with $\cA(F)$-lattice given by $\cA(F)e_0\oplus \cdots \oplus \cA(F)e_{s-1}$.
\end{Example}

The following lemma relates the definition of slopes from purity and from slope functions. It also implies that one can refer to \emph{the} slope of a pure isocrystal:
\begin{Lemma}\label{lem:pure-implies-isoclinic}
If $\underline{D}$ is a pure isocrystal of slope $\lambda$, then $\lambda(\underline{D}')=\lambda$ for any nonzero sub-isocrystal $\underline{D}'$ of $\underline{D}$. In particular, $\underline{D}$ is isoclinic (hence semistable).
\end{Lemma}

\begin{proof}
Let $L$ be an $\cA(F)$-lattice in $D$ such that $\langle\varphi^{r\delta} L\rangle=\fm^s L$ for integers $r>0$ and $d$ such that $\lambda=s/r$ (whose existence is ensured by the definition of pureness). If $\underline{D}'=(D',\varphi)$ is a nonzero subisocrystal of $\underline{D}$, then $L'=L\cap D'$ is an $\cA(F)$-lattice in $D'$ such that $\langle\varphi^{r\delta} L'\rangle=\fm^s L'$. As $L'$ is nonzero, let $\{t_1,\ldots,t_{\ell}\}$ be a basis of $L'$ over $\cA(F)$. We have
\begin{equation}\label{determinant-on-lattice}
(\det \varphi)^{r\delta} (t_1\wedge \cdots \wedge t_\ell)=\fm^{s\ell} (t_1\wedge \cdots \wedge t_\ell) \quad \text{in}~\bigwedge^{\ell}L'. \nonumber
\end{equation}
Hence $r\delta \deg \underline{D}'=s\operatorname{rk}\underline{D}'$, which yields 
\[
\lambda(\underline{D}')=\frac{\delta\deg \underline{D}'}{\operatorname{rk} \underline{D}'}=\frac{s}{r}=\lambda.
\]
\end{proof}

\begin{Definition}\label{def:slope-filtration}
A \emph{slope filtration for $\underline{D}$} is an increasing sequence of sub-isocrystals of $\underline{D}$
\begin{equation}
0=\underline{D}_0\subsetneq \underline{D}_1 \subsetneq \cdots \subsetneq \underline{D}_s=\underline{D}, \nonumber 
\end{equation}
satisfying:
\begin{enumerate}[label=$(\roman*)$]
\item\label{item:semistable-isocrystal} for all $i\in\{1,\ldots,s\}$, $\underline{D}_i/\underline{D}_{i-1}$ is semistable,
\item we have $\lambda(\underline{D}_1)<\lambda(\underline{D}_2/\underline{D}_1)<\cdots <\lambda(\underline{D}_s/\underline{D}_{s-1})$.
\end{enumerate}
\end{Definition}

\begin{Theorem}\label{thm:HN-filtration}
Let $\underline{D}$ be an isocrystal over $F$. Then $\underline{D}$ carries a unique slope filtration:
\begin{equation}\label{eq:HN-filtration-equation}
0=\underline{D}_0\subsetneq \underline{D}_1 \subsetneq \cdots \subsetneq \underline{D}_s=\underline{D}.
\end{equation}
In addition, for all $i\in \{1,\ldots,s\}$, the quotients $\underline{D}_{i}/\underline{D}_{i-1}$ are pure isocrystals. In particular, the conditions pure, semi-stable and isoclinic are equivalent. 
\end{Theorem}

\begin{Remark}
The proof presented below relies on \cite[Prop. 1.5.10]{hartl} which already uses Dieudonn\'e-Manin classification (in the case $A=\bF[t]$). It would be much more satisfactory to prove the equivalence between semistability and isoclinicity directly, so that Theorem \ref{thm:HN-filtration} would follow from Andr\'e's theory. 
\end{Remark}

\begin{proof}[Proof of Theorem \ref{thm:HN-filtration}]
The existence and uniqueness of the slope filtration follows from \cite[Thm 1.4.7]{andre} applied to the slope function $\underline{D}\mapsto -\lambda(\underline{D})$ on the abelian category $\operatorname{IC}_{F}$. 

Hence, we only need to prove existence of \eqref{eq:HN-filtration-equation} with pure subquotients since uniqueness follows from \cite[Thm 1.4.7]{andre}. If $\delta=1$, then $\cA(F)$ is identified with $F[\![\pi]\!]$ and Theorem \ref{thm:HN-filtration} is proved in \cite[Prop. 1.5.10]{hartl}. We now explain how the general case reduces to this one. Let $\bG$ be the finite field extension of $\bF$ corresponding to 
\begin{equation}
\bG:=\{f\in \bar{\bF}\cap F~|~f^{q^\delta}=f\}. \nonumber
\end{equation}
Let $\phi:\bG\to F$ denote the inclusion. This defines an embedding of $\bG$ in $k$, the residue field of $E$. Let $\cA_{\phi}(F)$ be the completion of $\cO\otimes_{\bG}F$ at the ideal $\fm\otimes_{\bG} F$. In the theory of isocrystals over $F$ with $\bG$ in place of $\bF$, $\cA_{\phi}(F)$ appears in place of $\cA(F)$ and $\delta=1$. In \cite[\S 4.2]{mornev-isocrystal}, Mornev defines an additive functor
\begin{equation}
[\phi]^*:\left(\cA(F)-\text{isocrystals}\right)\longrightarrow \left(\cA_{\phi}(F)-\text{isocrystals}\right) \nonumber
\end{equation}
which, by \cite[Prop 4.2.2]{mornev-isocrystal} (see also \cite[Prop 8.5]{hartl-bornhofen}), is an equivalence of categories such that $[\phi]^*(\underline{D})$ is a pure isocrystals of slope $\lambda$ if $\underline{D}$ is. Let  
\begin{equation}
[\phi]_*:\left(\cA_{\phi}(F)-\text{isocrystals}\right)\longrightarrow \left(\cA(F)-\text{isocrystals}\right) \nonumber
\end{equation}
be a quasi-inverse of $[\phi]^*$ and let $\ell:[\phi]_*[\phi]^*\stackrel{\sim}{\to}\id$ be a natural transformation.

Let $\underline{D}$ be an $\cA(F)$-isocrystal. We now prove existence of \eqref{eq:HN-filtration-equation} with pure subquotients for $\underline{D}$. By \cite[Prop. 1.5.10]{hartl}, there exists an increasing sequence of sub-$\cA_{\phi}(F)$-isocrystals of $[\phi]^*\underline{D}$:
\begin{equation}
0=\underline{G}_{0}\subsetneq \underline{G}_{1} \subsetneq \underline{G}_{2} \subsetneq \cdots \subsetneq \underline{G}_{s}=[\phi]^*\underline{D} \nonumber
\end{equation}
the subquotients $\underline{G}_{i}/\underline{G}_{i-1}$ being pure of slopes $\lambda_i$ with $\lambda_1<\cdots <\lambda_s$. Applying $[\phi]_*$ and then $\ell$, we obtain
\begin{equation}\label{eq:potential-HN-filtration}
0=\underline{D}_{0}\subsetneq \underline{D}_{1} \subsetneq \underline{D}_{2} \subsetneq \cdots \subsetneq \underline{D}_{s}=\underline{D} 
\end{equation}
with $\underline{D}_{i}:=\ell([\phi]_*[\phi]^*\underline{D}_{i})$ for all $i\in \{0,1,\ldots,s\}$. We claim that the isocrystals $\underline{D}_{i}/\underline{D}_{i-1}$ are pure of slope $\lambda_i$. Indeed, we have
\begin{equation}
\underline{D}_{i}/\underline{D}_{i-1}\cong [\phi]_*\underline{G}_{i}/[\phi]_*\underline{G}_{i-1}\cong [\phi]_*(\underline{G}_{i}/\underline{G}_{i-1}) \nonumber
\end{equation}
where the last isomorphism comes from the fact that $[\phi]_*$ is an exact functor (any equivalence of categories is exact). Because $\underline{G}_{i}/\underline{G}_{i-1}$ is pure of slope $\lambda_i$, $\underline{D}_{i}/\underline{D}_{i-1}$ is also pure of slope $\lambda_i$. We conclude that \eqref{eq:potential-HN-filtration} is the slope filtration for $\underline{D}$ and satisfies the assumption of the theorem.
\end{proof}

Let $\underline{D}$ be an isocrystal over $F$. It is useful to rewrite the slope filtration of $\underline{D}$ as $(\underline{D}_{\lambda_i})_{1\leq i\leq s}$ for rational numbers $\lambda_1<\cdots <\lambda_s$, where the successive quotients $\underline{D}_{\lambda_i}/\underline{D}_{\lambda_{i-1}}$ are pure of slope $\lambda_i$. We let $D_{\lambda_i}$ be the underlying module of $\underline{D}_{\lambda_i}$. For $\lambda\in \bQ$, let $\underline{D}_\lambda$ be the subisocrystal of $\underline{D}$ whose underlying module is
\begin{equation}
D_{\lambda}:=\bigcup_{\lambda_i\leq \lambda} D_{\lambda_i}. \nonumber
\end{equation}
We also let 
\[
\Gr_{\lambda}\underline{D}:=\underline{D}_{\lambda}/\bigcup_{\lambda'<\lambda}\underline{D}_{\lambda'},
\]
the symbol $\cup$ being understood as the isocrystal whose underlying module is given by the union.
\begin{Corollary}\label{cor:strict-slope-filtration}
For all $\lambda\in \bQ$, the assignment $\operatorname{IC}_{F}\to \operatorname{IC}_{F}$, $\underline{D}\mapsto \underline{D}_{\lambda}$ defines an exact functor. Equivalently, any morphism $f:\underline{D}\to \underline{C}$ of isocrystals over $F$ is \emph{strict} with respect to the slope filtration, that is:
\begin{equation}
\text{for~all~}\lambda\in \bQ,\quad f(D_{\lambda})=f(D)\cap C_{\lambda}. \nonumber
\end{equation}
\end{Corollary}
\begin{proof}
This follows at once from Theorem \ref{thm:HN-filtration} and Lemma \ref{lem:pure-implies-isoclinic} that any semistable isocrystal is isoclinic. Hence, the corollary follows from \cite[Thm 1.5.9]{andre}.
\end{proof}

We observe that the slope filtration is not split in general. However it does when the ground field $F$ is perfect: 
\begin{Theorem}\label{thm:Dieudonne-Manin}
If $F$ is perfect, the slope filtration of $\underline{D}$ splits, \emph{i.e.} $\underline{D}$ decomposes along a direct sum
\begin{equation}
\underline{D}\cong \bigoplus_{\lambda\in \bQ}\operatorname{Gr}_{\lambda}\underline{D}. \nonumber
\end{equation}
\end{Theorem}
\begin{proof}
The proof is similar to the argument given for Theorem \ref{thm:HN-filtration}: the corresponding result for $\delta=1$ is proven in \cite[Prop 1.5.10]{hartl} and the general $\delta$-case is easily deduced from \cite[Prop 4.2.2]{mornev-isocrystal}.
\end{proof}
\begin{Remark}
The above theorem is the \emph{Dieudonn\'e-Manin decomposition for isocrystals}. When $F$ is algebraically closed, given $\lambda\in \bQ$ there exists a unique (up to isomorphisms) simple and pure isocrystal $\underline{S}_{\lambda}$ of slope $\lambda$ (see \cite[Prop 4.3.4]{mornev-isocrystal}). Any pure isocrystal of slope $\lambda$ decomposes as a direct sum of $\underline{S}_{\lambda}$ (see \cite[Prop 4.3.7]{mornev-isocrystal}) and together with Theorem \ref{thm:Dieudonne-Manin} yields the \emph{Dieudonn\'e-Manin classification} (see \cite{laumon}). It does not hold for any $F$, even separably closed, as noticed by Mornev in \cite[Rmk 4.3.5]{mornev-isocrystal} (see also Example \ref{ex:mornev}).
\end{Remark}

\subsubsection*{Isocrystals attached to $A$-motives}
We now explain how to attach isocrystals to $A$-motives over fields. This construction (Definition \ref{def:isocrystal-A-motives}) will be required next in the definition of \emph{the weights} of $A$-motives (Definition \ref{def:weights-A-motives}).\\

Let $R$ be an $\bF$-algebra and let $\kappa:A\to R$ be an $\bF$-algebra morphism. We choose the rings $\cA(R)$ and $\cB(R)$ of the previous paragraph in the following way. Given a closed point $x$ on $C$, we let $O_{x}\subset K$ be the associated discrete valuation ring with maximal ideal $\fm_{x}$. We denote by $\cO_{x}$ the completion of $O_{x}$ and by $K_{x}$ the completion of $K$. We let $\bF_{x}$ denote the residue field of $x$. We let $\cA_{x}(R)$ and $\cB_{x}(R)$ be the rings obtained by completing $\cO_{x}\otimes R$ and $K_{x}\otimes R$ for the $\fm_{x}$-adic topology (as in \eqref{eq:A}).\\

Recall that $\fj_{\kappa}$ is the ideal of $A\otimes R$ generated by $\{a\otimes 1-1\otimes \kappa(a)~|~a\in A\}$. 
\begin{Lemma}\label{lem:j-invertible-isocrystals}
We have $\fj_{\kappa}\cB_{\infty}(R)=\cB_{\infty}(R)$. For $x$ a closed point of $C$ distinct from $\infty$ such that $\kappa(\fm_{x})R=R$, we further have $\fj_{\kappa}\cA_{x}(R)=\cA_{x}(R)$.
\end{Lemma}
\begin{proof}
We prove the first assertion. Let $a$ be a non constant element of $A$ so that $a^{-1}\in \fm_{\infty}$. Then $a\otimes 1-1\otimes \kappa(a)\in \fj$ is invertible with $-\sum_{n\geq 0}{a^{-(n+1)} \otimes \kappa(a)^{n}}$ as inverse, where the infinite sum converges in $\cA_{\infty}(R)\subset \cB_{\infty}(R)$.

We prove the second assertion. From $\kappa(\fm_x)R=R$, there exists a relation of the form $1=\sum_i{\kappa(\ell_i)r_i}$ for some $\ell_i\in \fm_x$ and $r_i\in R$. Then, the element
\[
\sum_i{(1\otimes \kappa(\ell_i)-\ell_i\otimes 1)\cdot (1\otimes r_i)}
\]
lies in $1+(\fm_x\otimes R)\subset \cA_x(R)^{\times}$. It also belongs to $\fj_\kappa$, and hence $\fj_\kappa \cA_x(R)=\cA_x(R)$.
\end{proof}

In order to use the results of the previous paragraph, we now assume that $R=F$ is a field and that $x$ is a closed point of $C$ distinct from $\ker \kappa$. Let $\underline{M}=(M,\tau_M)$ be an $A$-motive over $F$.\\

Lemma \ref{lem:j-invertible-isocrystals} yields an isomorphism of $\cB_x(F)$-modules:
\[
\tau_M\otimes_{A\otimes F}\mathbf{1}: \tau^*(M\otimes_{A\otimes F}\cB_{x}(F))\xrightarrow{\sim} M\otimes_{A\otimes F}\cB_{x}(F).
\]
This motivates the following construction:
\begin{Definition}\label{def:isocrystal-A-motives}
Let $\operatorname{I}_{x}(M)$ be the $\cB_{x}(F)$-module $M\otimes_{A\otimes R}\cB_{x}(F)$. We define $\operatorname{I}_{x}(\underline{M})$ as the pair $(\operatorname{I}_{x}(M),\tau_M\otimes \mathbf{1})$.
\end{Definition}

Because of the next proposition, we may call the datum of $\operatorname{I}_{x}(\underline{M})$ the \emph{isocrystal attached to $\underline{M}$}.
\begin{Proposition}\label{prop-fj-invertible-isocrystal}
We have the following: 
\begin{enumerate}[label=$(\roman*)$]
\item \label{item:taudefinedinfty} $\operatorname{I}_{x}(\underline{M})$ is an isocrystal over $F$.
\item \label{item:pureofwerightzero}If $x\neq \infty$, $\operatorname{I}_{x}(\underline{M})$ is pure of slope $0$. 
\end{enumerate}
\end{Proposition}
\begin{proof}
Because $M$ is locally free of constant rank and $\cB_{x}(F)$ is a finite product of fields, $\operatorname{I}_{x}(M)$ is a free $\cB_{x}(F)$-module. Thus, point \ref{item:taudefinedinfty} follows from Lemma \ref{lem:j-invertible-isocrystals}. To prove \ref{item:pureofwerightzero}, it suffices to note that $L=M\otimes_{A\otimes F}\cA_{x}(F)$ is an $\cA_{x}(F)$-lattice in $M\otimes_{A\otimes F}\cB_{x}(F)$ such that $\langle \tau_M L\rangle=L$.
\end{proof}

\begin{Proposition}\label{prop:Ix-exact}
The assignment $\operatorname{I}_{x}:\underline{M}\mapsto \operatorname{I}_{x}(\underline{M})$ defines an exact functor from the category of $A$-motives to the category of isocrystals at $x$. 
\end{Proposition}
\begin{proof}
It suffices to prove that $\cB_x(F)$ is flat over $A\otimes F$. First note that $\cO_x\otimes F$ is Noetherian, so that its completion $\cA_x(F)$ for the $\fm_x$-adic topology is flat over it. Tensoring by $K_x$ over $\cO_x$, we obtain that $\cB_x(F)$ is flat over $K_x\otimes F$. Yet the latter is flat over $A\otimes F$, which concludes.  
\end{proof}

The next lemma clarifies the relation between the isocrystal of an $A$-motive and that of its saturation (Definition \ref{def:saturation} and Lemma \ref{lem:satiso}). It will be important in the next subsection: 
\begin{Lemma}\label{lemma:saturation}
Let $f:\underline{N}\to \underline{M}$ be an isogeny. Then $\operatorname{I}_{x}(f):\operatorname{I}_{x}(\underline{N})\to \operatorname{I}_{x}(\underline{M})$ is an isomorphism. In particular, given $\underline{P}$ a sub-$A$-motive of $\underline{M}$ and $\underline{P}^{\operatorname{sat}}$ its saturation in $\underline{M}$, $\operatorname{I}_{x}(\underline{P})=\operatorname{I}_{x}(\underline{P}^{\operatorname{sat}})$. If $\underline{Q}$ is another sub-$A$-motive of $\underline{M}$ such that $\operatorname{I}_{x}(\underline{P})=\operatorname{I}_{x}(\underline{Q})$ inside $\operatorname{I}_{x}(\underline{M})$, then $\underline{P}^{\operatorname{sat}}=\underline{Q}^{\operatorname{sat}}$.
\end{Lemma}

\begin{proof}
The cokernel of an isogeny is $A$-torsion hence is annihilated under tensoring with $\cB_x(F)$, and the assertions of the second and third sentence of the lemma follow. \\
We prove the statement in the last sentence. Because $\operatorname{I}_x$ is exact, the pullback square 
\[
\begin{tikzcd}
P\cap Q \arrow[r]\arrow[d] & P \arrow[d] \\ 
Q \arrow[r] & M 
\end{tikzcd}
\]
is transformed to a pullback square, and it follows that $\operatorname{I}_x(P\cap Q)=\operatorname{I}_x(P)=\operatorname{I}_x(Q)$. In particular, $P\cap Q$, $P$ and $Q$ are locally free of the same rank. By the definition of saturation (Definition \ref{def:saturation}) $P^{\operatorname{sat}}/(P\cap Q)^{\operatorname{sat}}$ is an $A$-motive over $F$. By the hypotheses it maps to zero under $\operatorname{I}_x$. Because $\operatorname{I}_x$ is rank preserving, this quotient is zero, and by symmetry in $P$ and $Q$ we deduce $P^{\operatorname{sat}}=(P\cap Q)^{\operatorname{sat}}=Q^{\operatorname{sat}}$.
\end{proof}

\subsubsection*{Weight filtration and mixedness}
As before, let $F$ be a field containing $\bF$, let $\kappa:A\to F$ be an $\bF$-algebra morphism and let $\underline{M}$ be an $A$-motive over $F$. Let $\underline{D}:=\operatorname{I}_{\infty}(\underline{M})$ be the isocrystal at the place $\infty$ attached to $\underline{M}$. By Theorem \ref{thm:HN-filtration}, $\underline{D}$ carries a unique slope filtration:
\begin{equation}\label{eq:HN-filtration-Iinf(M)}
0=\underline{D}_0\subsetneq \underline{D}_{1}\subsetneq \cdots \subsetneq \underline{D}_s=\underline{D}
\end{equation}
with ascending slopes $\lambda_1<\cdots <\lambda_s$, where $\lambda_i:=\lambda(\underline{D}_i/\underline{D}_{i-1})$.
\begin{Definition}\label{def:weights-A-motives}
We call the set $w(\underline{M}):=\{-\lambda_i~|~1\leq i \leq s\}$ the \emph{set of weights} of $\underline{M}$.
\end{Definition}

The next definition dates back to the seminal paper of Anderson \cite[1.9]{anderson}.
\begin{Definition}\label{def:pure-A-motives}
We call $\underline{M}$ \emph{pure of weight $\mu$} if $\underline{D}$ is pure of slope $-\mu$. Equivalently, $\underline{M}$ is pure of weight $\mu$ if $w(\underline{M})=\{\mu\}$.
\end{Definition}

The sign convention--weights opposed to slopes--is made to fit with the number field situation:
\begin{Example}
Using notations of Example \ref{ex:carlitz-motive}, the Carlitz twist $\underline{A}(1)$ over $F$ is pure of weight $-1$ and, more generally, $\underline{A}(n)$ is pure of weight $-n$. It is analogous to the number field case, where the motive $\mathbb{Z}(n)$ is pure of weight $-2n$ (the factor $2$, reflecting the degree $[\mathbb{C}:\mathbb{R}]$, could be removed by renormalizing the weight filtration using half-integral integers).
\end{Example}

In analogy with number fields, we define mixedness for $A$-motives as follows:
\begin{Definition}\label{def:mixed-A-motives}
We call $\underline{M}$ \emph{mixed} if there exist rational numbers $\mu_1<\ldots <\mu_s$ and an increasing finite filtration by saturated sub-$A$-motives of $\underline{M}$:
\begin{equation}\label{eq:weight-filtration}
0=W_{\mu_0}\underline{M}\subsetneq W_{\mu_1}\underline{M}\subsetneq \cdots \subsetneq W_{\mu_s}\underline{M}=\underline{M},
\end{equation}
for which the successive quotients $W_{\mu_i}\underline{M}/W_{\mu_{i-1}}\underline{M}$ are pure of weight $\mu_i$.
\end{Definition}

Before pursuing on the properties of mixedness, let us explain why being mixed is a very restrictive condition over imperfect base fields. Suppose that $\underline{M}$ is a mixed $A$-motive and consider a filtration as in \eqref{eq:weight-filtration}. The functor $\operatorname{I}_{\infty}$ is exact (Proposition \ref{prop:Ix-exact}), and applying it to \eqref{eq:weight-filtration} yields a finite filtration of $\underline{D}=\operatorname{I}_{\infty}(\underline{M})$ by subisocrystals:
\begin{equation}\label{eq:reversed-filtration}
0=\underline{D}^0\subsetneq \underline{D}^{1}\subsetneq \cdots \subsetneq \underline{D}^s=\underline{D}
\end{equation}
whose successive quotients $\underline{D}^i/\underline{D}^{i-1}$ are pure isocrystals of slope $-\mu_i$. Note that the slopes of this filtration are decreasing, hence \eqref{eq:reversed-filtration} is \emph{not} the slope filtration of $\underline{D}$. 
\begin{Proposition}\label{prop:HN-filtration-splits}
If $\underline{M}$ is mixed, the slope filtration of $\operatorname{I}_{\infty}(\underline{M})$ splits.
\end{Proposition}

To prove \ref{prop:HN-filtration-splits}, we require a lemma on isocrystals:
\begin{Lemma}\label{lem:split-sequence}
Let $S:0\to \underline{D}\to \underline{D}'\to \underline{D}''\to 0$ be an exact sequence of isocrystals, where $\underline{D}$ and $\underline{D}''$ are pure and $\lambda(\underline{D})>\lambda(\underline{D}'')$. Then $S$ is split.
\end{Lemma}
\begin{proof}
Because $\Ext^1_{\operatorname{IC}_{F}}(\underline{D}'',\underline{D})\cong \Ext^1_{\operatorname{IC}_{F}}(\mathbbm{1},\underline{D}\otimes (\underline{D}'')^{\vee})$ by Corollary \ref{cor:dualizable-tensor}, the argument reduces to the case of $\underline{D}''=\mathbbm{1}$ and $\lambda(\underline{D})>0$ (\emph{e.g.} \cite[Prop. 3.4.10]{mornev-isocrystal}). Following the proof of Proposition \ref{prop-cohomology-in-tilde}, one sees that $\Ext^1_{\operatorname{IC}_{F}}(\mathbbm{1},\underline{D})$ is isomorphic to $D/\im(\id-\varphi)$ as a $K$-vector space, where by $\id-\varphi$ we mean the $K$-linear endomorphism of $D$ given by $d\mapsto d-\varphi(\tau^*d)$. In particular, to show that any sequence $S$ splits, it is enough to show that $\id–\varphi$ is surjective on $D$. As $\underline{D}$ is pure of positive slope, $D$ contains an $\cA(F)$-lattice $L$ satisfying $\langle \varphi^h L\rangle\subset \fm L$ for some $h>0$. This implies that the series 
\[
\xi_d:=d+\varphi(\tau^*d)+\varphi(\tau^* \varphi( \tau^* d))+\ldots 
\]
converges in $D$ for all $d\in D$, and the assignment $d\mapsto \xi_d$ defines an inverse of $\id-\varphi$.
\end{proof}

\begin{proof}[Proof of Proposition \ref{prop:HN-filtration-splits}]
Let $\underline{D}'_1,\ldots, \underline{D}'_s$ be pure isocrystals and let $\underline{D}''$ be a pure isocrystal such that $\lambda(\underline{D}'')>\lambda(\underline{D}'_i)$ for all $i$. Then, by Lemma \ref{lem:split-sequence}, 
\[
\Ext^{1}_{\operatorname{IC}_{F}}\left(\underline{D}'',\bigoplus_i \underline{D}_i'\right)\cong \bigoplus_i \Ext^{1}_{\operatorname{IC}_{F}}\left(\underline{D}'',\underline{D}_i'\right)=(0).
\]
Using the above observation, by induction one shows that $\operatorname{I}_{\infty}(W_{\mu_i})$ splits for all $i$, and for $i=s$ we obtain a splitting of the filtration \eqref{eq:reversed-filtration}. By reordering the pure parts, we build another split filtration of $\underline{D}$:
\[
0\subsetneq \underline{D}^s/\underline{D}^{s-1}\subsetneq  (\underline{D}^s/\underline{D}^{s-1})\oplus (\underline{D}^{s-1}/\underline{D}^{s-2})\subsetneq \cdots \subsetneq \bigoplus_{i=1}^s \underline{D}^i/\underline{D}^{i-1}
\]
having increasing slopes. By uniqueness, it coincides with the slope filtration of $\underline{D}$.
\end{proof}

Proposition \ref{prop:HN-filtration-splits} allows one to construct non mixed $A$-motives when $F$ is imperfect, by constructing an $A$-motive whose associated isocrystal is non split. One observes that by Theorem \ref{thm:Dieudonne-Manin} such examples can only exist if $F$ is imperfect. Example \ref{ex:mornev} below show that such examples exist. 

\begin{Example}\label{ex:mornev}
We suppose that $A=\bF[t]$ and denote by $\pi$ the uniformizer $t^{-1}$ in $\cB_{\infty}(F)$, then identified with $F(\!(\pi)\!)$. Assume that $F$ is imperfect, and let $\alpha\in F$ be such that $-\alpha$ does not have any $q$th root in $F$. Consider the $t$-motive $\underline{M}$ whose underlying module is $F[t]e_0\oplus F[t]e_1$, and where $\tau_M$ acts by $\tau_M(\tau^*e_0)=e_0$ and $\tau_M(\tau^*e_1)=\alpha e_0+(t-\theta)^{-1}e_1$. $\underline{M}$ inserts in a short exact sequence of $t$-motives:
\begin{equation}\label{eq:exact-sequence-wrong-sided-weights}
0\to \mathbbm{1}\cdot e_0 \to \underline{M} \to \underline{A}(1)\cdot e_1 \to 0.
\end{equation}
We prove that the slope filtration of $\operatorname{I}_{\infty}(\underline{M})$ does not split. More precisely, let us show that $\operatorname{I}_{\infty}(\mathbbm{1})\cdot e_0$ is the only non-zero strict subisocrystal of $\operatorname{I}_{\infty}(\underline{M})$. 

Let $\underline{L}\neq 0$ be a strict subisocrystal of $\operatorname{I}_{\infty}(\underline{M})$. $\underline{L}$ must have rank $1$ over $F(\!(\pi)\!)$, and we let $\ell_0 e_0+\ell_1e_1$ be a generator of its underlying module. Let $f\in F(\!(\pi)\!)$ be a non zero element such that $\tau_M(\tau^*(\ell_0e_0+\ell_1e_1))=f(\ell_0e_0+\ell_1e_1)$. This yields the equation:
\begin{equation}
\begin{pmatrix}
1 & \alpha \\ 0 & \pi (1-\theta \pi )^{-1}
\end{pmatrix}\begin{pmatrix}
\tau(\ell_0) \\
\tau(\ell_1)
\end{pmatrix}=f
\begin{pmatrix}
\ell_0 \\
\ell_1
\end{pmatrix}. \nonumber
\end{equation}
Because $\underline{L}\neq 0$, the first row implies $\ell_0\neq 0$.  If $\ell_1\neq 0$, as $(1-\theta \pi)$ is a unit in $F[\![\pi]\!]$ the bottom row imposes $v_\pi (f)=1$, whereas the first row reads
\begin{equation}\label{first-row-equation}
\tau(\ell_0)+\alpha \tau(\ell_1)=f\ell_0. 
\end{equation}
Because $v_\pi (f)=1$, \eqref{first-row-equation} also implies that $v_\pi (\ell_0)=v_\pi (\ell_1)$. In particular, if $l_0$ and $l_1$ are the first nonzero coefficients in $F$ of $\ell_0$ and $\ell_1$ in $F(\!(\pi)\!)$ respectively, \eqref{first-row-equation} gives
\begin{equation}
\alpha =-(l_0/l_1)^q. \nonumber
\end{equation}
This is in contradiction with our assumption. Hence $\ell_1=0$ and $f=\tau(\ell_0)/\ell_0$. Therefore, $\underline{L}$ coincides with $\operatorname{I}_{\infty}(\mathbbm{1})\cdot e_0$.

In this example, $\underline{M}$ is an extension of $\underline{A}(1)$ by $\mathbbm{1}$, two $t$-motives of respective weights $-1<0$. When the weights goes in ascending orders, such examples do not appear anymore (see Proposition \ref{prop:separated-weights-ext}).
\end{Example}

\begin{Remark}
Observe that the converse of Proposition \ref{prop:HN-filtration-splits} does not hold: there are non-mixed $A$-motives over perfect base fields (\emph{e.g.} \cite[Ex. 2.3.13]{hartl-juschka}). For a partial converse, we refer to Proposition \ref{prop:criterion-for-mixedness} below.
\end{Remark}

\begin{Proposition-Definition}\label{def:weight-filtration}
If $\underline{M}$ is mixed, a filtration $W=(W_{\mu_i}\underline{M})_{1\leq i \leq s}$ as in \eqref{eq:weight-filtration} is unique. In addition, the set $\{\mu_1,\ldots,\mu_s\}$ equals $w(\underline{M})$.
\begin{enumerate}[label=$(\roman*)$]
\item For all $i\in\{1,\ldots ,s\}$, we let $W_{\mu_i}M$ be the underlying module of $W_{\mu_i}\underline{M}$. 
\item\label{item:def:graded-weight-filtration} For all $\mu\in \bQ$, we set 
\begin{align*}
W_{\mu}M &:=\bigcup_{\mu_i\leq \mu}{W_{\mu_i}M}, \quad W_{\mu}\underline{M}:=(W_{\mu}M,\tau_M), \\
W_{<\mu}M &:=\bigcup_{\mu_i< \mu}{W_{\mu_i}M}, \quad W_{<\mu}\underline{M}:=(W_{<\mu}M,\tau_M), 
\end{align*}
and $\Gr_{\mu}\underline{M}:=W_\mu\underline{M}/W_{<\mu}\underline{M}$. Both $\Gr_\mu \underline{M}$ and $W_{\mu}\underline{M}$, as well as $W_{<\mu}\underline{M}$, define mixed $A$-motives over $F$ for all $\mu\in \bQ$.
\item\label{item:def:weight-filtration} We call $(W_{\mu}\underline{M})_{\mu\in \bQ}$ the \emph{weight filtration of $\underline{M}$}.
\item We let $\cM\cM_F$ (resp. $\cM\cM^{\text{iso}}_F$) be the subcategory of $\cM_F$ (resp. $\cM^{\text{iso}}_F$) whose objects are mixed, and whose morphisms $f:\underline{M}\to \underline{N}$ preserves the weight filtration:
\[
\text{for~all~} \mu\in \bQ: \quad f(W_{\mu}M)\subset W_\mu N.
\]
\end{enumerate}
\end{Proposition-Definition}
\begin{Remark}
We shall prove below (Corollary \ref{cor:Wmu-functor}) that every morphism of $A$-motives $f:\underline{M}\to \underline{N}$, $\underline{M}$ and $\underline{N}$ being mixed, preserves the weight filtration. In other words, that $\cM\cM_F$ is a full subcategory of~$\cM_F$.
\end{Remark}

\begin{proof}[Proof of Proposition \ref{def:weight-filtration}]
Let $\underline{D}=\operatorname{I}_{\infty}(\underline{M})$. By Proposition \ref{prop:HN-filtration-splits}, we have a canonical decomposition:
\begin{equation}\label{eq:HN-splits-canonical}
\underline{D}=\bigoplus_{i=1}^s \Gr_{\lambda_i}\underline{D},
\end{equation}
$\Gr_{\lambda_i}\underline{D}$ being canonically identified with a pure subisocrystal of $\underline{D}$ of slope $\lambda_i$. On the other-hand, let $W$ be as in \eqref{eq:weight-filtration}. By uniqueness of the decomposition \eqref{eq:HN-splits-canonical}, we have equalities of subisocrystals of $\underline{D}$:
\begin{equation}\label{eq:isocrys-of-weight}
\forall \mu\in \bQ:\quad \operatorname{I}_{\infty}(W_{\mu}\underline{M})=\bigoplus_{\lambda_i\geq -\mu}\Gr_{\lambda_i}\underline{D}.
\end{equation}
We conclude by Lemma \ref{lemma:saturation} that the above identity determines $W$ uniquely. The fact that $w(\underline{M})=\{\mu_1,\ldots,\mu_s\}$ also follows.
\end{proof}

Next, we suggest a criterion for mixedness:
\begin{Proposition}\label{prop:criterion-for-mixedness}
Suppose that the slope filtration of $\underline{D}$ splits. For all $\mu\in \bQ$, let $\operatorname{I}_{\infty}(\underline{M})^{\mu}$ be the subisocrystal of $\underline{D}$ corresponding to $\bigoplus_{\lambda_i\geq -\mu}\Gr_{\lambda_i}\underline{D}$. Denote by $\operatorname{I}_{\infty}(M)^{\mu}$ its underlying module. Then, we have:
\[
\text{for~all~} \mu\in \bQ:\quad \rank_{A\otimes F} (\operatorname{I}_{\infty}(M)^{\mu}\cap M) \leq \rank_{\cB(F)} \operatorname{I}_{\infty}(\underline{M})^{\mu}
\]
with equality if and only if $\underline{M}$ is mixed. In the latter case, $W_{\mu}M=\operatorname{I}_{\infty}(M)^{\mu}\cap M$.
\end{Proposition}
\begin{proof}
First note that, for all $\mu$, the couple $\underline{M}_{\mu}:=(\operatorname{I}_{\infty}(M)^{\mu}\cap M,\tau_M)$ defines a saturated sub-$A$-motive of $\underline{M}$. Furthermore, since the underlying module of $\operatorname{I}_{\infty}(\underline{M}_{\mu})$  corresponds to the completion of  $\operatorname{I}_{\infty}(M)^{\mu}\cap (M\otimes_A K)$ for the $\infty$-adic topology, $\operatorname{I}_{\infty}(\underline{M}_{\mu})$ defines a subisocrystal of $\operatorname{I}_{\infty}(\underline{M})^{\mu}$. The rank being additive in short exact sequences in the category of isocrystals, we get the desired inequality:
\[
\operatorname{rank}_{A\otimes F}(\operatorname{I}_{\infty}(M)^{\mu}\cap M)= \operatorname{rank}\underline{M}_{\mu}=\operatorname{rank} \operatorname{I}_{\infty}(\underline{M}_{\mu})\leq \operatorname{rank} \operatorname{I}_{\infty}(\underline{M})^{\mu}.
\]
If this is an equality for all $\mu$, then we obtain $\operatorname{I}_{\infty}(\underline{M}_\mu)=\operatorname{I}_{\infty}(\underline{M})^\mu$ and deduce that the family $(\underline{M}_\mu)_\mu$ satisfies the requested property of Definition \ref{def:mixed-A-motives}.

Conversely, if $\underline{M}$ is mixed, the following sequence of inclusions holds:
\[
(W_{\mu}M)\otimes_A K\subset \operatorname{I}_{\infty}(M)^{\mu}\cap (M\otimes_A K)\subset \operatorname{I}_{\infty}(M)^{\mu}.
\]
Note that the left-hand side forms a dense subset of the right-hand side for the $\infty$-adic topology. Hence, the first inclusion is an inclusion of a dense subset. Taking the completion, we obtain $\operatorname{I}_{\infty}(W_{\mu}M)=\operatorname{I}_{\infty}(\operatorname{I}_{\infty}(M)^{\mu}\cap M)=\operatorname{I}_{\infty}(\underline{M})^{\mu}$. The corresponding ranks are thus equal. 

To conclude, as both $W_{\mu}M$ and $\operatorname{I}_{\infty}(M)_{\mu}\cap M$ are saturated submodules of $M$, we deduce the equality $W_{\mu}M=\operatorname{I}_{\infty}(M)^{\mu}\cap M$ from Lemma \ref{lemma:saturation}.
\end{proof}

As an important corollary of the above formula for $W$, we obtain:
\begin{Corollary}\label{cor:Wmu-functor}
Let $f:\underline{M}\to \underline{N}$ be a morphism of $A$-motives, $\underline{M}$ and $\underline{N}$ being mixed. Then, $f$ preserves the weight filtration:
\[
\text{for~all~} \mu\in \bQ:\quad f(W_\mu M)\subset W_\mu N.
\]
In particular, for all $\mu\in \bQ$, the assignment $\underline{M}\mapsto W_\mu \underline{M}$ defines a functor from the category $\cM\cM_F$ to itself. 
\end{Corollary}
\begin{proof}
Let $\mu\in \bQ$. The claim follows from the following two easy observations: 
\begin{enumerate}
\item If $\underline{M}$ is mixed, then so is $W_{\mu} \underline{M}$,
\item If $f:\underline{M}\to \underline{N}$ is a morphism of mixed $A$-motives, then the associated morphism of isocrystals maps $\operatorname{I}_{\infty}(M)^{\mu}$ to $\operatorname{I}_{\infty}(N)^\mu$, and by Proposition \ref{prop:criterion-for-mixedness} we have ${f(W_{\mu}M)\subset W_{\mu} N}$.
\end{enumerate}
\end{proof}

\begin{Remark}
We end this subsection by describing how weights behave under linear algebra type operations. Proofs are presented in \cite[Prop. 2.3.11]{hartl-juschka} and extend without change to our larger setting. First note that $\mathbbm{1}$ is a pure $A$-motive over $F$ of weight $0$. Given two mixed $A$-motives $\underline{M}$ and $\underline{N}$, their biproduct $\underline{M}\oplus \underline{N}$ is again mixed with weight filtration $W_{\mu}(\underline{M}\oplus \underline{N})=W_{\mu}\underline{M}\oplus W_{\mu}\underline{N}$ ($\mu\in \bQ$). Their tensor product $\underline{M}\otimes \underline{N}$ is also mixed, with $\lambda$-part of its weight filtration being:
\begin{equation}
W_\lambda(\underline{M}\otimes \underline{N})=\left(\sum_{\mu+\nu=\lambda}{W_{\mu}\underline{M}\otimes W_{\nu}\underline{N}}\right)^{\text{sat}}. \nonumber
\end{equation} 
We took the saturation $A$-motive to ensure that the above is a saturated sub-$A$-motive of $\underline{M}\otimes \underline{N}$. The dual $\underline{M}^{\vee}$ is mixed, and the $\mu$-part of its weight filtration $W_{\mu}\underline{M}$ has for underlying module $W_{\mu}M^{\vee}=\{m\in M^{\vee}|\forall \lambda<-\mu:~m(W_{\lambda}M)=0\}^{\text{sat}}$. In general, given $\underline{M}$ and $\underline{N}$ two $A$-motives over $F$ (without regarding whether $\underline{M}$ or $\underline{N}$ are mixed) and an exact sequence $0\to \underline{M}'\to \underline{M}\to \underline{M}''\to 0$ in $\cM_F$, we have 
\begin{equation}
\begin{tabular}{lcl}
$w(\underline{0})$ & $=$ & $\emptyset$ \\
$w(\underline{M}^{\vee})$ & $=$ & $-w(\underline{M})$ \\
$w(\underline{M}\oplus \underline{N})$ & $=$ & $w(\underline{M})\cup w(\underline{N})$ \\
$w(\underline{M})$ & $=$ & $w(\underline{M}')\cup w(\underline{M}'')$ \\
$w(\underline{M}\otimes \underline{N})$ & $=$ & $\{w+v~|~w\in w(\underline{M}),~v\in w(\underline{N})\}$
\end{tabular}\nonumber
\end{equation}
\end{Remark}

\subsection{Extension modules of mixed $A$-motives}\label{subsec:extension-modules-of-mixed}
As before, let $F$ be a field containing $\bF$ and consider an $\bF$-algebra morphism $\kappa:A\to F$. In this subsection, we are concerned with extension modules in the category $\cM\cM_F$. The next proposition shows that they are well-defined.
\begin{Proposition}\label{prop:MMF-is-exact}
The category $\cM\cM_F$ together with the notion of exact sequences of Definition \ref{def:exact-sequence} is an exact category.
\end{Proposition}

To prove Proposition \ref{prop:MMF-is-exact}, we use Proposition \ref{prop:Functor-to-exactCat} of the appendix applied once again to the forgetful functor $\cM\cM_F\to \mathbf{Mod}_{A\otimes R}$ as in proof of Proposition \ref{prop:MF-is-exact}. Checking the required hypothesis of Proposition \ref{prop:Functor-to-exactCat} amounts to showing:
\begin{Proposition}\label{prop:exact-sequence-middl-mixedness}
Let $0\to \underline{M}'\to \underline{M}\to \underline{M}''\to 0$ be an exact sequence of $A$-motives over $F$. If $\underline{M}$ is mixed and one of $\underline{M}'$ or $\underline{M}''$ is mixed, then all three are mixed. 
\end{Proposition}

Before proving \eqref{prop:exact-sequence-middl-mixedness}, we begin by an intermediate result on isocrystals:
\begin{Lemma}\label{lem:HNfiltration-split}
Let $0\to \underline{D}'\to \underline{D}\to \underline{D}''\to 0$ be an exact sequence of isocrystals over $F$. If the slope filtration of $\underline{D}$ and one of $\underline{D}'$ or $\underline{D}''$ split, then all three are split. 
\end{Lemma}
\begin{proof}
Assume that the slope filtration of $\underline{D}$ and $\underline{D}'$ are split. Then $\underline{D}$ decomposes as a finite direct sum of pure subisocrystals $\Gr_\lambda \underline{D}$. We have a commutative diagram of subisocrystals of $\underline{D}$:
\begin{equation}\label{eq:square-subisocrystals-of-D}
\begin{tikzcd}
\underline{D} & \displaystyle\bigoplus_{\lambda \in \bQ} \Gr_\lambda\underline{D} \arrow[l,"\sim"'] \\
\underline{D}' \arrow[u,hook] &  \displaystyle\bigoplus_{\lambda \in \bQ}(\underline{D}'\cap \Gr_\lambda \underline{D}) \arrow[l,"\iota"']\arrow[u,hook]
\end{tikzcd}
\end{equation}
where the symbol $\cap$ is understood as the subisocrystal whose underlying module is given by the intersection. We claim that $\iota$ is an isomorphism. Firstly, it is injective as one checks by a diagram chase. By assumption, $\underline{D}'$ also decomposes as a direct sum of subisocrystals $\Gr_\lambda\underline{D}'$ which, when non zero, are pure of slope $\lambda$. By functoriality of the slope filtration, the inclusion $\Gr_\lambda \underline{D}'\hookrightarrow \underline{D}'\hookrightarrow \underline{D}$ factors through $\Gr_\lambda \underline{D}$. Hence,
\begin{equation}
\operatorname{rk} \underline{D}'= \sum_{\lambda\in \bQ}{\operatorname{rk} \Gr_\lambda \underline{D}'}\leq \sum_{\lambda\in \bQ}{\operatorname{rk}(\underline{D}'\cap \Gr_\lambda \underline{D})}\leq \operatorname{rk} \underline{D}' \nonumber
\end{equation}
where the last inequality follows from the injectivity of $\iota$. All the above inequalities are thus equalities, from which we deduce that $\iota$ is an isomorphism.

Therefore, we can extend the square \eqref{eq:square-subisocrystals-of-D} into exact sequences:
\begin{equation}
\begin{tikzcd}[column sep=1em]
0 \arrow[r] & \displaystyle\bigoplus_{\lambda \in \bQ}(\underline{D}'\cap \Gr_\lambda \underline{D}) \arrow[d,"\iota"] \arrow[r] & \displaystyle\bigoplus_{\lambda\in \bQ} \Gr_\lambda\underline{D} \arrow[d,"\wr"] \arrow[r] & \displaystyle\bigoplus_{\lambda\in \bQ}\Gr_\lambda \underline{D}/(\underline{D}'\cap \Gr_\lambda \underline{D}) \arrow[r]\arrow[d,dashed] & 0 \\
0 \arrow[r] & \underline{D}' \arrow[r] & \underline{D} \arrow[r] & \underline{D}'' \arrow[r] & 0
\end{tikzcd} \nonumber
\end{equation}
where the dashed arrow exists and is unique by the cokernel property of the upper right term. By the Snake Lemma, it is an isomorphism. In particular, $\underline{D}''$ decomposes as a direct sum of pure subisocrystals, and by uniqueness, we deduce that its slope filtration splits. 

The argument for the second part of the statement is similar enough to the first one to be skipped. 
\end{proof}

\begin{proof}[Proof of Proposition \ref{prop:exact-sequence-middl-mixedness}]
Suppose $\underline{M}$ and $\underline{M}'$ (resp. $\underline{M}''$) are mixed. For $\mu\in \bQ$, let $W_{\mu}M':=W_{\mu}M\cap M'$ and $W_{\mu}M'':=f(W_{\mu}M)^{\text{sat}}$, where $f$ is the epimorphism $\underline{M}\twoheadrightarrow \underline{M}''$. They are canonically endowed with $A$-motive structures denoted by $W_{\mu}\underline{M}'$ and $W_{\mu}\underline{M}''$ respectively, and our task is to prove that $(W_{\mu}\underline{M}')_{\mu}$ and $(W_{\mu}\underline{M}'')_{\mu}$  satisfy the property of Definition \ref{eq:weight-filtration}.

By Lemma \ref{lem:HNfiltration-split}, the slope filtrations of $\operatorname{I}_{\infty}(\underline{N})$ for $\underline{N}\in\{\underline{M},\underline{M}',\underline{M}''\}$ are split, and, for $\mu\in \bQ$, we denote by $\operatorname{I}_{\infty}(\underline{N})^\mu$ the direct sum of the subisocrystals of $\operatorname{I}_{\infty}(\underline{N})$ which are pure of slope $\geq -\mu$.

Since $\cB_{\infty}(F)$ is flat over $A\otimes F$, the functor $-\otimes_{A\otimes F}\cB_{\infty}(F)$ commutes with finite intersections (\emph{e.g.} \cite[\S I.2, Prop. 6]{bourbaki}): 
\begin{equation}
\operatorname{I}_{\infty}(W_{\mu}M')=(W_\mu M\cap M')\otimes_{A\otimes F}\cB_{\infty}(F) =\operatorname{I}_{\infty}(M)^{\mu}\cap \operatorname{I}_{\infty}(M')=\operatorname{I}_{\infty}(M')^{\mu} \nonumber
\end{equation}
where the last equality follows from the fact that the category of isocrystals has strict morphisms for the slope filtration (Corollary \ref{cor:strict-slope-filtration}). Similarly,
\begin{align*}
\operatorname{I}_{\infty}(W_{\mu}M'') &=f(W_{\mu}M)^{\text{sat}}\otimes_{A\otimes F}\cB_{\infty}(F)=f(\operatorname{I}_{\infty}(M)^{\mu})=\operatorname{I}_{\infty}(M'')^{\mu}. 
\end{align*}
This shows that $\underline{M}'$ and $\underline{M}''$ are both mixed with respective weight filtrations $(W_{\mu}\underline{M}')_{\mu\in \bQ}$ and $(W_{\mu}\underline{M}'')_{\mu\in \bQ}$.
\end{proof}

\begin{Corollary}\label{cor:W-exact}
For all $\mu\in \bQ$, the functor $\underline{M}\mapsto W_\mu \underline{M}$ is exact on $\cM\cM_F^{\operatorname{iso}}$.
\end{Corollary}
\begin{proof}
Let $0\to \underline{M}'\to \underline{M}\stackrel{f}{\to} \underline{M}''\to 0$ be an exact sequence of mixed $A$-motives. By the proof of Proposition \ref{prop:exact-sequence-middl-mixedness}, we have $W_{\mu}M'=W_{\mu}M\cap M'$. Hence, the sequence $0\to W_{\mu}M'\to W_\mu M\to W_\mu M''\to 0$ is left-exact. We also have $W_\mu M''=f(W_{\mu}\underline{M})^{\text{sat}}$, so that the $A$-motive $\underline{\im}(f|W_\mu \underline{M})$ is isogeneous to $W_\mu \underline{M}''$. In particular, the sequence $0\to W_{\mu}\underline{M}'\to W_\mu \underline{M}\to W_\mu \underline{M}''\to 0$ is exact in $\cM\cM_F^{\text{iso}}$.
\end{proof}

According to Proposition \ref{prop:MMF-is-exact}, we can consider extension modules of two mixed $A$-motives in $\cM\cM_F$. By Corollary \ref{cor:Wmu-functor}, we have an equality
\[
\Ext^0_{\cM\cM_F}(\underline{M},\underline{N})=\Ext^0_{\cM_F}(\underline{M},\underline{N}),
\]
but this does not necessary hold for higher extension modules. $\Ext^1_{\cM\cM_F}(\underline{M},\underline{N})$ can be interpreted as a submodule of $\Ext^1_{\cM_F}(\underline{M},\underline{N})$, but in general $\Ext^i_{\cM\cM_F}$ is not even a submodule of $\Ext^i_{\cM_F}$ ($i>1$).

\begin{Proposition}\label{prop:separated-weights-ext}
Let  $0\to \underline{M}'\to \underline{M}\stackrel{p}{\to} \underline{M}''\to 0$ be an exact sequence of $A$-motives in $\cM_F$. If $\underline{M}'$ and $\underline{M}''$ are mixed, and if the smallest weight of $\underline{M}''$ is bigger than the largest weight of $\underline{M}'$, then $\underline{M}$ is mixed.  
\end{Proposition}
\begin{proof}
By Proposition \ref{prop-cohomology-in-tilde}, there exists $u\in \Hom_{A\otimes F}(\tau^*M'',M')[\fj^{-1}]$ and a diagram
\begin{equation}
\begin{tikzcd}
0\arrow[r] & \underline{M}'  \arrow[r,"\iota"] &   \underline{M}  \arrow[r,"p"] & \underline{M}''  \arrow[r] & 0 \\
0\arrow[r] & \underline{M}' \arrow[u,"\id"] \arrow[r] & (M'\oplus M'', \left(\begin{smallmatrix} \tau_{M'} & u \\ 0 & \tau_{M''} \end{smallmatrix}\right)) \arrow[u,"\xi"] \arrow[r] & \underline{M}'' \arrow[u,"\id"] \arrow[r] & 0
\end{tikzcd}
\nonumber
\end{equation}
which commutes in $\cM_F$. From the the weight assumption on $\underline{M}'$ and $\underline{M}''$, $u$ automatically respects the weight filtration: for all $\mu\in \bQ$, $u(\tau^*W_\mu M'')\subset W_\mu M'$. In particular, 
\[
\underline{M}_{\mu}:=\xi\left(W_{\mu}M'\oplus W_{\mu}M'',\begin{pmatrix}
\tau_{M'} & u \\ 0 & \tau_{M''}
\end{pmatrix} \right)
\]
defines a sub-$A$-motive of $\underline{M}$ inserting in a short exact sequence in $\cM_F$:
\begin{equation}\label{eq:exact-sequence-potential-weight}
0\longrightarrow W_{\mu}\underline{M}'\longrightarrow \underline{M}_{\mu}\longrightarrow W_{\mu}\underline{M}''\longrightarrow 0.
\end{equation}
Similarly, $\underline{M}_{<\mu}:=\xi\left(W_{<\mu}M'\oplus W_{<\mu}M'',\left(\begin{smallmatrix} \tau_{M'} & u \\ 0 & \tau_{M''} \end{smallmatrix}\right) \right)$ defines a sub-$A$-motive of $\underline{M}_{\mu}$. By \eqref{eq:exact-sequence-potential-weight}, we obtain an exact sequence of $A$-motives:
\begin{equation}\label{eq:exact-sequence-pure-weight}
0\longrightarrow \Gr_{\mu}\underline{M}'\longrightarrow \underline{M}_{\mu}/\underline{M}_{<\mu}\longrightarrow \Gr_{\mu}\underline{M}''\longrightarrow 0.
\end{equation}
The extremal terms of \eqref{eq:exact-sequence-pure-weight} are pure of weight $\mu$ and at least one of them is zero,
and thus the middle term is pure of weight $\mu$. We deduce that the increasing sequence $(\underline{M}_{\mu})_\mu$ of sub-$A$-motives of $\underline{M}$ satisfies the condition of Definition \ref{def:mixed-A-motives}. Therefore $\underline{M}$ is mixed.
\end{proof}

\begin{Remark}
Contrary to the number fields situation, the full subcategory of $\cM\cM_F$ consisting of pure $A$-motives over $F$ is not semi-simple. This follows easily from the equality $\Ext^1_{\cM\cM_F}(\underline{N},\underline{M})=\Ext^1_{\cM_F}(\underline{N},\underline{M})$ for two pure motives of the same weight.
\end{Remark}

Proposition \ref{prop:separated-weights-ext} implies that $\Ext^1_{\cM_F}(\underline{M}'',\underline{M}')=\Ext^1_{\cM\cM_F}(\underline{M}'',\underline{M}')$ whenever the weights of $\underline{M}''$ are bigger than the biggest weight of $\underline{M}'$. In general this is not true. In this direction we record:
\begin{Proposition}\label{prop-torsion-in-Ext-positive-weights}
Let $0\to \underline{M}'\to \underline{M}\to \underline{M}''\to 0$ be an exact sequence of $A$-motives where $\underline{M}'$ and $\underline{M}''$ are mixed. We assume that all the weights of $\underline{M}''$ are strictly smaller than the smallest weight of $\underline{M}'$. Then, the sequence is torsion in $\Ext^1_{\cM_F}(\underline{M}'',\underline{M}')$ if, and only if, $\underline{M}$ is mixed. 
\end{Proposition}

\begin{proof}[Proof of Proposition \ref{prop-torsion-in-Ext-positive-weights}]
By Corollary \ref{cor:dualizable-tensor} of the appendix, taking $\underline{N}:=\underline{M}'\otimes (\underline{M}'')^{\vee}$ we can assume that the exact sequence is of the form $S:~0\to \underline{N}\to \underline{E}\to \mathbbm{1}\to 0$, the mixed $A$-motive $\underline{N}$ having positive weights. In view of Theorem \ref{thm:cohomology-in-MR}, we may assume that $\underline{E}=\iota(u)$ for some $u\in N[\fj^{-1}]$. Note that $0$ is a weight of $\underline{E}$, the smallest.

If $\underline{E}$ is mixed, then $\underline{E}$ contains a sub-$A$-motive $\underline{L}=(L,\tau_N)$ of weight $0$ which is isomorphic to $\mathbbm{1}$. Let $(m\oplus a)\in N\oplus (A\otimes F)$ be a generator of $L$ over $A\otimes F$. We have
\begin{equation}
\begin{pmatrix} \tau_N & u \\ 0 & 1 \end{pmatrix}\begin{pmatrix} \tau^*m \\ \tau^*a \end{pmatrix}=\begin{pmatrix} m \\ a \end{pmatrix}. \nonumber
\end{equation}
This amounts to $a\in A$ and $au\in \im(\id-\tau_N)$, and then that $a[\underline{E}]=0$ in $\Ext^1_{\cM_F}(\mathbbm{1},\underline{N})$. Conversely, if there exists a nonzero $a\in A$ such that $a[\underline{E}]$ is split, Theorem \ref{thm:cohomology-in-MR} implies that there exists $m\in N$ such that $au=m-\tau_N(\tau^*m)$. The nonzero $A\otimes F$-module $L$ generated by $m\oplus a$ together with $\tau_N$ defines a sub-$A$-module of $\underline{E}$ isomorphic to $\mathbbm{1}$. For all $\mu\in \bQ$, we define $A\otimes F$-modules
\begin{equation}
E_{\mu}:=W_{\mu}N+\mathbf{1}_{\mu\geq 0}L, \quad E_{<\mu}:=W_{<\mu}N+\mathbf{1}_{\mu>0}L \nonumber
\end{equation}
where $(W_{\mu}\underline{N})_{\mu\in \bQ}$ is the weight filtration of $\underline{N}$, and where $\mathbf{1}_{\mu\in S}$ is the indicator function of the set $S$. It is easy to see that $\underline{E}_\mu:=(E_{\mu},\tau_N)$ and $\underline{E}_{<\mu}:=(E_{<\mu},\tau_N)$ define sub-$A$-motives of $\underline{E}$ such that 
\[
\underline{E}_{\mu}/\underline{E}_{<\mu}\cong \begin{cases}
0 ~&\text{ if }~ \mu<0, \\
\underline{L} ~&\text{ if }~ \mu=0, \\
\Gr_\mu \underline{N} ~&\text{ if }~ \mu>0.
\end{cases}
\]
As desired, we have constructed an increasing sequence $(\underline{E}_\mu)_\mu$ of sub-$A$-motives of $\underline{E}$ satisfying the property of Definition \ref{def:mixed-A-motives}. Hence, $\underline{E}$ is mixed.
\end{proof}
\begin{Remark}\label{remk:Ext-nonzero}
Under the same hypothesis, Proposition \ref{prop-torsion-in-Ext-positive-weights} can be rephrased into
\begin{equation}
\Ext^1_{\cM_F}(\underline{M}'',\underline{M}')^{\text{tors}}=\Ext^1_{\cM\cM_F}(\underline{M}'',\underline{M}').\nonumber
\end{equation}
In particular, the $K$-vector space $\Ext^1_{\cM\cM^{\text{iso}}_F}(\underline{M}'',\underline{M}')$ vanishes. The latter is only conjectured to be true in the number fields setting (\cite[\S 1.3]{deligne-droite}).
\end{Remark}
\begin{Remark}
Nevertheless, $\Ext^1_{\cM\cM_F}(\mathbbm{1},\underline{M})$ is generally non zero for $\underline{M}$ having positive weights. In the notations of Example \ref{ex:carlitz-motive}, let $n$ be a positive integer that is a power of the characteristic $p$, and consider the $A$-motive $\underline{A}(-n)$ over a field extension $F$ of $K$ that contains a $(q-1)$-st root $\eta$ of $(-1/\theta)^n$. We claim that $\Ext^1_{\cM\cM_F}(\mathbbm{1},\underline{A}(-n))$ is non zero, equivalently that $\Ext^1_{\cM_F}(\mathbbm{1},\underline{A}(-n))$ has non zero torsion. Indeed, let $[\underline{E}]:=\iota(\eta^q)$. Then,
\[
-t^n\cdot [\underline{E}]=\iota(-t^n \eta^q)=\iota(\eta-(t-\theta)^n\eta^q)=0.
\]
On the other-hand, $[\underline{E}]$ is non zero: for degree reasons, there does not exist $p(t)\in F[t]$ such that $\eta^q=p(t)-(t-\theta)^n p(t)^{(1)}$.
\end{Remark}

\begin{Theorem}\label{thm:main0}
Let $\underline{M}$ be a mixed $A$-motive over $F$. For $i>1$, the $A$-module $\Ext^i_{\cM\cM_F}(\mathbbm{1},\underline{M})$ is torsion.
\end{Theorem}
\begin{proof}
Let $\eA$ be the $K$-linear abelian category $\cM\cM_F^{\text{iso}}$. From the relation
\[
\Ext^i_{\cM\cM_F^{\text{iso}}}(\mathbbm{1},\underline{M})=\Ext^i_{\cM\cM_F}(\mathbbm{1},\underline{M})\otimes_A K,
\]
it suffices to show that the left-hand side vanishes for $i>1$. Let us first treat the case where $\underline{M}$ only has non positive weights. Let $\eA_0$ be the full subcategory of $\eA$ consisting of objects whose weights are all non positive. As $\underline{M}$ is an object of $\eA_0$, note that 
\[
\Ext^1_{\eA_0}(\mathbbm{1},\underline{M})=\Ext^1_{\eA}(\mathbbm{1},\underline{M})=\Ext^1_{\cM_F^{\text{iso}}}(\mathbbm{1},\underline{M})
\]
where the last equality follows from Proposition \ref{prop:separated-weights-ext}. From Corollary \ref{cor:ontoonextension}, we deduce that the functor $\Ext^1_{\eA_0}(\mathbbm{1},-)$ is right-exact on $\eA_0$ hence $\Ext^2_{\eA_0}(\mathbbm{1},\underline{M})=(0)$ (Proposition \ref{prop:vanishing-higher-ext}). Now, from the exactness of $W_0$ over $\eA$ (Corollary \ref{cor:W-exact}), any $e\in \Ext^2_{\eA}(\mathbbm{1},\underline{M})$ produces a commutative diagram in $\eA$:
\begin{equation}
\begin{tikzcd}
e: & 0\arrow[r] & \underline{M} \arrow[r]  & \underline{E}_1 \arrow[r] & \underline{E}_2 \arrow[r] & \mathbbm{1} \arrow[r] & 0 \\
W_0e: & 0\arrow[r] & \underline{M} \arrow[r]\arrow[u,"\id"]  & W_0\underline{E}_1 \arrow[r]\arrow[u] & W_0\underline{E}_2 \arrow[r]\arrow[u] & \mathbbm{1} \arrow[r]\arrow[u,"\id"] & 0
\end{tikzcd}\nonumber
\end{equation}
from which we deduce $\Ext^2_{\eA_0}(\mathbbm{1},\underline{M})= \Ext^2_{\eA}(\mathbbm{1},\underline{M})$ as desired. Now, let $\underline{M}$ have arbitrary weights. Applying $\Ext^1_{\eA}(\mathbbm{1},-)$ to the exact sequence:
\[
0\longrightarrow W_0\underline{M} \longrightarrow \underline{M} \longrightarrow \underline{M}/W_0\underline{M} \longrightarrow 0,
\]
we obtain from Proposition \ref{prop-torsion-in-Ext-positive-weights} that the natural map:
\[
\Ext^1_{\eA}(\mathbbm{1},W_0\underline{M})\longrightarrow \Ext^1_{\eA}(\mathbbm{1},\underline{M})
\]
is surjective. Given an epimorphism $f:\underline{M}\to \underline{N}$ in $\eA$, we obtain a commutative square:
\begin{equation}
\begin{tikzcd}
\Ext^1_{\eA}(\mathbbm{1},W_0\underline{M}) \arrow[r,->>]\arrow[d] & \Ext^1_{\eA}(\mathbbm{1},\underline{M}) \arrow[d] \\
\Ext^1_{\eA}(\mathbbm{1},W_0\underline{N}) \arrow[r,->>] & \Ext^1_{\eA}(\mathbbm{1},\underline{N}).
\end{tikzcd}\nonumber
\end{equation}
The left vertical arrow is surjective: indeed, $W_0f:W_0\underline{M}\to W_0\underline{N}$ is an epimorphism by Corollary \ref{cor:W-exact}, and we already proved that over $\eA_0$ the functor $\Ext^1_{\eA}(\mathbbm{1},-)=\Ext^1_{\eA_0}(\mathbbm{1},-)$ is right-exact. Hence, the right vertical arrow is surjective and the functor $\Ext^1_{\eA}(\mathbbm{1},-)$ is right-exact on $\eA$. Hence $\Ext^i_{\eA}(\mathbbm{1},\underline{M})=(0)$ (Proposition \ref{prop:vanishing-higher-ext} again).
\end{proof}

\section{Models and the integral part of $A$-motivic cohomology}\label{chapter:integral models}
In this section we illustrate the notion of \emph{maximal integral models}. For $A$-motives, maximal integral models are understood as an analogue of N\'eron models of abelian varieties. The notion dates back to Gardeyn's work on \emph{models of $\tau$-sheaves} \cite{gardeyn2} and their reduction \cite{gardeyn1}, where he proved a N\'eron-Ogg-Shafarevi\v{c} type criterion (see Proposition \ref{prop-good-reduction-local-case} in our context). However our setting differs by the fact that, in opposition to $\tau$-sheaves, $A$-motives might not be \emph{effective}. We also removed Gardeyn's assumption for an \emph{integral model} to be locally free. We will show in Theorems \ref{thm:integral-models-are-locally-free-local} and \ref{thm:integral-models-are-locally-free} that this is automatic for maximal ones over local and global function fields. Our presentation thus avoids the use of a technical lemma due to Lafforgue in Gardeyn's exposition \cite[\S 2]{gardeyn2}. In that sense, the content of this section is original. \\

In practice, to make maximal integral models of $A$-motives explicit is a difficult task. In Subsection \ref{subsec:integral-models-of-frobenius-modules}, we consider the easier problem of finding maximal integral models of \emph{Frobenius spaces}. Those are pairs $(V,\varphi)$ where $V$ is a finite dimensional vector space over a local field $E$ containing $\bF$ and $\varphi$ is a $q$-linear endomorphism of $V$. We show in Proposition \ref{prop:maximal} that there exists a unique $\cO_E$-lattice in $V$ stable by $\varphi$ and which is maximal for this property. We end this section by the review of Katz's equivalence of categories, and its application to study maximal integral models.\\

In Subsections \ref{sec:integral-models-A-motives-local} and \ref{sec:integral-models-A-motives-global}, we shall be concerned with integral models of $A$-motives. Given an inclusion $R\subset S$ of $A$-algebras and an $A$-motive $\underline{M}=(M,\tau_M)$ over $S$, an $R$-model for $\underline{M}$ is a finite sub-$A\otimes R$-module of $M$ \emph{stable by $\tau_M$} (Definition \ref{def:integral-models-of-motives}). \\
We study the case where $\underline{M}$ is an $A$-motive over a local function field $S=E$ and where $R=\cO_E$ is its valuation ring in Subsection \ref{sec:integral-models-A-motives-local}. In Proposition \ref{prop-existence-of-maximal-integral-model-local}, we prove existence and uniqueness of integral $\cO_E$-models which are maximal for the inclusion, and we prove that they are projective in Theorem \ref{thm:integral-models-are-locally-free-local}. We show that, given a well-chosen maximal ideal $\ell\subset A$ and a positive integer $n$, the data of $(M/\ell^n M,\tau_M)$ defines a Frobenius space over $E$. Theorem \ref{thm:integral-model:motive-to-frobenius-module}, our main result of this section, describes how to recover the maximal integral model of $\underline{M}$ in terms of the data of the maximal integral model of $(M/\ell^n M,\tau_M)$ for all $n$. The latter is of fundamental importance in the proof of Theorem \ref{mthm:Integral-inside-Good}, and permits to obtain the N\'eron-Ogg-Shafarevi\v{c} type criterion for $A$-motives (Proposition \ref{prop-good-reduction-local-case}). \\

In Subsection \ref{sec:integral-models-A-motives-global}, we treat the case where $\underline{M}$ is an $A$-motive over a global function field $S=F$ where $R$ is a Dedekind domain whose fraction field is $F$. If $\fp$ is a non zero prime ideal of $R$, we obtain an $A$-motive $\underline{M}_{F_{\fp}}$ by base field extension from $F$ to its completion $F_{\fp}$. Our Proposition \ref{prop-existence-of-maximal-integral-model-global} explains how to recover the maximal integral model of $\underline{M}$ from the data of the maximal integral models of $\underline{M}_{F_{\fp}}$ for all $\fp$. \\

The full force of this section is used in Subsection \ref{subsec:integral-part} to prove Theorems \ref{mthm:Integral-inside-Good} and \ref{mthm:Global-Integral} of the introduction (respectively Theorems \ref{thm:main1} and \ref{thm:main2} in the text).

\subsection{Models of Frobenius spaces}\label{subsec:integral-models-of-frobenius-modules}
In this subsection we work with notations that are slightly more general to what we need in the sequel. Let $(\cO_E,v_E)$ be a discrete valuation ring that contains the field $\bF$. We denote by $\fp=\fp_E$ its maximal ideal and we fix $\varpi$ a uniformizer. We let $E$ be the field of fractions of $\cO_E$. For what we have in mind, it will be enough to take for $\cO_E$ a valuation subring of a function field over $\bF$ or its completion. We let $\sigma:E\to E$ denote the $q$-Frobenius on $E$ (it fixes $\bF$). \\

Our object of study are pairs $(V,\varphi)$ where $V$ is a finite dimensional $E$-vector space and $\varphi:\sigma^*V\to V$ is an $E$-linear isomorphism. In the existing literature, they are generally referred to as \emph{\'etale finite $\bF$-shtukas over $E$} (\emph{e.g.} \cite[\S 4]{hartl-isogeny}). We prefer here the shorter name \emph{Frobenius spaces}. By an \emph{$\cO_E$-lattice in $V$} we mean a finitely generated sub-$\cO_E$-module $L$ of $V$ which generates $V$ over $E$. A sub-$\cO_E$-module \emph{$L$ is stable by $\varphi$} if $\varphi(\sigma^* L)\subset L$. 
\begin{Definition}
We say that $L$ is an \emph{integral model for $(V,\varphi)$} if $L$ is an $\cO_E$-lattice in $V$ stable by $\varphi$. We say that $L$ is \emph{maximal} if it is not strictly included in another integral model for $(V,\varphi)$.
\end{Definition}
\begin{Proposition}\label{prop:maximal}
A maximal integral model for $(V,\varphi)$ exists and is unique.
\end{Proposition}
\begin{proof}
Our proof follows closely \cite[Prop. 2.2, 2.13]{gardeyn2}. First note that there exists an integral model. Indeed, let $T'$ be an arbitrary $\cO_E$-lattice in $V$. There exists a positive integer $k$ such that $\varphi(\sigma^*T')\subset \varpi^{-k}T'$. We let $T:=\varpi^kT'$ so that
\begin{equation}
\varphi(\sigma^*T)=\varpi^{qk}\varphi(\sigma^*T')\subset \varpi^{(q-1)k}T'=\varpi^{(q-2)k}T\subset T. \nonumber
\end{equation}
Hence, the $\cO_E$-module $T$ is an $\cO_E$-lattice in $V$ stable by $\varphi$.

We turn to the existence and uniqueness of the maximal integral model. If $L'\subset L$ is an inclusion of integral models, we have:
\[
\operatorname{length}_{\cO_E}\left(\varphi(\sigma^* L)/\varphi(\sigma^* L')\right)=q\cdot \operatorname{length}_{\cO_E}\left(L/L'\right).
\]
We define the \emph{discriminant of $L$} to be the non-negative integer 
\[
\Delta(L):=\operatorname{length}_{\cO_E}\left(L/\varphi(\sigma^*L)\right).
\]
Since,
\begin{align}\label{eq:discriminant-multiple-of-q-1}
\Delta(L')-\Delta(L) &=\operatorname{length}_{\cO_E}\left(L'/\varphi(\sigma^*L')\right)-\operatorname{length}_{\cO_E}\left(L/\varphi(\sigma^*L)\right) \nonumber\\
&= \operatorname{length}_{\cO_E}\left(\varphi(\sigma^*L)/\varphi(\sigma^*L')\right)-\operatorname{length}_{\cO_E}\left(L/L'\right) \nonumber \\
&= (q-1)\cdot \operatorname{length}_{\cO_E}\left(L'/L\right),
\end{align}
we have $\Delta(L')>\Delta(L)$ whenever the inclusion $L'\subset L$ is strict.

Now let $L$ be an integral model with minimal discriminant. We claim that $L$ equals the union of all integral models of $(V,\varphi)$, which proves both existence and uniqueness of the maximal integral model. Indeed, if $L'$ is another integral model for $(V,\varphi)$ not contained in $L$, then $L+L'$ is an integral model such that the inclusion $L\subset L+L'$ is strict. But this contradicts the minimality assumption as we would have $\Delta(L)>\Delta(L+L')$.
\end{proof}

\begin{Example}
Assume that we are in the local function field situation, namely $\cO_E=k[\![\varpi]\!]$ for some field $k$ containing $\bF$. Suppose $V:=E$, $f\in \cO_E$ a nonzero element and $\varphi$ is the morphism corresponding to $x\mapsto fx^q$. Write $f=u\varpi^k h^{q-1}$ where $u\in \cO_E^{\times}$, $0\leq k<q-1$ is an integer and $h\in \cO_E$. Then, the maximal integral model of $(V,\varphi)$ is given by $h^{-1}\cO_E$. This is because $\Delta(h^{-1}\cO_E)=k$ together with \eqref{eq:discriminant-multiple-of-q-1}.
\end{Example}

Let $T$ be an integral model for $(V,\varphi)$ and let $r$ be its rank as a free $\cO_E$-module. The cokernel of the inclusion $\varphi(\sigma^*T)\subset T$ is a torsion $\cO_E$-module of finite type and there exists elements $g_1,\ldots , g_r$ in $\cO_E$ with $v_E(g_i)\leq v_E(g_{i+1})$ such that
\begin{equation}
T/\varphi(\sigma^*T)\cong \cO_E/(g_1)\oplus \cO_E/(g_2)\oplus \cdots \oplus \cO_E/(g_r). \nonumber
\end{equation}
Equivalently, there exists a basis $(v_1,\ldots,v_r)$ of $T$ over $\cO_E$ such that
\begin{equation}
\varphi(\sigma^*T)=(g_1)v_1\oplus (g_2)v_2\oplus \cdots \oplus (g_r)v_r. \nonumber
\end{equation}
The elements $g_1,\ldots ,g_r$ are unique up to multiplication by units and are called \emph{the elementary divisors relative to the inclusion of $\cO_E$-lattices $\varphi(\sigma^*T)\subset T$}.
\begin{Lemma}\label{lem:smith}
Let $\textbf{t}$ be a basis of $T$ over $\cO_E$ and let $F$ be the matrix of $\varphi$ written in the bases $\sigma^*\textbf{t}$ and $\textbf{t}$. The elementary divisors relative to the inclusion $\varphi(\sigma^*T)\subset T$ are the elementary divisors of the matrix $F$, up to units in $\cO_E$.
\end{Lemma}
\begin{proof}
If $(f_1,\ldots,f_r)$ denotes the elementary divisors of $F$ with $v_E(f_i)\leq v_E(f_{i+1})$, the Smith's normal form Theorem implies that there exists $U,V\in \operatorname{GL}_r(\cO_E)$ such that $UF=\operatorname{diag}(f_1,\ldots ,f_r)V$. If we let $\textbf{v}=(v_1,\ldots,v_r)$ be the basis of $T$ corresponding to $V\cdot \mathbf{t}$, this relation reads
\begin{equation}
\varphi(\sigma^*T)=(f_1)v_1\oplus (f_2)v_2\oplus \cdots \oplus (f_r)v_r. \nonumber
\end{equation}
By uniqueness of the ideals $(g_1),\ldots,(g_r)$, we conclude that $(f_i)=(g_i)$ for all $i\in \{1,\ldots,r\}$.
\end{proof}

\begin{Remark}\label{rem:base-change-Frob}
Let $\widehat{\cO}_E$ denote the completion of $\cO_E$ with respect to its valuation, and let $\widehat{E}$ be its field of fractions. The datum of $V\otimes_{E}\widehat{E}$ equipped with $\varphi\otimes \sigma$ defines a Frobenius space over $\widehat{E}$; let $V_{\widehat{\cO}}$ be its maximal integral model. Given a matrix $M$ with coefficients in $\cO_E$, the valuation of its elementary divisors are unchanged if we consider $M$ in $\widehat{\cO}_E$. Consequently, one deduces from Lemma \ref{lem:smith} that the inclusion $V_{\cO}\otimes_{\cO_E}\widehat{\cO}_E\subseteq V_{\widehat{\cO}}$ is an equality. 
\end{Remark}

\begin{Definition}\label{def:range}
We let the \emph{type of $T$} be the sequence $(e_1,\ldots,e_r)$ of the valuations of the elementary divisors relative to the inclusion $\varphi(\sigma^*T)\subset T$ ordered such that $0\leq e_1\leq e_2 \leq \cdots\leq e_r$. We define \emph{the range $r_T$ of T} to be the integer $e_r$.
\end{Definition}
\begin{Remark}
We have $\Delta(T)=e_1+\ldots +e_r$ so that $r\leq \Delta(T)\leq r\cdot r_T$ where $\Delta(T)$ denotes the discriminant of $T$. 
\end{Remark}

We shall denote by $V_{\cO}$ the maximal integral model of $(V,\varphi)$. The following proposition enables us to say \emph{how far} an integral lattice is from being maximal given its range.
\begin{Proposition}\label{prop-bound-integral-lattice}
Let $T$ be an $\cO_E$-lattice in $V$ stable by $\varphi$. Let $s$ be a non-negative integer. If the range of $T$ satisfies $r_T\leq s(q-1)$, then $V_{\cO}\subset \varpi^{-s}T$. 
\end{Proposition}
We start by a lemma:
\begin{Lemma}\label{lem-converse-inclusion-contains-L}
Let $U$ be an $\cO_E$-lattice in $V$ such that $U\subset \varphi(\sigma^* U)$. Then $V_{\cO}\subset U$.
\end{Lemma}
\begin{proof}
For $n\geq 0$, we let $\sigma^{n*}:=(\sigma^n)^*$ and denote by $\varphi^n:\sigma^{n*}V\to V$ the $E$-linear morphism given by the composition
\begin{equation}
\sigma^{n*}V\xrightarrow{\sigma^{(n-1)*}\varphi} \sigma^{(n-1)^*}V\longrightarrow \cdots \longrightarrow \sigma^*V \stackrel{\varphi}{\longrightarrow} V. \nonumber
\end{equation}
We consider the following sub-$\cO_E$-module of $V$:
\[
\tilde{U}:=\bigcup_{n=0}^{\infty} \varphi^n(\sigma^{n*}U).
\]
As an increasing union of $\cO_E$-lattices, $\tilde{U}$ is a sub-$\cO_E$-module of $V$. The intersection $V_{\cO}\cap \tilde{U}$ is finitely generated over $\cO_E$, stable by $\varphi$ and generates $V$ over $E$ as it contains the $\cO_E$-lattice $V_{\cO}\cap U$. That is, $V_{\cO}\cap \tilde{U}$ is an integral model whose discriminant 
\[
\Delta(V_{\cO}\cap \tilde{U}):=\operatorname{length}_{\cO_E}\left(\frac{V_{\cO}\cap \tilde{U}}{\varphi(\sigma^*V_{\cO}\cap \tilde{U})}\right)=\operatorname{length}_{\cO_E}\left(\frac{V_{\cO}\cap \tilde{U}}{\varphi(\sigma^*V_{\cO})\cap \tilde{U}}\right)
\]
is smaller than that of $V_\cO$ (indeed, any chain of strict submodules from $\varphi(\sigma^*L)$ to $L$ induces, by intersecting with $\tilde{U}$, a smaller chain of strict submodules from $\varphi(\sigma^*L)\cap \tilde{U}$ to $L\cap \tilde{U}$). By maximality, $\Delta(L)$ is minimal hence $V_{\cO}=V_{\cO}\cap \tilde{U}$. We deduce that there exists a non-negative integer $m$ such that $V_{\cO}\subset \varphi^m(\sigma^{m*}U)$. Because $\varphi(\sigma^*V_{\cO})\subset V_{\cO}$, we have $\sigma^*V_{\cO}\subset \varphi^{-1}(V_{\cO})$ and by immediate recursion one gets $\sigma^{m*}V_{\cO}\subset \varphi^{-m}(V_{\cO})\subset \sigma^{m*}U$. We conclude that $V_{\cO}\subset U$ because $\sigma:\cO_E\to \cO_E$ is faithfully flat.
\end{proof}
\begin{proof}[Proof of Proposition \ref{prop-bound-integral-lattice}]
Let $(e_1,\ldots,e_r)$ be the type of $T$. Recall that $\fp=\fp_E$ denotes the maximal ideal of $\cO_E$. There exists a basis $(t_1,\ldots,t_r)$ of $T$ such that $\varphi(\sigma^*T)=\fp^{e_1}t_1\oplus \fp^{e_2}t_2\oplus \cdots\oplus \fp^{e_r}t_r$. By assumption, $e_1,\ldots,e_r\leq s(q-1)$ and thus
\begin{align*}
\varpi^{-s}T &\subset \varpi^{-s}\left(\fp^{e_1-s(q-1)}t_1\oplus \fp^{e_2-s(q-1)}t_2\oplus \cdots \oplus\fp^{e_r-s(q-1)}t_r\right) \\
&=\fp^{e_1-sq}t_1\oplus \fp^{e_2-sq}t_2\oplus \cdots \oplus\fp^{e_r-sq}t_r\\
&=\varpi^{-qs}\varphi(\sigma^*T) \\
&=\varphi(\sigma^*(\varpi^{-s}T)).
\end{align*}
Hence, $U:=\varpi^{-s}T$ satisfies $U\subset \varphi(\sigma^*U)$ and we deduce that $V_{\cO}\subset U$ by Lemma \ref{lem-converse-inclusion-contains-L}.
\end{proof}

Akin to integral models, there is also a notion of \emph{good sublattices}.
\begin{Definition}
Let $L$ be a finitely generated $\cO_E$-submodule of $V$. We say that $L$ is a \emph{good sublattice for $(V,\varphi)$} if $\varphi(\sigma^*L)=L$. We say that $L$ is \emph{maximal} if it is not strictly included in another good sublattice of $(V,\varphi)$.
\end{Definition}
\begin{Proposition}\label{prop:maximal-good-model-Frob-space}
A maximal good sublattice for $(V,\varphi)$ exists, is unique and contained in $V_{\cO}$.
\end{Proposition}
\begin{proof}
First note that any good sublattice $L$ for $(V,\varphi)$ is contained in $V_{\cO}$: indeed, $L+V_{\cO}$ is again an integral model, hence included in $V_{\cO}$. The sum $U$ over all good sublattice for $(V,\varphi)$ exists (it is non-empty as the zero module is a good sublattice) and therefore included in $V_{\cO}$. Because $\cO_E$ is Noetherian, $U$ is a finitely generated $\cO_E$-module. We also have $\varphi(\sigma^* U)=U$. We deduce that $U$ is maximal, unique and contained in $V_{\cO}$. 
\end{proof}

We shall denote by $V_{\text{good}}$ the maximal good sublattice of $(V,\varphi)$; it is a sub-$\cO_E$-module of $V_\cO$.
\begin{Proposition}
The inclusion $V_{\operatorname{good}}\subset V_\cO$ splits. 
\end{Proposition}
\begin{proof}
Consider the subobject of $(V,\varphi)$ whose underlying $E$-vector space $U$ is generated by $V_{\operatorname{good}}$ over $E$. Its maximal integral model is $V_{\operatorname{good}}$ (it has zero discriminant) and thus $V_{\cO}\cap U\subset V_{\operatorname{good}}$. In particular, the quotient module $V_{\cO}/V_{\operatorname{good}}$ is finite torsion-free, hence finite free over $\cO_E$. In particular, the surjection $V_{\cO}\to V_{\cO}/V_{\operatorname{good}}$ splits and the proposition follows.  
\end{proof}

\begin{Definition}\label{def:good-red-frob-space}
We call the rank of $V_{\text{good}}$ the \emph{non-degenerate rank of $(V,\varphi)$}. We say that $(V,\varphi)$ has \emph{good reduction} if $V_{\text{good}}=V_{\cO}$; \emph{i.e.} if the non-degenerate rank is maximal. 
\end{Definition}

We can give a rather useful formula for $V_{\operatorname{good}}$ in terms of $V_{\cO}$.
\begin{Lemma}\label{lem:formula-for-Vgood}
$\displaystyle V_{\operatorname{good}}=\bigcap_{n\geq 0}{\varphi^n(\sigma^{n*} V_{\cO})}$.
\end{Lemma}
\begin{proof}
The right-hand side $L$ is a decreasing intersection of $\cO_E$-lattices hence is itself a finitely generated $\cO_E$-module. Since it further verifies $\varphi(\sigma^*L)=L$, one obtains $L\subset V_{\operatorname{good}}$. But $ V_{\operatorname{good}}=\varphi^n(\sigma^{n*}V_{\operatorname{good}})\subset \varphi^n(\sigma^{n*}V_{\cO})$ for all $n\geq0$, therefore $L=V_{\operatorname{good}}$.
\end{proof}

Maximal good sublattices have an interpretation in terms of \emph{Frobenius sheaves} that we now recall. Let $X$ be a smooth connected scheme over $\bF$, and let $\pi(X)$ be its \'etale fundamental group. We still denote by $\sigma$ the Frobenius on $X$. Let $\cF(X)$ be the category whose objects are pairs $(\cV,\varphi)$ where $\cV$ is a locally-free $\cO_X$-module of finite rank and $\varphi:\sigma^*\cV\to \cV$ is an isomorphism of $\cO_X$-modules. Morphisms in this category are morphisms of the underlying $\cO_X$-modules with commuting $\varphi$-action. 
\begin{Example}
Objects of $\cF(\Spec E)$ are Frobenius spaces over $E$, and objects of $\cF(\Spec \cO_E)$ are pairs $(V,\varphi)$ where $V$ is a finite free $\cO_E$-module, and where $\varphi:\sigma^*V\to V$ is an $\cO_E$-linear isomorphism.
\end{Example}

Until the end of this subsection, $E$ is a local function field with perfect residue field $k$. As a consequence, $E=k(\!(\varpi)\!)$. We fix a separable closure $E^s$ of $E$ and denote by $G_E$ the absolute Galois group $\Gal(E^s|E)$ of $E$. Let also $I_E\subset G_E$ be the inertia subgroup. The following result is due to Katz in \cite[Prop. 4.1.1]{katz-modular}. 
\begin{Theorem}\label{thm:katz-equivalence}
There is a rank-preserving equivalence of categories from $\cF(X)$ to the category of $\bF$-linear continuous representation of $\pi(X)$, which commutes with base change. For $X=\Spec E$, it is explicitly given by
\[
\underline{V}=(V,\varphi) \longmapsto T\underline{V}=\{x\in V\otimes_E E^s~|~x=\varphi(\sigma^* x)\}
\]
where $\pi(\Spec E)$ is identified with $\Gal(E^s|E)$, and acts on the right-hand side of $V\otimes_E E^s$. 
\end{Theorem}

The following proposition is almost immediate from Katz's equivalence:
\begin{Proposition}\label{prop:babyNOS}
Let $\underline{V}$ be a Frobenius space over $E$. The non-degenerate rank of $(V,\varphi)$ equals the rank of $(T\underline{V})^{I_E}$. In particular, $\underline{V}$ has good reduction if and only if $T\underline{V}$ is unramified. 
\end{Proposition}
\begin{proof}
As $\pi(\Spec \cO_E)\cong G_E/I_E$, the representation $(T\underline{V})^{I_E}$ is the maximal subobject of $T\underline{V}$ which comes from an object in $\cF(\Spec \cO_E)$ by Katz's equivalence. Yet, elements in $\cF(\Spec \cO_E)$ which specialize to subobjects of $\underline{V}$ by base change to $E$ are exactly the good sublattices of $\underline{V}$.
\end{proof}
 
We end this subsection by the next result, which will be the main ingredient to obtain Theorem \ref{mthm:Integral-inside-Good} of the introduction.
\begin{Proposition}\label{prop:artin-schreier}
Let $\underline{V}=(V,\varphi)$ be a Frobenius space over $E$, and let $x\in V$. The following are equivalent:
\begin{enumerate}[label=$(\roman*)$]
\item\label{item:i-unramified-solution} There exists $y\in V\otimes_E E^{\operatorname{ur}}$ such that $x=y-\varphi(\tau^*y)$,
\item\label{item:ii-x-is-good} $x\in V_{\operatorname{good}}+(\id_V-\varphi)(V)$,
\item\label{item:iii-x-is-inetgral} $x\in V_{\cO}+(\id_V-\varphi)(V)$.
\end{enumerate}
\end{Proposition}
\begin{proof}
We first prove the equivalence between \ref{item:i-unramified-solution} and \ref{item:ii-x-is-good}. Let $\mathbf{1}:\sigma^*E\to E$ be the canonical $E$-linear isomorphism and let $\mathbbm{1}$ be the \emph{neutral} Frobenius space $(E,\mathbf{1})$ over $E$ . Let also $\Ext^1(\mathbbm{1},\underline{V})$ be the $\bF$-vector space of Yoneda extensions of $\mathbbm{1}$ by $\underline{V}$ in the category $\cF(\Spec E)$. We have an isomorphism of $\bF$-vector spaces, natural in $\underline{V}$,
\begin{equation}\label{eq:ext}
\frac{V}{(\id-\varphi)(V)}\stackrel{\sim}{\longrightarrow}\Ext^1(\mathbbm{1},\underline{V})
\end{equation}
mapping a representative $v\in V$ to the class of the extension of $\mathbbm{1}$ by $\underline{V}$ whose underlying module is $V\oplus E$ and whose Frobenius action is given by $\left(\begin{smallmatrix}\varphi & v\cdot \mathbf{1} \\ 0 & \mathbf{1} \end{smallmatrix} \right)$. Katz's equivalence leads to a commutative square
\begin{equation}
\begin{tikzcd}
\displaystyle\frac{V_{\text{good}}}{(\id_V-\varphi)(V_{\text{good}})} \arrow[r,"\sim"] \arrow[d] & \operatorname{H}^1(\pi(\Spec \cO_E),(T\underline{V})^{I_E})  \arrow[d,hook] \\
\displaystyle\frac{V}{(\id_V-\varphi)(V)} \arrow[r,"\sim"] & \operatorname{H}^1(G_E,T\underline{V})
\end{tikzcd} \nonumber
\end{equation}
where, by diagram chasing, the bottom row is given as follows: for $v\in V$, let $w\in V\otimes_E E^s$ be such that $v=w-\varphi(\tau^*w)$, then $c_v:\rho\mapsto w-\,^{\rho}w$ defines a cocycle $c_v:G_E\to T\underline{V}$ whose class does not depend on the choice of $w$. The bottom row maps $v$ to $c_v$. Hence, \ref{item:i-unramified-solution} holds if and only if $c_x$ comes from a cocycle in $\operatorname{H}^1(\pi(\Spec \cO_E),(T\underline{V})^{I_E})$, that is, if and only if \ref{item:ii-x-is-good} holds.

It remains to prove that \ref{item:ii-x-is-good} and \ref{item:iii-x-is-inetgral} are equivalent. The map $v\mapsto \varphi(\sigma^*v)$ is topologically nilpotent on $V_{\cO}/V_{\operatorname{good}}$, as can be deduced from Lemma \ref{lem:formula-for-Vgood}, and thus $\id-\varphi$ acts as an $\bF$-linear automorphism on it. Therefore, the Snake Lemma implies that the natural map from the cokernel of $\id-\varphi$ on $V_{\operatorname{good}}$ to that on $V_{\cO}$ is an isomorphism. This implies $V_{\cO}\subset V_{\operatorname{good}}+(\id-\varphi)(V_{\cO})$ and yields
\[
V_{\operatorname{good}}+(\id_V-\varphi)(V)=V_{\cO}+(\id_V-\varphi)(V)
\]
as desired.
\end{proof}

\subsection{Models of $A$-motives over a local function field}\label{sec:integral-models-A-motives-local}
\subsubsection*{General theory}
Let $R$ be a commutative $\bF$-algebra given together with an $\bF$-algebra morphism $\kappa:A\to R$. Let $S$ be a commutative ring which contains $R$ as a subring. Let $\underline{M}=(M,\tau_M)$ be an $A$-motive over $S$ (with characteristic morphism $A\stackrel{\kappa}{\to} R\to S$).
\begin{Definition}\label{def:integral-models-of-motives}
We define an \emph{$R$-integral model $L$ for $\underline{M}$} to be a finitely generated sub-$A\otimes R$-module of $M$ satisfying:
\begin{enumerate}[label=$(\roman*)$]
\item $L$ generates $M$ over $S$ (equivalently, over $A\otimes S$),
\item\label{item:def-integral-models} $\tau_M(\tau^*L)\subset L[\fj^{-1}]$.
\end{enumerate}
We say that $L$ is \emph{maximal} if it is not strictly contained in any other $R$-integral model of $\underline{M}$.
\end{Definition}

The next proposition is inspired by \cite[Prop. 2.2]{gardeyn2}:
\begin{Proposition}\label{prop:existence-of-S-models}
If $S$ is obtained from $R$ by localization, an $R$-model for $\underline{M}$ exists.
\end{Proposition}
\begin{proof}
Let $(m_1,\ldots,m_s)$ be generators of $M$ as an $A\otimes S$-module, and let $L_0$ be the sub-$A\otimes R$-module of $M$ generated by $(m_1,\ldots,m_s)$. Let $d\in R$ be such that $\tau_M(\tau^*L_0)\subset d^{-1}L_0[\fj^{-1}]$, and set $L:=d L_0$. We have 
\begin{equation}
\tau_M(\tau^*L)=d^{q}\tau_M(\tau^*L_0)\subset d^{q-1}L_0[\fj^{-1}]=d^{q-2}L[\fj^{-1}]\subset L[\fj^{-1}]. \nonumber
\end{equation}
Thus $L$ is an $R$-model. 
\end{proof}

When it exists, we have the following claim for a maximal integral model.
\begin{Lemma}\label{lem:uniqueness-of-maximal-integral}
A maximal $R$-integral model for $\underline{M}$ contains all the $R$-integral models for $\underline{M}$. In particular, if it exists it is unique.
\end{Lemma}
\begin{proof}
Given $L_1$ and $L_2$ two $R$-integral models, their sum $L_1+L_2$ again defines an $R$-integral model. Hence, if $L_1$ is maximal, the inclusion $L_1\subseteq L_1+L_2$ is not strict: we deduce $L_1+L_2=L_1$, then $L_2\subset L_1$.
\end{proof}

As for Frobenius spaces, we have the notion of a good sublattice. 
\begin{Definition}\label{def:good-models-of-motives}
We define an \emph{$R$-good sublattice $L$ for $\underline{M}$} to be a finitely generated sub-$A\otimes R$-module of $M$ such that $\tau_M(\tau^*L)[\fj^{-1}]=L[\fj^{-1}]$. We say that $L$ is maximal if it is not strictly contained in any other $R$-good sublattice of $\underline{M}$.
\end{Definition}

Following the argument given in the proof of Proposition \ref{prop:maximal-good-model-Frob-space}, we obtain:
\begin{Lemma}\label{lem:existence-of-R-good-models}
Assume that there exists a maximal $R$-integral model for $\underline{M}$. Then, a maximal $R$-good sublattice for $\underline{M}$ exists and is unique.
\end{Lemma}

We continue this section by recording additional properties of maximal $R$-models in the general situation. Those will eventually by useful in Subsection \ref{subsec:integral-part} for the proof of Theorem \ref{mthm:Global-Integral} (Theorem \ref{thm:main2} in the text) when specializing to $R\subset S$ being the inclusion of a Dedekind domain in its fraction field.\\ 

Let $\underline{M}$ be an $A$-motive over $S$ and suppose that it admits a maximal integral $R$-integral model denoted by $M_R$.
\begin{Proposition}\label{prop:finitely-generated-stable-module-model}
Let $N$ be a finitely generated sub-$A\otimes R$-module of $M$ such that $\tau_M(\tau^*N)\subset N[\fj^{-1}]$. Then, $N\subset M_{R}$. In particular, any element $m\in M$ such that $\tau_M(\tau^*m)=m$ belongs to $M_{R}$.
\end{Proposition}
\begin{proof}
It suffices to notice that the module $L$ generated by $M_{R}$ and $N$ over $A\otimes R$ is an $R$-model for $\underline{M}$, and hence $N\subset L\subset M_{R}$.
\end{proof}

\begin{Corollary}\label{cor:(id-tau)(M0)=(id-tau)(M)capMO}
We have $(\id-\tau_M)(M_{R})=(\id-\tau_M)(M)\cap M_{R}[\fj^{-1}]$.
\end{Corollary}
\begin{proof}
The inclusion $(\id-\tau_M)(M_{R})\subset (\id-\tau_M)(M)\cap M_{R}[\fj^{-1}]$ is clear. Conversely, let $m \in M_{R}[\fj^{-1}]$ and let $n\in M$ be such that $m=n-\tau_M(\tau^*n)$. The sub-$A\otimes R$-module $\langle M_{R},n\rangle$ of $M$ generated by elements of $M_{R}$ together with $n$ over $A\otimes R$ is an $R$-model for $\underline{M}$. In particular, $\langle M_{R},n\rangle\subset M_R$ and $n\in M_R$. 
\end{proof}

We end this subsection with a remark on the assignment $\underline{M}\mapsto M_R$. For it to be well-defined assume that the inclusion $R\subset S$ is such that every object in $\cM_S$ admits a maximal $R$-integral model (this is the case when $R\subset S$ is the inclusion of a Dedekind domain into its field of fractions, as will be shown below in Proposition \ref{prop-existence-of-maximal-integral-model-global}).
\begin{Corollary}\label{cor:integral-model-functor}
Let $f:\underline{M}\to\underline{N}$ be a morphism in $\cM_S$. Then $f(M_{R})\subset N_{R}$. In particular, the assignment $\underline{M}\mapsto M_{R}$ is functorial.
\end{Corollary}

\subsubsection*{Existence and first properties over DVR}
We examine the existence of maximal models over discrete valuation rings. Let $(\cO_E,v_E)$ be a discrete valuation ring with maximal ideal $\fp=\fp_E$ and fraction field $E$. We fix $\varpi\in \fp$ a uniformizer. We assume that $\kappa:A\to \cO_E$ is such that $\kappa^{-1}(\fp)$ admits a non zero element. In this subsection, we shall be concerned by models of $A$-motives in the situation where $R\subset S$ is the inclusion $\cO_E\subset E$.\\

Let $\underline{M}$ be an $A$-motive over $E$. In order to prove the existence of a maximal integral model for $\underline{M}$, we start with a series of lemmas. Let $L$ and $L'$ be two finitely generated sub-$A\otimes \cO_E$-modules of $M[\fj^{-1}]$ and assume that $L'$ generates $M[\fj^{-1}]$ over $(A\otimes E)[\fj^{-1}]$.
\begin{Lemma}\label{lem:existence-of-index}
There exists an integer $k$ such that $\varpi^kL\subset L'[\fj^{-1}]$.
\end{Lemma}
\begin{proof}
We have $L\subset M[\fj^{-1}]=\displaystyle \bigcup_{k\geq 0}\varpi^{-k}L'[\fj^{-1}]$ (colimits commute). In particular, the sequence $(L\cap \varpi^{-k}L'[\fj^{-1}])_{k\geq 0}$ defines an ascending chain of sub-$A\otimes \cO_E$-modules of $L$. As $L$ is Noetherian, this chain eventually stabilizes, from which we deduce $L\subset \varpi^{-k}L'[\fj^{-1}]$ for $k$ large enough. 
\end{proof}

We denote by $[L:L']$ the smallest integer $k$ for which Lemma \ref{lem:existence-of-index} holds. Before listing important properties of this numerical invariant, the following is needed:
\begin{Lemma}\label{lem:intersection-projective}
Let $T$ be a flat $A\otimes \cO_E$-module and let $N:=T\otimes_{\cO_E}E$. Inside $N[\fj^{-1}]$, 
\begin{enumerate}
\item\label{item:multiplication-by-a-intersection} For all $a\in A$, we have $T\cap aN=aT$.
\item\label{item:j-intersection} We have  $T=T[\fj^{-1}]\cap N$.
\end{enumerate}
\end{Lemma}
\begin{proof}
We prove \ref{item:multiplication-by-a-intersection}: if $t\in N$ and $a\in A$ are such that ${(a\otimes 1)t\in T}$, then $t\in T$. Indeed, the map $T/(a\otimes 1)T\to N/(a\otimes 1)N$ arises as 
\[
T\otimes_{A\otimes \cO_E}(A/aA\otimes \cO_E\hookrightarrow A/aA\otimes E),
\]
and hence is injective by flatness of $T$.

We turn to \ref{item:j-intersection}. Let $c$ be an element of $T[\fj^{-1}]\cap N$. Let $a\in \kappa^{-1}(\fp)$ be non zero (which exists by our assumption on $\kappa$). Then, there exists $u,v\geq 0$ such that $(1\otimes \varpi^u)c$ and $(a\otimes 1-1\otimes \kappa(a))^v c$ are in $T$. Taking $w\geq 0$ such that $q^w\geq u,v$, we have
\[
(a^{q^w}\otimes 1)c=(a\otimes 1-1\otimes \kappa(a))^{q^w}c+(1\otimes \kappa(a)^{q^{w}})c\in T.
\] 
By part \ref{item:multiplication-by-a-intersection} we have $(a^{q^w}\otimes 1)c\in T\cap a^{q^w}N=a^{q^w}T$ and then $c\in T$.
\end{proof}

\begin{Lemma}\label{lem:properties-index}
Let $L,L',L''$ be finitely generated sub-$A\otimes \cO_E$-modules of $M[\fj^{-1}]$ which generates it over $(A\otimes E)[\fj^{-1}]$. 
\begin{enumerate}[label=$(\alph*)$]
\item\label{item:negative} If $L\subset L'[\fj^{-1}]$, then $[L:L']\leq 0$,
\item\label{item:inequality-in-chain} $[L:L'']\leq [L:L']+[L':L'']$.
\end{enumerate}
Let $T$ be a finite flat sub-$A\otimes \cO_E$-module of $M$ which generates it over $A\otimes E$. If $L\subset M$ and $L$ generates $M$ over $A\otimes E$, we have
\begin{enumerate}[resume,label=$(\alph*)$]
\item\label{item:frobenius-index} $[\tau_M(\tau^* L):\tau_M(\tau^* T)]=q[L:T]$.
\end{enumerate}
If, in addition, $L$ is a model of $\underline{M}$, then  
\begin{enumerate}[resume,label=$(\alph*)$]
\item\label{item:index-bounded} $\displaystyle [L:T]\leq \frac{[T:\tau_M(\tau^*T)]}{q-1}$.
\end{enumerate}
\end{Lemma}
\begin{proof}
Property \ref{item:negative} is obvious. Writing $k$ and $k'$ for $[L:L']$ and $[L':L'']$ respectively, we have $\varpi^k L\subset L'[\fj^{-1}]$ and $\varpi^{k'} L'\subset L''[\fj^{-1}]$ and hence $\varpi^{k+k'}L\subset L''[\fj^{-1}]$ proving \ref{item:inequality-in-chain}.

We prove \ref{item:frobenius-index}. By Lemma \ref{lem:intersection-projective}, $[L:T]$ is equivalently the smallest integer $k$ such that $\varpi^k L\subset T$. In particular, $\varpi^{qk} \tau_M(\tau^*L)\subset \tau_M(\tau^*T)$ and thus
\[
[\tau_M(\tau^*L):\tau_M(\tau^*L)]\leq qk.
\]
If this inequality was strict, we would have have $\varpi^{qk}\tau_M(\tau^*L)\subset \varpi \tau_M(\tau^*T)[\fj^{-1}]$, hence:
\begin{align*}
\tau_M(\tau^*(\varpi^k L))=\varpi^{qk}\tau_M(\tau^*L) &\subset \left(\varpi \tau_M(\tau^*T)[\fj^{-1}]\right)\cap \tau_M(\tau^*M) \\
&=\varpi \left(\tau_M(\tau^*T)[\fj^{-1}]\cap \tau_M(\tau^*M) \right) \\
&=\varpi \tau_M((\tau^*T)[\fj^{-1}]\cap \tau^*M)\\
&=\varpi \tau_M(\tau^* T) \\
&\subset \tau_M(\tau^* (T\otimes_{\cO_E} \varpi^{1/q}\cO_{E^{1/q}}))
\end{align*}
where, to pass from the first to the second line we used that $\tau_M(\tau^* M)$ is an $E$-vector space, from second to third that $\tau_M$ is injective and from third to fourth we applied Lemma \ref{lem:intersection-projective} to the modules $\tau^* T\subset \tau^*M$. Above, we denoted by $E^{1/q}$ the field of $q$-roots of elements of $E$ and by $\cO_{E^{1/q}}$ its ring of integers. Using the flatness of $T$, we obtain:
\[
\varpi^k L\subset (T\otimes_{\cO_E}\varpi^{1/q}\cO_{E^{1/q}})\cap T=\varpi T
\]
which contradicts $[L:T]=k$. 

To conclude, we prove \ref{item:index-bounded}:
\begin{align*}
q[L:T]&=[\tau_M(\tau^* L):\tau_M(\tau^* T)] \quad(\text{by~\ref{item:frobenius-index}}) \\
&\leq [\tau_M(\tau^* L):L]+[L:\tau_M(\tau^* T)] \quad(\text{by~\ref{item:inequality-in-chain}}) \\
&\leq [L:\tau_M(\tau^* T)] \quad(\text{by~\ref{item:negative}}) \\
&\leq [L:T]+[T:\tau_M(\tau^* T)] \quad(\text{by~\ref{item:inequality-in-chain}}).
\end{align*}
and therefore $(q-1)[L:T]\leq  [T:\tau_M(\tau^* T)]$. 
\end{proof}

Let $\underline{M}$ be an $A$-motive over $E$ of characteristic morphism $\kappa$. We are ready to prove the existence of the maximal $\cO_E$-model for $\underline{M}$.
\begin{Proposition}\label{prop-existence-of-maximal-integral-model-local}
A maximal $\cO_{E}$-model for $\underline{M}$ exists and is unique. In particular, a maximal $\cO_E$-good sublattice for $\underline{M}$ exists and is unique.
\end{Proposition}
\begin{proof}
It is enough to show existence of a maximal integral model. Fix a flat and generating sub-$A\otimes \cO_E$-module\footnote{Such a $T$ exists: as $M$ is finite projective over $A\otimes E$, it is a direct summand of free $A\otimes E$-module, say $(A\otimes E)^s$. Then one could take $T=(A\otimes \cO_E)^s\cap M$.} $T$ of $M$. Denote by $c_T$ the integer
\[
c_T:=\left\lfloor-\frac{[T:\tau_M(\tau^*T)]}{q-1} \right\rfloor.
\]
Let $U$ be the $A\otimes \cO_{E}$-module given by the sum over all the $\cO_{E}$-models for $\underline{M}$. As $U$ is non-empty by Proposition \ref{prop:existence-of-S-models}, it generates $M$ over $E$. We also have $\tau_M(\tau^*U)\subset U[\fj^{-1}]$. By Lemma \ref{lem:properties-index}\ref{item:index-bounded}, any $\cO_{E}$-model $L$ of $\underline{M}$ is a submodule of $\varpi^{c_T}T$ and thus $U\subset \varpi^{c_T}T$. Since $A\otimes \cO_E$ is Noetherian, $U$ is finitely generated. Hence $U$ is a model of $\underline{M}$ which is maximal by construction. 
\end{proof}

\begin{Definition}
We denote by $M_{\cO}$ the unique maximal $\cO_E$-integral model of $\underline{M}$, and by $M_{\text{good}}$ the maximal $\cO_E$-good sublattice of $\underline{M}$.
\end{Definition}

As announced earlier, maximal $\cO_E$-models are projective as we show next:
\begin{Theorem}\label{thm:integral-models-are-locally-free-local}
Both $M_{\cO}$ and $M_{\operatorname{good}}$ are projective over $A\otimes \cO_{E}$.
\end{Theorem}

This will result from the following general statement:
\begin{Proposition}\label{prop:intersection-equivalent-to-flat}
Let $N$ be a finite sub-$A\otimes \cO_E$-module of a projective $A\otimes E$-module $M$ which generates $M$ over $E$. We claim that the following are equivalent:
\begin{enumerate}[label=$(\roman*)$]
\item\label{item:intersection} For all $a\in A$, the equality $N\cap aM=aN$ holds inside $M$, and
\item\label{item:flat} the module $N$ is projective.
\end{enumerate}
\end{Proposition}

Before proving the proposition, we recall a useful fact from dimension theory which generalizes the well-known statement that finite torsion-free modules over Dedekind domains are projective:
\begin{Lemma}\label{lem:Krull-and-Proj-dim}
Let $R$ be a Noetherian regular integral domain of Krull dimension $n$. Let $M$ be a finite torsion-free $R$-module. Then, the projective dimension of $M$ is $<n$.
\end{Lemma}
\begin{proof}
We know that the projective dimension $\operatorname{pd}_R(M)$ of $M$ is $\leq n$. Suppose it is $n$. Then, there exists a maximal ideal $\fm$ of $R$ such that $\operatorname{pd}_{R_{\fm}}(M_{\fm})=n$. By the Auslander--Buchsbaum formula \cite[Thm. 19.9]{eisenbud}, we have:
\[
\operatorname{depth}(M_\fm)=\operatorname{depth}(R_\fm)-n\leq 0
\]
and then $\operatorname{depth}(M_\fm)=0$. This means that $M_\fm$ is torsion, a contradiction since $M$ is not.
\end{proof}

\begin{proof}[Proof of Proposition \ref{prop:intersection-equivalent-to-flat}]
That \ref{item:flat} implies \ref{item:intersection} corresponds to Lemma \ref{lem:intersection-projective}, point \ref{item:multiplication-by-a-intersection}.

We focus on the converse. It is enough to show that $N$ is flat, which is equivalent to the injectivity of the map $\mu:I\otimes_{A\otimes \cO_E} N\to IN$ for all ideal $I$ of $A\otimes \cO_E$ (\emph{cf}. \cite[00HD]{stack}). As $M$ is flat over $A\otimes E$, we have $IM\cong I\otimes_{A\otimes \cO_E}M$ and the injectivity of $\mu$ is equivalent to that of $I\otimes_{A\otimes \cO_E} N\to I\otimes_{A\otimes \cO_E} M$. In that respect, it is sufficient to show the vanishing of the module:
\begin{equation}\label{eq:Tor}
\Tor_1^{A\otimes \cO_E}(I,M/N).
\end{equation}
We shall derive it by the following pair of lemmas:
\begin{Lemma}\label{lem:no-torsion-in-Tor}
Let $I$ be an ideal of $A\otimes \cO_E$ and let $P$ be an $A\otimes \cO_E$-module with no $A$-torsion. Then $\Tor^{A\otimes \cO_E}_1(I,P)$ has no $A$-torsion.
\end{Lemma}
\begin{proof}
As $A\otimes \cO_E$ has Krull dimension $2$ and $I$ is a finite torsion-free $A\otimes \cO_E$-module, Lemma \ref{lem:Krull-and-Proj-dim} applies to show that $I$ has projective dimension $\leq 1$. By the dimension theorem \cite[Prop. 4.15]{weibel}, projective and Tor dimension coincide, hence $\Tor_2^{A\otimes \cO_E}(I,-)=(0)$. On the other-hand, given $a\in A$ non zero, the sequence $0\to P\xrightarrow{a}P \to P/aP \to 0$ is exact. The induced Tor-sequence:
\[
(0)=\Tor_2^{A\otimes \cO_E}(I,P/aP) \to \Tor_1^{A\otimes \cO_E}(I,P) \xrightarrow{r} \Tor_1^{A\otimes \cO_E}(I,P)
\]
implies that multiplication-by-$a$ on $\Tor_1^{A\otimes \cO_E}(I,P)$ is injective; \emph{i.e.} that it has no $a$-torsion. This concludes the proof.
\end{proof}

\begin{Lemma}\label{lem:Tor-has-torsion}
For any $A\otimes \cO_E$-module $P$, we have $K\otimes_A\Tor_1^{A\otimes \cO_E}(I,P)=(0)$. 
\end{Lemma}
\begin{proof}
As localization commutes with Tor, we have 
\[
K\otimes_A\Tor_1^{A\otimes \cO_E}(I,P) \cong \Tor_1^{K\otimes \cO_E}(K\otimes_A I,K\otimes_A P).
\]
But $K\otimes \cO_E$ is a Dedekind domain, thus $K\otimes_A I$ is projective and the right-hand side is zero.
\end{proof}

We conclude the proof by showing that \eqref{eq:Tor} is zero. Observe that \ref{item:intersection} implies that $M/N$ has no $A$-torsion. In particular Lemma \ref{lem:no-torsion-in-Tor} implies that \eqref{eq:Tor} has no torsion. Lemma \ref{lem:Tor-has-torsion} then implies that \eqref{eq:Tor} vanishes. 
\end{proof}

To deduce Theorem \ref{thm:integral-models-are-locally-free-local} from Proposition \ref{prop:intersection-equivalent-to-flat}, we are left to show:
\begin{Lemma}\label{lem:MO-intersects-aM-local}
Let $\fa \subset A$ be an ideal. Then $M_{\cO}\cap \fa M=\fa M_{\cO}$ and $M_{\operatorname{good}}\cap \fa M=\fa M_{\operatorname{good}}$.
\end{Lemma}
\begin{proof}
The proofs being similar, we solely explicit it in the case of $M_{\cO}$. The inclusion $\supset$ is clear. We assume $\fa\neq 0$ and consider the sub-$A$-motive $(\fa M,\tau_M)$ of $\underline{M}$. If $T$ is an $\cO_{E}$-model for $(\fa M,\tau_M)$, then $\fa^{-1}T$ is an $\cO_{E}$-model for $\underline{M}$ and we have $\fa^{-1}T\subset M_{\cO}$. This implies that $\fa M_{\cO}$ is the maximal $\cO_{E}$-model of $(\fa M,\tau_M)$ so that $(\fa M)_{\cO}=\fa (M_{\cO})$. Therefore, the inclusion $M_{\cO}\cap \fa M\subset \fa M_{\cO}$ follows from the fact that $M_{\cO}\cap \fa M$ is an $\cO_{E}$-model for $(\fa M,\tau_M)$.
\end{proof}

\begin{proof}[Proof of Theorem \ref{thm:integral-models-are-locally-free-local}]
For all $a\in A$, we have $M_{\cO}\cap aM=a M_{\cO}$ (resp. $M_{\operatorname{good}}\cap a M=a M_{\operatorname{good}}$) by Lemma \ref{lem:MO-intersects-aM-local}. Since $M_{\cO}$ generates $M$ over $E$ (resp. $M_{\operatorname{good}}$ generates the projective $A\otimes E$-module $M_{\operatorname{good}}\otimes_{\cO_E}E$), we deduce from Proposition \ref{prop:intersection-equivalent-to-flat} that $M_{\cO}$ (resp. $M_{\operatorname{good}}$) is projective.
\end{proof}

\begin{Remark}
Let $\underline{M}$ and $\underline{N}$ be two $A$-motives over $E$, and let $M_{\cO}$ and $N_{\cO}$ be their respective integral models. While the maximal integral model of $\underline{M}\oplus \underline{N}$ is easily shown to be ${M_{\cO}\oplus N_{\cO}}$, it is not true in general that the maximal integral model of $\underline{M}\otimes \underline{N}$ is the image of $M_{\cO}\otimes_{A\otimes \cO_E} N_{\cO}$ in $M\otimes_{A\otimes E} N$. To find a counter-example, we assume $q>2$ and consider $\varpi\in \cO_E$ a uniformizer. We consider the $A$-motive $\underline{M}$ over $E$ where $M=A\otimes E$ and where $\tau_M=\varpi\cdot \mathbf{1}$. The maximal integral model of $\underline{M}$ is $M_{\cO}=A\otimes \cO_E$. However, $\underline{M}^{\otimes (q-1)}$ has $\varpi^{-1}M_{\cO}^{\otimes (q-1)}$ for maximal integral model.  
\end{Remark}

\subsubsection*{Comparison with Frobenius spaces}
As in Subsection \ref{subsec:integral-models-of-frobenius-modules}, $E$ is the fraction field of a discrete valuation ring $(\cO_E,v_E)$ with maximal ideal $\fp=\fp_E$. We assume that $\kappa:A\to \cO_E$ is such that $\kappa^{-1}(\fp)$ admits a non zero element. This implies in particular that $\kappa$ is injective and hence that $\fj (A/\fa\otimes E)=A/\fa\otimes E$ for all non zero ideal $\fa\subset A$. We fix a maximal ideal $\ell$ of $A$. The assumption on $\kappa$ implies $\fj (A/\ell^n\otimes E)=A/\ell^n\otimes E$ for all positive integers $n$.\\

Let $\underline{M}$ be an $A$-motive over $E$. We have canonical isomorphisms
\begin{equation}\label{eq:modulo-m^n-j}
\text{for~all~} n\geq 1:\quad M/\ell^n M \cong M[\fj^{-1}]/\ell^n M[\fj^{-1}]. 
\end{equation}
In particular, $\tau_M$ defines an $A\otimes E$-linear morphism $\tau^*(M/\ell^nM)\to M/\ell^n M$ through the composition
\begin{equation}
\tau^*(M/\ell^nM)\stackrel{\tau_M}{\longrightarrow} M[\fj^{-1}]/\ell^n M[\fj^{-1}]\stackrel{\eqref{eq:modulo-m^n-j}}{\longrightarrow}M/\ell^n M \nonumber
\end{equation}
which we still denote by $\tau_M$. The pair $(M/\ell^n M,\tau_M)$ thusly defined is a Frobenius space over $E$ in the sense of Section \ref{subsec:integral-models-of-frobenius-modules} which we denote by $\underline{M}/\ell^n \underline{M}$. Let $L_n\subset M/\ell^nM$ be its maximal integral model.

\begin{Remark}\label{rmk:frobenius-module-comparison}
We caution the reader, that in general we cannot claim equality between $(M_{\cO}+\ell^nM)/\ell^n M$ and $L_n$. Here is a counter-example.

Suppose that $A=\bF[t]$, so that $A\otimes \cO_E$ is identified with $\cO_E[t]$, and let $\ell=(t)$. Let $\kappa:A\to \cO_E$ be the $\bF$-algebra morphism which maps $t$ to $\varpi$. In this setting, $\fj$ is the principal ideal of $\cO_E[t]$ generated by $(t-\varpi)$. Consider the $A$-motive $\underline{M}:=(E[t],f\cdot \mathbf{1})$ over $E$ where $f=\varpi^{q-1}-\varpi^{q-2}t$. We claim that the maximal integral model of $\underline{M}$ is $\cO_E[t]$. Clearly, $\cO_E[t]$ is an integral model for $\underline{M}$ so that $\cO_E[t]\subset M_{\cO}$. Conversely, by \cite[Thm. 4]{quillen-projective}, $M_{\cO}$ is free of rank one over $\cO_E[t]$. If $h$ generates $M_{\cO}$, there exists $b\in \cO_E[t]$ such that $fh^{(1)}=bh$. For $p\in E[t]$, let $v(p)$ be the infinimum of the valuations of the coefficients of $p$. We have 
\begin{equation}
v(h)\geq -\frac{v(f)}{q-1}=-\frac{q-2}{q-1}>-1 \nonumber
\end{equation}
and $h\in \cO_E[t]$. We get $M_{\cO}\subset \cO_E[t]$.

On the other-hand, the Frobenius space $(M/\ell M,\tau_M)$ is isomorphic to the pair $(\cO_E,\varpi^{q-1}\mathbf{1})$, whose maximal integral model is $\varpi^{-1}\cO_E$, not $\cO_E$.
\end{Remark}

If one wants to compare $M_{\cO}$ with $(L_n)_{n\geq 1}$, then one wishes that $(M_{\cO}+\ell^nM)/\ell^nM$ defines an integral model for $\underline{M}/\ell^n \underline{M}$ for all $n\geq 1$. This is the case in Remark \ref{rmk:frobenius-module-comparison}, although it is not maximal, because the considered $A$-motive $\underline{M}$ is effective. In general, this is not true\footnote{For instance, consider the $t$-motive $(E[t],(t-\varpi)^{-1}\mathbf{1})$ over $E$, whose maximal $\cO_E$-model is $\cO_E[t]$, together with $\ell=(t)$.}. From now on, we assume
\begin{enumerate}[label=$(C_{\ell})$]
\item\label{assumption:ideal-m} The ideal $\ell\subset A$ is such that $\kappa(\ell)\cO_E=\cO_E$.
\end{enumerate}
The above assumption ensures that $\fj (A/\ell^n\otimes \cO_E)=A/\ell^n\otimes \cO_E$ (\emph{e.g.} the proof Proposition \ref{prop-fj-invertible-isocrystal}), and thus that $(M_{\cO}+\ell^nM)/\ell^nM$ is an integral model for $(M/\ell^nM,\tau_M)$.
\begin{Remark}\label{rmk:assumption-always-sat}
Note that there always exists a maximal ideal $\ell$ in $A$ satisfying \ref{assumption:ideal-m}; in fact, any maximal ideal $\ell$ coprime to $\kappa^{-1}(\fp)$ will do.
\end{Remark}
Even though we cannot claim always equality between $(M_{\cO}+\ell^nM)/\ell^n M$ and $L_n$, the data of $L_n$ for all $n\geq 1$ is enough to recover $M_{\cO}$ as we show in the next theorem.
\begin{Theorem}\label{thm:integral-model:motive-to-frobenius-module}
Let $L_n$ be the maximal integral model of the Frobenius space $\underline{M}/\ell^n \underline{M}$. Let $m\in M$. Then $m\in M_{\cO}$ if and only if $m+\ell^nM\in L_n$ for all large enough positive integers $n$. 
\end{Theorem}
We start with some lemmas:
\begin{Lemma}\label{lem:Ln-module}
The $\cO_{E}$-module $L_n$ is an $A/\ell^n\otimes \cO_{E}$-module.
\end{Lemma}
\begin{proof}
For an elementary tensor $r\otimes f$ in $A/\ell^n\otimes \cO_E$, the $\cO_E$-module $(r\otimes f) L_n$ is stable by $\tau_M$. Indeed, we have $\tau_M(\tau^*(r\otimes f) L_n)=(r\otimes f^q)\tau_M(\tau^*L_n)\subset (r\otimes f)L_n$. By maximality of $L_n$, we have $(r\otimes f)L_n\subset L_n$.
\end{proof}

Recall the notion of \emph{range} from Definition \ref{def:range}.
\begin{Lemma}\label{lem:Mn-is-bounded}
Let $r_n$ be the range of the $\cO_{E}$-lattice $(M_{\cO}+\ell^nM)/\ell^nM$ in $M/\ell^n M$. Then $(r_n)_{n\geq 1}$ is bounded.
\end{Lemma}

\begin{proof}
Let $e$ and $f$ be large enough positive integers such that $\tau_M(\tau^*M)\subset \fj^{-e}M$ and $M\subset\fj^{-f}\tau_M(\tau^*M)$. For this choice of $e$, there is an exact sequence of $A\otimes \cO_E$-modules
\begin{equation}\label{eq:discriminant}
0\longrightarrow \tau^*M_{\cO} \xrightarrow{\tau_M} \fj^{-e}M_{\cO}\longrightarrow C_e \longrightarrow 0
\end{equation}
where we denoted by $C_e$ the cokernel of the former map. We claim that there exists an integer $k\geq 0$ such that $C_e$ is annihilated by $(1\otimes\varpi^k) \fj^f$. Indeed, one could take $k=[M_{\cO}:\tau_M(\tau^*M_{\cO})]$ (as in notations following Lemma \ref{lem:existence-of-index}), as then:
\[
\varpi^k M_{\cO}\subset \tau_M(\tau^*M_{\cO})[\fj^{-1}]\cap \fj^{-f} \tau_M(\tau^*M)=\tau_M((\tau^*M_{\cO})[\fj^{-1}]\cap \fj^{-f} \tau^*M)=\fj^{-f}\tau_M(\tau^*M_{\cO})
\]
where, for the last inequality, we used Lemma \ref{lem:intersection-projective} with respect to the flat submodule $\fj^{-f}\tau^*M_{\cO}$ of $\fj^{-f}\tau^*M$ (where flatness follows from Theorem \ref{thm:integral-models-are-locally-free-local}).

The proof of the lemma will follow from the inequality $r_n\leq k$ for all $n\geq 1$. To prove the latter, note that the range $r_n$ is the smallest integer $r\geq 0$ such that $\varpi^r$ annihilates the cokernel of 
\begin{equation}\label{eq:tau_m-mod-elln}
\tau^*(M_{\cO}/\ell^nM_\cO)\xrightarrow{\tau_M} M_{\cO}/\ell^n M_{\cO}
\end{equation}
(we implicitly used $M_{\cO}/\ell^n M_{\cO}\cong (M_{\cO}+\ell^n M)/\ell^n M$, following Lemma \ref{lem:MO-intersects-aM-local}). Yet, using the right-exactness of $-\otimes_A A/\ell^n$ applied to the sequence \eqref{eq:discriminant}, together with the fact that $\fj$ is invertible modulo $\ell^n$ (by our assumption \ref{assumption:ideal-m}), the cokernel of \eqref{eq:tau_m-mod-elln} is $C_e/\ell^n C_e$. Hence, it is annihilated by $\varpi^k$ and then $r_n\leq k$.
\end{proof}

For $n\geq 1$, let $\tilde{L}_n$ be the inverse image in $M$ of $L_n\subset M/\ell^n M$.
\begin{proof}[Proof of Theorem \ref{thm:integral-model:motive-to-frobenius-module}]
First observe that the sequence of subsets $(\tilde{L}_n+\ell^nM)_{n\geq 1}$ of $M$ decreases when $n$ grows: for $n\geq 1$, we have $\tilde{L}_{n+1}+\ell^{n+1}M\subset \tilde{L}_{n+1}+\ell^{n}M$ and, because $(\tilde{L}_{n+1}+\ell^nM)/\ell^n M$ defines an integral model for $(M/\ell^n M,\tau_M)$, we also have 
\[
\tilde{L}_{n+1}+\ell^{n}M\subset \tilde{L}_n+\ell^n M.
\]
Hence, the statement of the theorem is equivalent to the equality
\begin{equation}
M_{\cO}=L, \quad \text{where}~L:=\bigcap_{n=1}^{\infty}{(\tilde{L}_n+\ell^nM)}. \nonumber
\end{equation}
By Lemma \ref{lem:Ln-module}, $L$ is an $A\otimes \cO_{E}$-module. The inclusion $M_{\cO}\subset L$ follows from the fact that, for all $n$, $(M_{\cO}+\ell^nM)/\ell^nM$ is an integral model for $(M/\ell^nM,\tau_M)$. To prove the converse inclusion, we show that $L$ is an integral model for $\underline{M}$. From $M_{\cO}\subset L$, one deduces that $L$ generates $M$ over $E$. Because $\tau_M(\tau^*(\tilde{L}_n+\ell^n M))\subset \tilde{L}_n+\ell^n M[\fj^{-1}]$, we also have $\tau_M(\tau^*L)\subset L[\fj^{-1}]$. The theorem follows once we have proved that $L$ is finitely generated. 

Assume that $L$ is not finitely generated. From the Noetherianity of $A\otimes \cO_{E}$, for all $s\geq 0$, it follows that $L\not\subset \varpi^{-s}M_{\cO}$. Equivalently, there exists an unbounded increasing sequence $(s_n)_{n\geq 0}$ of non-negative integers such that $\varpi^{s_n}L_n\not\subset (M_{\cO}+\ell^nM)/\ell^nM$. By Proposition \ref{prop-bound-integral-lattice}, the range of $(M_{\cO}+\ell^nM)/\ell^nM$ is $>s_n(q-1)$. But this contradicts Lemma \ref{lem:Mn-is-bounded}.
\end{proof}

In the remainder of this subsection, we still assume \ref{assumption:ideal-m}. We record a useful corollary to Theorem \ref{thm:integral-model:motive-to-frobenius-module}.

\begin{Corollary}\label{Cor:good-reduction}
Let $\underline{N}=(N,\tau_N)$ be an $A$-motive over $\cO_E$ and $\underline{N}_E$ its base change to $E$. Then, $N=(N_E)_{\cO}=(N_E)_{\operatorname{good}}$.
\end{Corollary}
\begin{proof}
For all $n\geq 1$, $L_n:=(N+\ell^nN_E)/\ell^nN_E$ is an integral model of the Frobenius space $(N_E/\ell^n N_E,\tau_N)$ which further satisfies $\tau_M(\tau^*L_n)=L_n$. Hence, $(N_E/\ell^n N_E,\tau_N)$ has good reduction with $L_n$ as the maximal integral model (\emph{e.g.} Lemma \ref{lem-converse-inclusion-contains-L}). Theorem \ref{thm:integral-model:motive-to-frobenius-module} implies that the maximal integral model of $\underline{N}_E$ is the $\ell$-adic (topological) closure of $N$ in $N_E$. Yet, as $N$ is a finite projective $A\otimes \cO_E$-module, it is $\ell$-adically closed in $N_E$ (one reduces--considering $N$ as a direct summand in a finite free $A\otimes \cO_E$-module and similarly for $N_E$ by base change along $\cO_E\to E$--to  the question whether $A\otimes \cO_E$ is $\ell$-adically closed in $A\otimes E$, which is clear). 
\end{proof}

There is a version of Theorem \ref{thm:integral-model:motive-to-frobenius-module} for the maximal good model. The proof is along the same lines, so we do not repeat it.
\begin{Proposition}\label{prop:From-good-models-to-Frob-spaces}
Let $T_n$ be the maximal good sublattice of the Frobenius space $\underline{M}/\ell^n \underline{M}$. Let $m\in M$. Then $m\in M_{\operatorname{good}}$ if and only if $m+\ell^nM\in T_n$ for all large enough positive integers $n$.
\end{Proposition}

For the next section, we shall not only be interested in how to recover $M_{\cO}$ from $L_n$, but also in how to recover $M_{\cO}+(\id-\tau_M)(M)$. While we do not give a complete answer, we at least show next how to recover its $\ell$-adic closure in $M$ in the case $\underline{M}$ effective. We continue with some finer technicalities.\\

Even if we do not have equality between $\tilde{L}_n+\ell^n M$ and $M_{\cO}+\ell^n M$, the former is a good approximation of the latter as we show next.
\begin{Lemma}\label{lem:sequence-Lm-is-bounded-below}
Let $n\geq 1$. The sequence $(\tilde{L}_m+\ell^nM)_{m\geq n}$ is decreasing for the inclusion, stationary and converges to $M_{\cO}+\ell^nM$.
\end{Lemma}
\begin{proof}
Let $m\geq 1$. $(\tilde{L}_{m+1}+\ell^{m}M)/\ell^m M$ is an $\cO_E$-lattice stable by $\tau_M$ in $M/\ell^mM$ so that $\tilde{L}_{m+1}+\ell^mM\subset \tilde{L}_m+\ell^mM$. If $m\geq n$, we have $\tilde{L}_{m+1}+\ell^nM\subset \tilde{L}_m+\ell^nM$ which shows that $(\tilde{L}_m+\ell^nM)_{m\geq n}$ decreases. Similarly, $M_{\cO}+\ell^n M\subset \tilde{L}_m+\ell^n M$ for all $m\geq n$. Because the set of $\cO_E$-lattices $\Lambda$ such that $M_{\cO}+\ell^nM\subseteq \Lambda \subseteq \tilde{L}_n+\ell^nM$ is finite, the sequence $(\tilde{L}_m+\ell^nM)_{m\geq n}$ is stationary. We denote by $\mathcal{L}_n$ its limit. By Theorem \ref{thm:integral-model:motive-to-frobenius-module}, we have
\begin{equation}
\mathcal{L}_{n}=\bigcap_{m=n}^{\infty}(\tilde{L}_m+\ell^{n}M)=\bigcap_{m=n}^{\infty}(\tilde{L}_m+\ell^{m}M)+\ell^n M=M_{\cO}+\ell^n M. \nonumber
\end{equation}
This concludes the proof.
\end{proof}

\begin{Lemma}\label{lem:sequence-km-unbounded}
There exists an unbounded and increasing sequence $(k_n)_{n\geq 1}$ of non-negative integers such that, $\tilde{L}_n+\ell^nM\subset M_{\cO}+\ell^{k_n}M$ (typically, $k_n\leq n$ for all $n$).
\end{Lemma}
\begin{proof}
For $m\geq 1$, let $I_m$ be the set of non-negative integers $k$ such that $\tilde{L}_m+\ell^mM\subset M_{\cO}+\ell^kM$. $I_m$ is nonempty as it contains $0$. $I_m$ is further bounded: otherwise we would have 
\begin{equation}
\tilde{L}_{m}+\ell^mM\subset \bigcap_k (M_{\cO}+\ell^k M)=M_{\cO},
\end{equation}
where the last equality is due to the fact that $M_{\cO}$ is $\ell$-adically closed in $M$ (by the same argument given in the proof of Corollary \ref{Cor:good-reduction}, using that $M_\cO$ is finite projective by Theorem \ref{thm:integral-models-are-locally-free-local}). Yet, this is impossible as $\tilde{L}_{m}+\ell^mM$ is an $A\otimes \cO_E$-module which is not of finite type. Hence $I_m$ has a maximal element, which we denote by $k_m$. Because $\tilde{L}_{m+1}+\ell^{m+1}M\subset \tilde{L}_m+\ell^mM$, we have $k_{m+1}\geq k_m$. This shows that $(k_m)_{m\geq 1}$ increases. We show that it is unbounded. Let $n\geq 1$. By Lemma \ref{lem:sequence-Lm-is-bounded-below}, there exists $m\geq n$ such that $M_{\cO}+\ell^n M=\tilde{L}_m+\ell^nM$. Thus $\tilde{L}_m+\ell^mM\subset M_{\cO}+\ell^nM$. In particular, there exists $m\geq n$ such that $k_m\geq n$. 
\end{proof}

\begin{Proposition}\label{prop:variation-integral-model:motive-to-frobenius}
Assume that $\underline{M}$ is effective, and let $m\in M$. The following are equivalent:
\begin{enumerate}[label=$(\roman*)$]
\item\label{item:aaa} $m$ belongs to the $\ell$-adic closure of $M_{\cO}+(\id-\tau_M)(M)$ in $M$,
\item\label{item:bbb} for all $n\geq 1$, $m\in \tilde{L}_n+(\id-\tau_M)(M)+\ell^nM$.
\end{enumerate}
\end{Proposition}
\begin{proof}
By Theorem \ref{thm:integral-model:motive-to-frobenius-module}, the inclusion 
\begin{equation}
M_{\cO}+(\id-\tau_M)(M)\subset\bigcap_{n=1}^{\infty}\left[\tilde{L}_n+(\id-\tau_M)(M)+\ell^nM\right] \nonumber
\end{equation}
holds as subsets of $M$, and the right-hand side is $\ell$-adically complete. Hence \ref{item:aaa} implies \ref{item:bbb}. The converse follows from Lemma \ref{lem:sequence-km-unbounded}:
\[
\bigcap_{n=1}^{\infty}\left[\tilde{L}_n+(\id-\tau_M)(M)+\ell^nM\right] \subset \bigcap_{n=1}^{\infty}\left[M_{\cO}+(\id-\tau_M)(M)+\ell^{k_n}M\right] \]
where the right-hand side is identified with the $\ell$-adic completion of $M_{\cO}+(\id-\tau_M)(M)$.
\end{proof}

\subsubsection*{N\'eron-Ogg-Shafarevi\v{c}-type criterion}\label{subsec:good-reduction}
Let $E$ be a local function field with finite residue field $k$ and valuation ring $\cO_E$ with maximal ideal $\fp=\fp_E$. This paragraph is an \emph{apart\'e} offering a good reduction criterion, very much in the spirit of Gardeyn's N\'eron-Ogg-Shafarevi\v{c}-type criterion \cite[Thm. 1.1]{gardeyn1}. Proposition \ref{prop-good-reduction-local-case} below is not needed in the sequel, although it is useful in examples to compute maximal models. \\

Consider a characteristic morphism $\kappa:A\to \cO_E$ (\emph{i.e.} an $\bF$-algebra morphism) such that $\kappa^{-1}(\fp_E)\neq (0)$. 
\begin{Proposition}\label{prop-good-reduction-local-case}
Let $\underline{M}$ be an $A$-motive over $E$, and let $\ell$ be a maximal ideal of $A$ such that $\kappa(\ell)\cO_E=\cO_E$. The following statements are equivalent:
\begin{enumerate}[label=$(\roman*)$]
\item\label{item:good-model} There exists an $A$-motive $\underline{N}$ over $\cO_E$ such that $\underline{N}_{E}$ is isomorphic to $\underline{M}$.
\item\label{item:Mgood=MO} The inclusion $M_{\operatorname{good}}\subset M_{\cO}$ is an equality.
\item\label{item:NOS} The representation $\operatorname{T}_{\ell}\underline{M}$ is unramified.
\end{enumerate}
\end{Proposition}
\begin{proof}
The equivalence between \ref{item:good-model} and \ref{item:Mgood=MO} follows from Theorem \ref{thm:integral-model:motive-to-frobenius-module} and Proposition \ref{prop:From-good-models-to-Frob-spaces}. Let $\underline{M}_n$ denote the Frobenius space $(M/\ell^nM,\tau_M)$ of the previous section, and let $L_n$ and $T_n$ be its maximal integral and good sublattice respectively. The equivalence between \ref{item:Mgood=MO} and \ref{item:NOS} follows from the following sequence of equivalent statements:
\begin{align*}
\operatorname{T}_{\ell}\underline{M}~ \text{is~unramified} &\Longleftrightarrow \forall n\geq 1,~ T\underline{M}_n~ \text{is~unramified} \quad (\text{notations~of~Thm.}~\ref{thm:katz-equivalence}) \\
&\Longleftrightarrow \forall n\geq 1,~T_n=L_n \quad (\text{by~Prop.}~\ref{prop:babyNOS})\\
&\Longleftrightarrow M_{\text{good}}=M_{\cO} \quad (\text{by~Thm.~\ref{thm:integral-model:motive-to-frobenius-module}~and~Prop.~\ref{prop:From-good-models-to-Frob-spaces}}).
\end{align*}
\end{proof}

\begin{Definition}
We say that $\underline{M}$ has \emph{good reduction} if one of the equivalent points of Proposition \ref{prop-good-reduction-local-case} is satisfied.
\end{Definition}

\subsection{Models of $A$-motives over a global function field}\label{sec:integral-models-A-motives-global}
Let $R$ be a Dedekind domain which is an $A$-algebra via a morphism $\kappa$. We assume that the only ideal mapped to the zero ideal under the induced map $\kappa^*:\Spec R\to \Spec A$ is the zero ideal itself; equivalently, for all prime ideal $\fp\subset R$, we have $\kappa^{-1}(\fp)=(0)$ if and only if $\fp=(0)$. In particular, the results of Subsection \ref{sec:integral-models-A-motives-local} holds for $\cO_E$ being either the localization $R_{{(\fp)}}$ of $R$ at the multiplicative subset $A\setminus \fp$ of $A$ or its completion $R_\fp$. Let $F$ be the fraction field of $R$, and $F_\fp$ the fraction field of $R_\fp$. \\

In this subsection, we show existence of maximal $R$-model for $A$-motives over $F$ and compare them with maximal local models of Subsection \ref{sec:integral-models-A-motives-local}. Let $\underline{M}=(M,\tau_M)$ be an $A$-motive over $F$, and let $\underline{M}_{F_{\fp}}$ be the $A$-motive over $F_{\fp}$ obtained from $\underline{M}$ by base change from $F$ to $F_{\fp}$. We let $M_{R_{\fp}}$ denote the maximal $R_\fp$-model of $\underline{M}_{F_{\fp}}$ whose existence is ensured by Proposition \ref{prop-existence-of-maximal-integral-model-local}. 

\begin{Proposition}\label{prop-existence-of-maximal-integral-model-global}
There exists a unique maximal $R$-model for $\underline{M}$. It equals the intersection $\bigcap_{\fp}(M\cap M_{R_{\fp}})$ for $\fp$ running over the maximal ideals of $R$. We denote it by $M_{R}$.
\end{Proposition}

\begin{proof}
From Lemma \ref{lem:uniqueness-of-maximal-integral}, uniqueness is automatic. Let $N$ be an $R$-model for $\underline{M}$ (whose existence is ensured by Proposition \ref{prop:existence-of-S-models}). For any maximal ideal $\fp$ of $R$, we have $N\subset N\otimes_{R}R_{\fp}\subset M_{R_{\fp}}$ by maximality of $M_{R_{\fp}}$. Therefore $N\subset \bigcap_{\fp}(M\cap M_{R_{\fp}})$. Hence, it is sufficient to show that $\bigcap_{\fp}(M\cap M_{R_{\fp}})$ is an $R$-model. First note that it is a sub-$A\otimes R$-module of $M$ which, as it contains $N$, generates $M$ over $F$. To show stability by $\tau_M$, let $e$ be a large enough integer such that $\tau_M(\tau^*M)\subset \fj^{-e}M$. One easily checks that:
\begin{equation}
\tau_M\left(\tau^* \bigcap_{\fp}(M\cap M_{R_{\fp}})\right)\subset \bigcap_{\fp}\fj^{-e}(M\cap M_{R_{\fp}})\subset \left(\bigcap_{\fp}M\cap M_{R_{\fp}}\right)[\fj^{-1}]. \nonumber
\end{equation}
It remains to show that $\bigcap_{\fp}(M\cap M_{R_{\fp}})$ is finitely generated over $A\otimes R$. Let $T\subset M$ be a generating finite flat sub-$A\otimes R$-module and let $d\in R$ be such that $d\cdot T\subset \tau_M(\tau^*T)[\fj^{-1}]$. In particular, $T_\fp:=T\otimes_{R}R_\fp$ is a generating finite flat sub-$A\otimes R_\fp$-module of $M_\fp$, and $[T_\fp:\tau_M(\tau^* T_\fp)]\leq v_{\fp}(d)$. We denote ${c_T(\fp):=\lfloor -[T_\fp:\tau_M(\tau^*T_\fp)]/(q-1)\rfloor}$ so that, as in the proof of Proposition \ref{prop-existence-of-maximal-integral-model-local}, $M_{R_\fp}\subset \varpi^{c_T(\fp)}T_\fp$. Note that $c_T(\fp)=0$ for almost all $\fp$. We have
\[
M\cap M_{R_\fp}\subset M\cap \varpi^{c_T(\fp)}T_\fp=(F\otimes_R T)\cap (\fp^{c_T(\fp)}R_{\fp}\otimes_R T)=\fp^{c_T(\fp)}T
\]
where we used the flatness of $T$ for the last equality. Therefore, 
\[
\bigcap_{\fp}(M\cap M_{R_{\fp}})\subset \bigcap_{\fp}\fp^{c_T(\fp)}T=\fd\cdot T,
\]
where $\fd:=\cap_{\fp} \fp^{c_T(\fp)}$ is a fractional ideal of $R$. The module $\fd\cdot T$ being finitely generated, we conclude using the Noetherianity of~$A\otimes R$. 
\end{proof}

From Lemma \ref{lem:existence-of-R-good-models}, we obtain:
\begin{Corollary}
The maximal good sublattice $M_{\operatorname{good}}$ of $\underline{M}$ exists and is unique.
\end{Corollary}

We now state the global version of Theorem \ref{thm:integral-models-are-locally-free-local}. The argument is similar, so we omit proofs.

\begin{Theorem}\label{thm:integral-models-are-locally-free}
Both $M_{R}$ and $M_{\operatorname{good}}$ are projective over $A\otimes R$.
\end{Theorem}
\begin{Remark}
Note, however, that an integral model for $\underline{M}$, when not maximal, is not necessarily projective. For instance, the $\bF[t]$-motive $\mathbbm{1}=(\bF[t](\theta),\mathbf{1})$ over $\bF(\theta)$ admits $L:=t\bF[t,\theta]+\theta\bF[t,\theta]$ as $\bF[\theta]$-model. But it is well-known that $L$ is not a flat $\bF[t,\theta]$-module. A short way to see this consists in considering the element $\Delta:=(t\otimes \theta-\theta\otimes t)\in L\otimes_{\bF[t,\theta]} L$. $\Delta$ is nonzero in $L\otimes_{\bF[t,\theta]} L$, but 
\begin{equation}
\theta\cdot \Delta=(\theta t)\otimes \theta-\theta\otimes (\theta t)=(\theta t)\otimes \theta-(\theta t)\otimes \theta=0. \nonumber
\end{equation}
Then $L$ is not flat because $L\otimes_{\bF[t,\theta]} L$ has non trivial torsion.
\end{Remark}

By the formula given in Proposition \ref{prop-existence-of-maximal-integral-model-global}, the datum of all the maximal local models is enough to recover $M_R$. Conversely, the knowledge of $M_R$ is sufficient to recover $M_{R_\fp}$ as we show next:
\begin{Theorem}\label{thm:from-global-to-local}
The inclusion $M_R\otimes_R R_\fp\subseteq M_{R_\fp}$ is an equality.
\end{Theorem}

We begin with a useful statement on its own which allows to reduce the proof of Theorem \ref{thm:from-global-to-local} to the case where $R$ is a discrete valuation ring. We denote by $(-)_{(\fp)}$ the localization functor with respect to the multiplicative subset $A\setminus \fp$ of $A$. We denote by $M_{R_{(\fp)}}$ the maximal $R_{(\fp)}$-model of $\underline{M}$.
\begin{Lemma}\label{lem:from-Dedekind-to-DVR}
The inclusion $(M_R)_{(\fp)}\subseteq M_{R_{(\fp)}}$ is an equality.
\end{Lemma}
\begin{proof}
As both modules generate $M$ over $F$, the cokernel of the inclusion is $\fp$-torsion and thus, given any $m\in M_{R_{(\fp)}}$, there exists $a\in A$ such that $(a)=\fp^c$ for some $c\geq 1$ and $am\in (M_R)_{(\fp)}$. By definition of localization, there exists $b\in A\setminus \fp$ for which $abm\in M_R$. By Proposition \ref{prop-existence-of-maximal-integral-model-global}, we have $abm\in M\cap M_{R_\fq}$ and then $b m\in M \cap M_{R_\fq}$ for all $\fq\neq \fp$. Hence, 
\[
bm\in M_{R_{(\fp)}}\bigcap_{\fq\neq \fp}(M\cap M_{R_\fq})=(M\cap M_{R_\fp})\bigcap_{\fq\neq \fp}(M\cap M_{R_\fq})=M_R,
\]
and finally $m\in (M_R)_{(\fp)}$. For the first equality, we used that $M_{R_{(\fp)}}=M\cap M_{R_\fp}$, which follows by applying Proposition \ref{prop-existence-of-maximal-integral-model-global} but to the ring $R_{(\fp)}$ instead of $R$. 
\end{proof}

\begin{proof}[Proof of Theorem \ref{thm:from-global-to-local}]
In virtue of Lemma \ref{lem:from-Dedekind-to-DVR}, we may assume that $R$ is a discrete valuation ring with maximal ideal $\fp$ and that $R_\fp$ is its completion. Let $C$ denote the cokernel of $M_R\otimes_R R_\fp \subseteq M_{R_\fp}$; as $C$ is both finite as a module over $A\otimes R_\fp$ and is $\id\otimes \fp$-torsion, it is finite over $A$.

Let $\ell$ be a maximal ideal of $A$ not lying under $\fp$ (\emph{i.e.} such that condition \ref{assumption:ideal-m} holds). Let $(k_n)_n$ and $(k_n')_n$ be the sequences of integers produced by Lemma \ref{lem:sequence-km-unbounded} for $\underline{M}$ and $\underline{M}_{F_\fp}$ respectively. Let $n$ be such that $k_n,k_n'\geq 1$. Let $L$ and $L_\fp$ be the maximal $R$-model and $R_\fp$-models of the Frobenius spaces $\underline{M}/\ell^n\underline{M}$ and $\underline{M}_{F_\fp}/\ell^n\underline{M}_{F_\fp}$. As remarked in \ref{rem:base-change-Frob}, we have $L_{\fp}=L\otimes_R R_\fp$. By choice of $n$, we also have
\begin{align}
&L+\ell M=M_R+\ell M \label{eq:choice-of-k}\\
&L_\fp+\ell M_\fp=M_{R_\fp}+\ell M_{F_\fp} \label{eq:choice-of-k'}
\end{align}
Using $F_\fp=F\otimes_R R_\fp$, we get:
\begin{align*}
M_R\otimes_R R_\fp +\ell M_{F_\fp} &= (M_R+\ell M)\otimes_R R_\fp=(L+\ell M)\otimes_R R_\fp =L\otimes_R R_\fp+\ell M_{F_\fp}=L_\fp+\ell M_{F_\fp} \\
&=M_{R_\fp}+\ell M_{F_\fp}. 
\end{align*}
As a result, we obtain:
\begin{equation}
\begin{tikzcd}
(M_R\otimes_R R_\fp +\ell M_{F_\fp})/\ell M_{F_\fp} \arrow[r,"="] & (M_{R_\fp}+\ell M_{F_\fp})/\ell M_{F_\fp} \\
(M_R\otimes_R R_\fp)/\ell(M_R\otimes_R R_\fp) \arrow[r]\arrow[u,"\wr"] & M_{R_\fp}/\ell M_{R_\fp}\arrow[u,"\wr"']
\end{tikzcd}\nonumber
\end{equation}
That the vertical arrows are isomorphism as depicted on the diagram is deduced from the flatness of the modules $M_R\otimes_R R_\fp$ and $M_{R_\fp}$ over $A\otimes R_\fp$ (\emph{e.g.} as in the beginning of the proof of Lemma \ref{lem:intersection-projective}). This implies that the bottom row is an isomorphism, and hence, $C/\ell C=(0)$. \\
This being true for infinitely many $\ell$, and because $C$ is a finite $A$-module, we obtain $C=(0)$, proving the theorem.
\end{proof}

\begin{Definition}
We say that $\underline{M}$ has \emph{good reduction at $\fp$} if $\underline{M}_{F_\fp}$ has good reduction. We say that $\underline{M}$ has \emph{everywhere good reduction} if $\underline{M}$ has good reduction at $\fp$ for all maximal ideals $\fp$ of $R$.
\end{Definition}

\subsection{The integral part of $A$-motivic cohomology}\label{subsec:integral-part}
In this subsection, we introduce the notion of \emph{integral} and \emph{$\fp$-integral} extension and compare it with that of \emph{good extension at $\fp$ with respect to $\ell$} as defined earlier in Definition \ref{def:ext-good-red-with-resp-to-ell}. More precisely, we shall prove Theorems C and D stated in the introduction. 

\subsubsection*{Over local function fields}
Let $\fm\subset A$ be a maximal ideal of $A$ with associated local function field $K_\fm$. Let $F_\fp$ be a finite field extension of $K_\fm$ with valuation ring $\cO_\fp$. Let $\kappa:A\to\cO_{\fp}$ be the inclusion; in particular, $\kappa^{-1}(\fp)\neq (0)$ and the results of the previous subsections apply. Let $F_{\fp}^{\operatorname{ur}}$ be the maximal unramified extension of $F_{\fp}$ in $F_{\fp}^s$. Let $I_\fp$ be the inertia subgroup of $G_\fp=G_{F_{\fp}}$.  \\

Let $\underline{M}$ be an $A$-motive over $F_{\fp}$ and let $M_{\cO_\fp}$ be its maximal $\cO_\fp$-model.
\begin{Definition}\label{def:integral-part-local}
We define $\Ext^1_{\cO_\fp}(\mathbbm{1},\underline{M})$ as the sub-$A$-module of $\Ext^1_{\cM_{F_{\fp}}}(\mathbbm{1},\underline{M})$ given by the image of $M_{\cO_\fp}[\fj^{-1}]$ through $\iota$ (Theorem \ref{thm:cohomology-in-MR}). Extensions whose class belongs to the latter module will be called \emph{$\fp$-integral} (or simply \emph{integral} when $\fp$ is clear from the context).
\end{Definition}
From Corollary \ref{cor:(id-tau)(M0)=(id-tau)(M)capMO}, $\iota$ induces an isomorphism of $A$-modules:
\[
\frac{M_{\cO_{\fp}}[\fj^{-1}]}{(\id-\tau_M)(M_{\cO_\fp})} \stackrel{\sim}{\longrightarrow} \Ext^1_{\cO_\fp}(\mathbbm{1},\underline{M}).
\]
\begin{Remark}
An important remark is that the assignment $\underline{M}\mapsto \Ext^1_{\cO_{\fp}}(\mathbbm{1},\underline{M})$ is functorial, thanks to Corollary \ref{cor:integral-model-functor}.
\end{Remark}

Our main result states that for $\ell\neq \fm$ integral extensions have good reduction with respect to $\ell$:
\begin{Theorem}\label{thm:main1}
Let $\ell$ be a maximal ideal in $A$ distinct from $\fm$. Then, 
\[
\Ext^1_{\cO_\fp}(\mathbbm{1},\underline{M})\subset \Ext^1_{\operatorname{good}}(\mathbbm{1},\underline{M})_{\ell}.
\] 
\end{Theorem}
\begin{proof}
By choice of $\ell$, we have $\kappa(\ell)\cO_\fp=\cO_\fp$. Let $[\underline{E}]\in \Ext^1_{\cO_\fp}(\mathbbm{1},\underline{M})$. By definition, there exists $m\in M_{\cO_{\fp}}[\fj^{-1}]$ such that $[\underline{E}]=\iota(m)$. If $\tilde{L}_n$ denotes a lift in $M$ of the maximal integral model of the Frobenius space $\underline{M}/\ell^n \underline{M}$, we obtain $m\in \tilde{L}_n+\ell^nM$ for  all $n$. By Proposition \ref{prop:artin-schreier}, there exists $y_n\in M\otimes_{F_{\fp}}\otimes F_{\fp}^{\text{nr}}$ such that
\begin{equation}\label{equation-to-be-lifted}
m\equiv y_n-\tau_M(\tau^*y_n)\pmod{\ell^n}.
\end{equation}
Note that, for each $n$, there are only finitely many such $y_n \pmod{\ell^n}$. We next show that we can choose \emph{compatibly} $y_n$ for all $n$ (that is $y_{n+1}\equiv y_n\pmod{\ell^n}$). Let us define a tree $T$ indexed by $n\geq 1$ whose nodes at the height $n$ are the solutions $y_n$ of \eqref{equation-to-be-lifted} in $(M\otimes_{F_{\fp}}F_{\fp}^{\operatorname{ur}})/\ell^n(M\otimes_{F_{\fp}}F_{\fp}^{\operatorname{ur}})$. There is an edge between $z_{n}$ and $z_{n+1}$ if and only if $z_{n+1}$ coincides with $z_n$ modulo $\ell^n$. The tree has finitely many nodes at each height and it is infinite from the fact that a solution of \eqref{equation-to-be-lifted} exists for all $n$. By K\"onig's Lemma, there exists an infinite branch on $T$. This branch corresponds to a converging sequence $(y_n)_{n\geq 1}$ whose limit $y$ in $(M\hat{\otimes}_{F_{\fp}}F_{\fp}^{\operatorname{ur}})_{\ell}$ satisfies $m=y-\tau_M(\tau^*y)$. Therefore, we conclude that $[\underline{E}]\in \Ext^1_{\text{good}}(\mathbbm{1},\underline{M})_{\ell}$ thanks to Proposition \ref{prop-characterization-of-extension-good}. 
\end{proof}

\subsubsection*{Over global function fields}
Let $F$ be a finite field extension of $K$ and let $\cO_F$ be the integral closure of $A$ in $F$. We let $\kappa:A\to \cO_F$ denote the inclusion. We fix $S$ to be a set of nonzero prime ideals of $\cO_F$ and consider the subring $R:=\cO_F[S^{-1}]$ of $F$. The ring $R$ is a Dedekind domain whose fraction field is $F$. We have $\kappa^{-1}(\fp)=(0)$ if and only if $\fp=(0)$, so the result of Subsection \ref{sec:integral-models-A-motives-global} apply.\\

Let $\underline{M}=(M,\tau_M)$ be an Anderson $A$-motive over $F$. Given a maximal ideal $\fp\subset R$, we let $\underline{M}_{\fp}$ be the $A$-motive over $F_{\fp}$ obtained from $\underline{M}$ by base-change from $F$ to $F_{\fp}$. Given an extension $[\underline{E}]\in \Ext^1_{\cM_F}(\mathbbm{1},\underline{M})$, the exactness of the base change functor defines an extension $[\underline{E}_{\fp}]\in \Ext_{\cM_{F_{\fp}}}^1(\mathbbm{1}_\fp,\underline{M}_{\fp})$. This allows us to define the following submodule of $\Ext^1_{\cM_F}(\mathbbm{1},\underline{M})$:
\[
\Ext^1_{R}(\mathbbm{1},\underline{M})=\bigcap_{\substack{\fp\subset R \\ \text{maximal}}} \left\{[\underline{E}]\in \Ext^1_{\cM_F}(\mathbbm{1},\underline{M})~\bigg |~[\underline{E}_{\fp}]\in \Ext^1_{R_{\fp}}(\mathbbm{1}_{\fp},\underline{M}_{\fp}) \right\}.
\]
\begin{Definition}
We say that an extension of $\mathbbm{1}$ by $\underline{M}$ is \emph{$R$-integral} (or simply \emph{integral}) if it belongs to $\Ext^1_{R}(\mathbbm{1},\underline{M})$.
\end{Definition}

Our second main result consists of the next theorem.
\begin{Theorem}\label{thm:main2}
Let $M_{R}$ denote the maximal integral $R$-model of $\underline{M}$. The $A$-module $\Ext^1_{R}(\mathbbm{1},\underline{M})$ equals the image of $M_{R}[\fj^{-1}]$ through $\iota$. In addition, $\iota$ induces a natural isomorphism of $A$-modules:
\[
\frac{M_{R}[\fj^{-1}]}{(\id-\tau_M)(M_{R})}\stackrel{\sim}{\longrightarrow} \Ext^1_{R}(\mathbbm{1},\underline{M}).
\]
\end{Theorem}

The proof of the above theorem will result after a sequence of lemmas.
\begin{Lemma}\label{lem:decompletion-local}
Let $M_{R_\fp}$ be the maximal integral model of $\underline{M}_\fp=(M_{\fp},\tau_M)$. Inside $M[\fj^{-1}]$, we have:
\[M[\fj^{-1}]\cap \left(M_{R_{\fp}}[\fj^{-1}]+(\id-\tau_M)(M_{\fp})\right)=M[\fj^{-1}]\cap M_{R_{\fp}}[\fj^{-1}]+(\id-\tau_M)(M).\]
\end{Lemma}
\begin{proof}
The inclusion $\supset$ is clear. Since $M$ is generated over $F$ by elements in $M\cap M_{R_{\fp}}$ and as $F_\fp=F+R_{\fp}$, we have $M_{\fp}=M+M_{R_{\fp}}$. Let $m$ be an element in the left-hand side. We can write $m$ as $m_{\fp}+n_{\fp}-\tau_M(\tau^*n_{\fp})+n-\tau_M(\tau^*n)$ where $m_{\fp}\in M_{R_\fp}[\fj^{-1}]$, $n_{\fp}\in M_{R_{\fp}}$ and $n\in M$. In particular, $m_{\fp}+n_{\fp}-\tau_M(\tau^*n_{\fp})$ belongs to $M[\fj^{-1}]\cap M_{R_{\fp}}[\fj^{-1}]$ which implies that $m\in M[\fj^{-1}]\cap M_{R_{\fp}}[\fj^{-1}]+(\id-\tau_M)(M)$.
\end{proof}

\begin{Lemma}\label{lem:m-belongs-to-finitely-many}
Let $m\in M$. Then $m\in M_{R_{\fp}}$ for almost all maximal ideals $\fp$ of $R$.
\end{Lemma}
\begin{proof}
There exists a nonzero element $d\in R$ such that $dm \in M_{R}$. Let $\{\fq_1,\ldots,\fq_s\}$ be the finite set of maximal ideals in $R$ that contain $(d)$. By Proposition \ref{prop-existence-of-maximal-integral-model-global}, $m\in M_{R_{\fp}}$ for all $\fp$ not in $\{\fq_1,\ldots,\fq_s\}$.
\end{proof}

Let $N$ be a finite dimensional vector space over $F$ (resp. $F_{\fp}$). By a \emph{lattice in $N$} we mean a finitely generated module over $R$ (resp. $R_{\fp}$) in $N$ that contains a basis of $N$.
\begin{Lemma}[Strong approximation]\label{lem-strong-approximation}
Let $N$ be a finite dimensional $F$-vector space and, for all maximal ideals $\fp$ of $R$, let $N_{R_{\fp}}$ be an $R_{\fp}$-lattice in $N_{\fp}:=N\otimes_{R}F_{\fp}$ such that the intersection $\bigcap_{\fp}\left(N\cap N_{R_{\fp}}\right)$, over all maximal ideals $\fp$ of $R$, is an $R$-lattice in $N$. Let $T$ be a finite set of maximal ideals in $R$ and, for $\fq\in T$, let $n_{\fq}\in N_{\fq}$. Then, there exists $n\in N$ such that $n-n_{\fq}\in N_{R_{\fq}}$ for all $\fq\in T$ and $n\in N_{R_{\fp}}$ for all $\fp$ not in $T$.
\end{Lemma}
\begin{proof}
Let $N_{R}$ denote the intersection $\bigcap_{\fp}\left(N\cap N_{R_{\fp}}\right)$ over all maximal ideals $\fp$ of $R$. By the structure Theorem for finitely generated modules over the Dedekind domain $R$, there exists a nonzero ideal $\fa\subset R$ and elements $\{b_1,\ldots,b_r\}\subset M$ such that 
\begin{equation}
N_{R}=Rb_1\oplus \cdots \oplus Rb_{r-1}\oplus \fa b_r.
\nonumber
\end{equation}
Because $N_{R}\otimes_R R_{\fp}\subset N_{R_{\fp}}$ for $\fp\subset R$, we have $R_{\fp}b_1\oplus \cdots \oplus\fp^{v_{\fp}(\fa)}R_{\fp} b_r\subset N_{R_{\fp}}$. For $\fq\in T$, let us write $n_{\fq}=\sum_i{f_{\fq,i}b_i}$ with $f_{\fq,i}\in F_{\fq}$. By the strong approximation Theorem \cite[Thm.~6.13]{rosen}, for all $i\in\{1,\ldots,r\}$, there exists $f_i\in F$ such that 
\begin{enumerate}
\item for $\fq \in T$ and $i\in\{1,\ldots,r-1\}$, $v_{\fq}(f_i-f_{\fq,i})\geq 0$,
\item for $\fq \in T$,  $v_{\fq}(f_r-f_{\fq,r})\geq v_{\fq}(\fa)$,
\item for $\fp\notin T$ and $i\in\{1,\ldots,r-1\}$, $v_{\fp}(f_i)\geq 0$,
\item for $\fp\notin T$, $v_{\fp}(f_r)\geq v_{\fp}(\fa)$.
\end{enumerate}
The element $n=\sum_i{f_ib_i}\in N$ satisfies the assumption of the lemma.
\end{proof}

\begin{Lemma}\label{lem:strong-approximation}
We have 
\begin{equation}\label{eq:translated-intersection}
\bigcap_{\fp\subset R} \left(M[\fj^{-1}]\cap M_{R_{\fp}}[\fj^{-1}]+(\id-\tau_M)(M)\right)=M_{R}[\fj^{-1}]+(\id-\tau_M)(M)
\end{equation}
where the intersection is indexed over the maximal ideals of $R$.
\end{Lemma}
\begin{proof}
The inclusion $\supset$ follows from Proposition \ref{prop-existence-of-maximal-integral-model-global}. Conversely, let $m$ be an element of the left-hand side of \eqref{eq:translated-intersection}. By Lemma \ref{lem:m-belongs-to-finitely-many}, there exists a finite subset $T$ of maximal ideals of $R$ such that $m\in M_{R_{\fp}}[\fj^{-1}]$ for $\fp\notin T$. For $\fq\in T$, there exists $n_{\fq}\in M$ and $m_{\fq}\in M[\fj^{-1}]\cap M_{R_{\fq}}[\fj^{-1}]$ such that $m=m_{\fq}+n_{\fq}-\tau_M(\tau^*n_{\fq})$. 

Let $N$ be a finite dimensional sub-$F$-vector space of $M$ that contains $m$ and $n_{\fq}$ for all $\fq\in T$. For a maximal ideal $\fp$ of $R$, let $N_{R_\fp}:=M_{R_{\fp}}\cap (N\otimes_{F}F_{\fp})$. Let also $N_{R}:=\bigcap_{\fp}(N\cap N_{R_\fp})$. By Proposition \ref{prop-existence-of-maximal-integral-model-global}, we have $N_R=N\cap M_R$. In particular, this implies that $N_R$ generates $N$ over $F$. This also implies that $N_R$ is a finitely generated $R$-module : as $M_R$ is finite projective (Theorem \ref{thm:integral-models-are-locally-free}), it is included in a finite free $A\otimes R$-module from which we consider a basis $(m_1,\ldots,m_t)$. On the other-hand, if $(n_1,\ldots,n_s)$ is a basis of $N$ over $F$, there exists a constant $C\geq 0$ large enough for which each $n_i$ belongs to $\bigoplus_j (A\otimes F)_{\deg<C}\cdot m_j$, $i\in\{1,\ldots,s\}$, where we denoted $(A\otimes F)_{\deg<C}$ the finite dimensional sub-$F$-vector space of $A\otimes F$ of elements built as sums of elementary tensors $a\otimes f$ with $\deg(a)<C$. This yields:
\[
N_R\subseteq \bigoplus_{j=1}^t (A\otimes R)_{\deg<C}\cdot m_j
\]
and hence $N_R$ is finitely generated. We have therefore proved that $N_R$ is an $R$-lattice in $N$, and Lemma \ref{lem-strong-approximation} applies: there exists $n\in N$ such that $n-n_{\fq}\in N_{R_{\fq}}$ for all $\fq\in T$ and $n\in N_{R_{\fp}}$ for all $\fp$ not in $T$. Then $m+n-\tau_M(\tau^*n)\in N_{R}\subset M_{R}$, which ends the proof.
\end{proof}

\begin{proof}[Proof of Theorem \ref{thm:main2}]
Let $[\underline{E}]\in \Ext^1_{\cM_F}(\mathbbm{1},\underline{M})$ and let $m\in M[\fj^{-1}]$ be such that $[\underline{E}]=\iota(m)$. The proof of Theorem \ref{thm:main2} is achieved via the sequence of equivalence:
\begin{align*}
[\underline{E}]\in \Ext^1_{R}(\mathbbm{1},\underline{M}) &\Longleftrightarrow \forall \fp\in \Spm R: ~[\underline{E}_{\fp}]\in \Ext^1_{R_{\fp}}(\mathbbm{1}_\fp,\underline{M}_\fp) \\
&\Longleftrightarrow \forall \fp\in \Spm R: ~m\in M[\fj^{-1}]\cap [M_{R_{\fp}}[\fj^{-1}]+(\id-\tau_M)(M_{\fp})] \\
&\Longleftrightarrow \forall \fp\in \Spm R: ~m\in M[\fj^{-1}]\cap M_{R_{\fp}}[\fj^{-1}]+(\id-\tau_M)(M) \\
&\Longleftrightarrow m\in M_R[\fj^{-1}]+(\id-\tau_M)(M)  \\
&\Longleftrightarrow [\underline{E}]\in \iota(M_R[\fj^{-1}])
\end{align*}
where the second equivalence stems from Definition \ref{def:integral-models-of-motives}, the third from Lemma \ref{lem:decompletion-local} and the fourth from Lemma \ref{lem:strong-approximation}. The second assertion follows from Corollary \ref{cor:(id-tau)(M0)=(id-tau)(M)capMO}.
\end{proof}

\section{Regulated extensions}\label{sec:regulated-extensions}
Let $\fm\subset A$ be a maximal ideal of $A$ with associated local function field $K_\fm$, let $F_\fp$ be a finite field extension of $K_\fm$ with valuation ring $\cO_\fp$ and, as before, denote by $\kappa:A\to\cO_{\fp}$ the inclusion.

Consider an $A$-motive $\underline{M}$ over $\cO_{\fp}$. In the previous section, we proved that given any maximal ideal $\ell$ in $A$ distinct from $\fm$, there is an inclusion:
\begin{equation}\label{eq:inclusion-of-ext}
\Ext^1_{\cO_{\fp}}(\mathbbm{1},\underline{M})\subset \Ext^1_{\text{good}}(\mathbbm{1},\underline{M})_{\ell}
\end{equation}
of sub-$A$-modules of $\Ext^1_{\cM_{F_{\fp}}}(\mathbbm{1},\underline{M})$. Surprisingly, this is almost never an equality. In Subsection \ref{subsec:counter-example}, we construct explicitly a class in the right-hand side of \eqref{eq:inclusion-of-ext} which does not belong the left-hand side. In the remaining part of this text, we offer a conjectural framework in which we expect to solve this default.

\subsection{A particular extension of $\mathbbm{1}$ by itself}\label{subsec:counter-example}
We consider the case where $A=\bF[t]$ and consider the maximal ideal $\ell=(t)$ in $A$. Let $E$ be the local function field $\bF(\!(\pi)\!)$, $\cO=\bF[\![\pi]\!]$, with structure morphism $\kappa:A\to \cO$ defined by $\kappa(t)=1+\pi$ (that is, $\pi=\theta-1$ where $\theta=\kappa(t)$). We have $\kappa(\ell)\cO=(1+\pi)\bF[\![\pi]\!]=\cO$ so $\kappa(\ell)\cO=\cO$. Let $\underline{M}=\mathbbm{1}$ over $E$. By Proposition \ref{prop:diagram-explicit-cocycle} enriched with Proposition \ref{prop:artin-schreier}, there is a commutative square of $A$-modules:
\[
\begin{tikzcd}
\Ext^1_{\cM_E}(\mathbbm{1},\mathbbm{1}) \arrow[r] & \operatorname{H}^1(I_E,T_{(t)}\mathbbm{1}) \\
\displaystyle\frac{E[t,\frac{1}{t-\theta}]}{(\id-\tau)(E[t])} \arrow[u,"\wr","\iota"']\arrow[r]  & \displaystyle\frac{E^{\operatorname{ur}}[\![t]\!]}{\cO[\![t]\!]+(\id-\tau)(E[\![t]\!])} \arrow[u,"\wr"']
\end{tikzcd}
\]
where the bottom arrow is induced by the inclusion of $E[t,(t-\theta)^{-1}]$ in $E[\![t]\!]$. \\

Hereafter, we construct an element $m$ in $E[t,(t-\theta)^{-1}]$ of the form
\[
m=\frac{m_k}{(t-\theta)^k}+\cdots+\frac{m_1}{(t-\theta)}
\]
for some $m_1,\ldots,m_k$ in $E$, not all in $\cO$, such that $m$ belongs to $(\id-\tau)(E[\![t]\!])$. Then, $\iota(m)$ has good reduction with respect to $(t)$ (in the sense of Definition \ref{def:ext-good-red-with-resp-to-ell}) but does not belong to $\cO[t,(t-\theta)^{-1}]+(\id-\tau)(E[t])$.\\

Let $k=q^2$ where $q$ is the number of elements of $\bF$, and for $i\in\{0,\ldots,k-1\}$, define $n'_i$ as $\theta^i(\pi^{-1}-\pi^{-q})$ in $E$. So $n_i'$ has valuation $-q$. For all $c\geq 0$, let $f_{ck}$ be a root in $E$ of the polynomial:
\[
X^q-X+\theta^{-ck}n'_0.
\] 
Such a root exists in $E$ as $\theta^{-k}\equiv 1\pmod{\pi^{q^2}}$. We now define $f_l\in E$ for all $l\geq 0$ by the rule $f_l:=f_{ck}$ if $l=ck+r$ for $c\geq 0$ and $0\leq r<k$. We obtain
\begin{equation}\label{eq:fl}
\text{for~all~}l\geq 0:\quad \theta^{-l}n'_{\bar{l}}=f_l-f_l^q
\end{equation}
where $\overline{l}\in\{0,\ldots,k-1\}$ denotes the remainder of the euclidean division of $l$ by $k$. \\

For $l\geq 0$, let $S_k(l)$ be the \emph{Pascal matrix} whose $i$th row-$j$th column entry is the binomial coefficient $\binom{i+j+l}{i+l}$ ($0\leq i,j<k$). The following claims are easily proven:
\begin{enumerate}[label=$(\roman*)$]
\item\label{item:1} The determinant of $S_k(0)$ is $1$.
\item\label{item:2} Let $p$ be the characteristic of $\bF$. For $l\geq 0$, we have the formula
\[
S_k(l+1)\equiv \begin{pmatrix}
 & 1 & & & \\
 & & 1 & & \\
 & & & \ddots & \\
 & & & & 1 \\
1 & & & &
\end{pmatrix}S_k(l) \pmod{p}
\]
\item The application $l\mapsto S_k(l)$ is $k$-periodic modulo $p$.
\end{enumerate}
We now define $m_i'\in E$ for $i\in\{0,\ldots,k-1\}$ by mean of the formula:
\[
S_k(0)\begin{pmatrix}m_0' \\ m_1'\\ \vdots \\ m_{k-1}'\end{pmatrix}=\begin{pmatrix}n_0' \\ n_1'\\ \vdots \\ n_{k-1}'\end{pmatrix}
\]
Since the $n_i$'s have negative valuation, at least one of the $m_i$'s has negative valuation (by \ref{item:1}). From \ref{item:2}, we have 
\[
\text{for~all~}l\geq 0:\quad S_k(l)\begin{pmatrix}m_0' \\ m_1'\\ \vdots \\ m_{k-1}'\end{pmatrix}=\begin{pmatrix}n_{\bar{l}}' \\ n_{\overline{l+1}}'\\ \vdots \\ n_{\overline{l+k-1}}'\end{pmatrix}
\]
From \eqref{eq:fl}, we obtain:
\begin{equation}\label{eq:coefficients}
\text{for~all~}l\geq 0:\quad \theta^{-l}\left(\sum_{i=0}^{k-1}{m'_i\binom{i+l}{l}}\right)=f_l-f_l^q.
\end{equation}
Finally, for $i\in\{1,\ldots,k\}$, let $m_i:=(-\theta)^{i}m_{i-1}'$. Formula \eqref{eq:coefficients} amounts to:
\[
m:=\frac{m_{k}}{(t-\theta)^k}+\ldots+\frac{m_{1}}{(t-\theta)}=f-f^{(1)}
\]
where $f:=\sum_{l\geq 0}{f_l t^l}$. Therefore $\iota(m)$ has good reduction, although $m$ does not belong to $\cO[t,(t-\theta)^{-1}]+(\id-\tau)(E[t])$.

\subsection{Hodge poylgons of $A$-motives}\label{subsec:Hodge-polygon}
We recognize that the extension of \emph{Hodge-Pink structures} corresponding to $\iota(m)$ constructed in the previous subsection is not \emph{Hodge additive} in the sense of \cite[\S 6 \& 7]{pink}. We now introduce the notion of \emph{regulated extensions} which is the counterpart of Hodge additivity for $A$-motives. \\

Let $F$ be (any) field equipped with an $\bF$-algebra morphism $\kappa:A\to F$. Let $\underline{M}$ be an $A$-motive of rank $r$ over $F$ and characteristic morphism $\kappa$. Given $e$ a large enough integer for which $\fj^{e}\tau_M(\tau^*M)\subset M$, we have an isomorphism 
\[
M/\fj^{e}\tau_M(\tau^*M) \cong \bigoplus_{i=1}^r (A\otimes F)/\fj^{e+w_i}
\]
for some uniquely determined integers $w_1\leq \cdots \leq w_r$ independent of $e$ nor of the isomorphism. We call $(w_1,\ldots,w_r)$ the \emph{multiset of Hodge weights of $\underline{M}$}. Hodge weights are graphically organized into the \emph{Hodge polygon of $\underline{M}$} whose construction we now recall. 

By \emph{polygon}, we mean the graph in $\bR^2$ of a piecewise linear convex function $[0,n]\to \bR$ starting at $(0,0)$. All slopes are assumed to be rational numbers, and the length of the subinterval of $[0,n]$ on which the function as a given slope $q\in \bQ$ is called the \emph{multiplicity of $q$}. Therefore, each polygon $P$ is uniquely determined by its \emph{multiplicity function} $m_P:\bQ\to \bN$ which assigns to $q\in \bQ$ the multiplicity of the slope $q$. Such a function $m$ is the multiplicity function of a polygon if, and only if, the set $\{q\in \bQ|m(q)> 0\}$ is finite. The \emph{Hodge Polygon of $\underline{M}$} is the polygon whose multiplicity function is $q\mapsto \#\{i\in \{1,\ldots,r\}|w_i=q\}$. \\

Let $0\to \underline{M}\to \underline{E}\to \underline{N}\to 0$ be an exact sequence of $A$-motives over $F$. An argument due to Katz shows that the Hodge Polygon of $\underline{M}\oplus \underline{N}$ lies above that of $\underline{E}$ and has the same end points (see \cite[Lem. 1.2.3]{katz-slope} in the context of $F$-crystals, or \cite[Prop. 6.9]{pink} in our situation). In general, one cannot claim equality among those Hodge Polygons which makes our situation differ from classical mixed motives over $\bQ$. In the context of function fields Hodge structures, in order the compute Hodge groups explicitly, Pink introduced the condition of \emph{Hodge additivity} \cite[\S 6 \& 7]{pink} whose counterpart for $A$-motives is that of \emph{regulation}.
\begin{Definition}\label{def:regulated-extensions}
We call an exact sequence $0\to \underline{M}\to \underline{E} \to \underline{N} \to 0$ in $\cM_F$ \emph{regulated} if the Hodge polygon of $\underline{M}\oplus \underline{N}$ coincide with that of $\underline{E}$.
\end{Definition}

Clearly, any exact sequence equivalent to a regulated exact sequence in $\Ext^1_{\cM_F}(\underline{N},\underline{M})$ is itself regulated. We denote by $\Ext^{1,\operatorname{reg}}_{\cM_F}(\underline{N},\underline{M})$ the subset of regulated extensions.

\begin{Remark}\label{rmk:regulation}
The choice of the naming \emph{regulated} deserves some explanations. It is derived from the name \emph{regulator} which corresponds classically to the group morphism from extensions of classical mixed motives to extensions of classical mixed Hodge structures induced by the exactness of the Hodge realization functor. In function field arithmetic, one associates to an $A$-motive $\underline{M}$ the data of an $F$-vector space equipped with a decreasing filtration:
\begin{equation}\label{eq:hodge-filtration}
H:=\tau^*M/\fj\tau^* M, \quad p\in \bZ:~\operatorname{Fil}^p H:=\operatorname{image}\left(\tau^*M\cap \fj^p\cdot \tau_M^{-1}(M)\to \tau^*M/\fj\tau^* M\right),
\end{equation}
which one could legitimately called \emph{its associated Hodge structure} $(H,\operatorname{Fil})$ (\emph{e.g.} \cite{hartl-juschka}). However, following considerations by Pink, this functor is not exact\footnote{Call a sequence $S$ of Hodge structures \emph{exact} if the sequence $\operatorname{Fil}^p S$ is exact in the category of $F$-vector spaces for all $p$.} hence preventing the well-definedness of a regulator morphism in this setting. More precisely, as one could derived from \cite[Prop. 6.11]{pink}, an exact sequence of $A$-motives $\underline{S}$ induces an exact sequence of Hodge structures if and only if $\underline{S}$ is regulated. The notion of regulation is exploited in a sequel to this text to define a function field regulator among finite dimensional vector spaces (see \cite{gazda2} and \cite{gazda-maurischat-ext}).
\end{Remark}

The next proposition allows to compute regulated extensions in the category $\cM_F$. Let $\underline{M}$ and $\underline{N}$ be two objects in $\cM_F$. Recall the map $\iota=\iota_{\underline{N},\underline{M}}$ from Proposition \ref{prop-cohomology-in-tilde}.
\begin{Proposition}\label{prop:explicit-regulation}
Let $u\in \Hom_{A\otimes F}(\tau^*N,M)[\fj^{-1}]$ and let $[\underline{E}]$ be the extension of $\underline{N}$ by $\underline{M}$ given by $\iota_{\underline{N},\underline{M}}(u)$. Then, $[\underline{E}]$ is regulated if and only if there exists $f\in \Hom_{A\otimes F}(\tau^*N,\tau^* M)$ and $g\in \Hom_{A\otimes F}(N,M)$ such that $u=\tau_M\circ f-g\circ \tau_N$.
\end{Proposition}
\begin{proof}
We may assume that $\underline{E}$ is the $A$-motive $\left[M\oplus N, \left(\begin{smallmatrix}\tau_M & u \\ 0 & \tau_N \end{smallmatrix}\right)\right]$ (in full generality, the extension $[\underline{E}]$ is only equivalent to it). We preserve the notation~$E$ instead of~$M\oplus N$ below to lighten notations. For $e\geq 0$ large enough, there is an exact sequence of $A\otimes F$-modules
\[
S:\quad 0\longrightarrow M/\fj^e \tau_M(\tau^*M) \stackrel{m}{\longrightarrow} E/\fj^e \tau_E(\tau^* E) \stackrel{n}{\longrightarrow} N/\fj^e \tau_N(\tau^* N)\longrightarrow 0
\]
resulting from the Snake Lemma. We interpret the terms of $S$ as finite torsion modules over the discrete valuation ring $F[\![\fj]\!]:=\varprojlim_n A\otimes F/\fj^n$. The elementary divisors are of the form $e+i$, where $i$ runs through the slopes of the corresponding Hodge polygon. Thus, similar to the proof of \cite[Prop. 8.7]{pink}, the regulation of $[\underline{E}]$ is equivalent to the splitting of $S$ over $F[\![\fj]\!]$, itself equivalent to the splitting of $S$ over $A\otimes F$.

Assume that $[\underline{E}]$ is regulated. From a splitting of $S$, we get an isomorphism 
\begin{equation}\label{eq:reduced-comparison}
\varepsilon: M/\fj^e \tau_M(\tau^*M)\oplus N/\fj^e\tau_N(\tau^*N)\stackrel{\sim}{\longrightarrow} E/\fj^e \tau_E(\tau^* E)
\end{equation}
compatible with $S$. Consider the linear map $\bar{g}:N\to E/\fj^e \tau_E(\tau^* E)$, $n\mapsto \varepsilon(0,n)-(0,n)$, where the subtracted term is understood through the surjection $M\oplus N\twoheadrightarrow E/\fj^e \tau_E(\tau^* E)$. By compatibility of $\varepsilon$ with~$S$, we have $n\circ \bar{g}=0$, hence $\bar{g}$ factors through $m$. As $M$ is a projective module, $\bar{g}:N\to M/\fj^e \tau_M(\tau^*M)$ lifts to a map $g:N\to M$. By construction, we obtain a commutative square 
\[
\begin{tikzcd}[ampersand replacement = \&]
M/\fj^e \tau_M(\tau^*M)\oplus N/\fj^e\tau_N(\tau^*N) \arrow[r,"\varepsilon"] \& E/\fj^e \tau_E(\tau^* E) \\
M\oplus N \arrow[->>,u] \arrow[r,"{\left(\begin{smallmatrix} \id_M & g \\ 0 & \id_N \end{smallmatrix}\right)}"] \& M\oplus N \arrow[u,->>]
\end{tikzcd}
\]
The bottom row maps $\fj^e\left(\begin{smallmatrix}\tau_M & 0 \\ 0 & \tau_N\end{smallmatrix}\right)(\tau^*M\oplus \tau^*N)$ isomorphically to $\fj^e \left(\begin{smallmatrix} \tau_M & u \\ 0 & \tau_N \end{smallmatrix}\right)(\tau^*M\oplus \tau^* N)$. In particular, for any $n\in \tau^*N$, there exists a necessarily unique $m'\in \tau^*M$ such that 
\[
\begin{pmatrix} \id_M & g \\ 0 & \id_N \end{pmatrix}\begin{pmatrix} \tau_M & 0 \\ 0 & \tau_N \end{pmatrix}\begin{pmatrix} 0 \\ n \end{pmatrix}=\begin{pmatrix} \tau_M & u \\ 0 & \tau_N \end{pmatrix}\begin{pmatrix} m' \\ n \end{pmatrix}.
\] 
The assignment $n\mapsto m'$ is linear, and we denote it $f$. We have $u=\tau_M\circ f-g\circ \tau_N$ by construction, as desired.

It becomes clear from the proof how to construct a splitting of $S$, from the existence of $f$ and $g$ such that $u=\tau_M\circ f-g\circ \tau_N$, proving thusly the converse statement. 
\end{proof}

It follows from the above proposition that the subset $\Ext^{1,\text{reg}}_{\cM_F}(\underline{N},\underline{M})$ of regulated extensions of $\underline{N}$ by $\underline{M}$ is well-defined and is a sub-$A$-module of $\Ext^{1}_{\cM_F}(\underline{N},\underline{M})$. We also deduce that the notion of regulation is compatible with the duality property of extension modules:
\begin{Corollary}
Let $d:\Ext^1_{\cM_F}(\underline{N},\underline{M})\to \Ext^1_{\cM_F}(\mathbbm{1},\underline{M}\otimes \underline{N}^{\vee})$ be the canonical map of \eqref{eq:explici-ext1-map-dual}. Let $[\underline{E}]$ be an extension of $\underline{N}$ by $\underline{M}$. Then, $[\underline{E}]$ is regulated if and only if $d([\underline{E}])$ is.
\end{Corollary}
\begin{proof}
Fix an extension $[\underline{E}]:=[0\to \underline{M}\to \underline{E}\stackrel{p}{\to} \underline{N}\to 0]$. Under $d$, it is mapped to the extension $[\underline{E}']:=[0\to \underline{M}\otimes\underline{N}^{\vee} \to \underline{E}' \to \mathbbm{1} \to 0]$ where $\underline{E}'$ is the $A$-motive obtained as the pullback of
\[
\underline{E}\otimes \underline{N}^{\vee}\xrightarrow{p\otimes \id} \underline{N}\otimes \underline{N}^{\vee} \stackrel{\varepsilon}{\longleftarrow} \mathbbm{1}
\]
and where $\varepsilon$ is the unit morphism (\emph{cf}. Subsection \ref{subsec:monoidal-exact-cat} of the appendix). Let $u:(\tau^*N)[\fj^{-1}]\to M[\fj^{-1}]$ be an $A\otimes F$-linear map for which $\iota_{\underline{N},\underline{M}}(u)=[\underline{E}]$. This means that, up to replacing $[\underline{E}]$ by an equivalent extension, we can assume $\underline{E}$ to be the $A$-motive
\[
(E,\tau_E):=\left(M\oplus N, \begin{pmatrix} \tau_M & u \\ 0 & \tau_N \end{pmatrix}\right)
\]
with $p:E\to N$ being the projection $\operatorname{pr}_N$. Then, $\underline{E}'$ is described as the $A$-motive whose underlying module is 
\[
E'=\left\{v\in \Hom_{A\otimes F}(N,N\oplus M)~|~\operatorname{pr}_N\circ g \in (A\otimes F) \cdot \id_N \right\}
\]
and where $\tau_{E'}$ maps $v\in \tau^*\Hom(N,E)=\Hom(\tau^*N,\tau^*E)$ to $\tau_E\circ v\circ \tau_N^{-1}\in \Hom(N,E)[\fj^{-1}]$. Now observe that any element in $E'$ can be written uniquely as $h\oplus a\cdot \id_N$ for some $h:N\to M$ and $a\in A\otimes F$. Taking $v\in E'$ in this form, we get 
\[
\tau_{E'}(\tau^*v)= (\tau_M\circ \tau^*v\circ \tau_N^{-1}+a\cdot u\circ \tau_N^{-1})\oplus a\cdot \id_N.
\]
From this computation we deduce that the following diagram commutes
\begin{equation}
\begin{tikzcd}[column sep=4em]
\Hom_{A\otimes F}(\tau^*N,M)[\fj^{-1}] \arrow[r,"\iota_{\underline{N},\underline{M}}"]\arrow[d] & \Ext^1_{\cM_F}(\underline{N},\underline{M}) \arrow[d,"d"] \\
\Hom_{A\otimes F}(\tau^*(A\otimes F),M\otimes N^{\vee})[\fj^{-1}] \arrow[r,"\iota_{\mathbbm{1},\underline{M}\otimes \underline{N}^{\vee}}"] & \Ext^1_{\cM_F}(\mathbbm{1},\underline{M}\otimes \underline{N}^{\vee})
\end{tikzcd}\nonumber
\end{equation}
where the left vertical map sends $u$ to the morphism mapping $a\in \tau^*(A\otimes F)=A\otimes F$ to $a\cdot (u\circ \tau_N^{-1})$. Applying Proposition \ref{prop:explicit-regulation} we get that $[\underline{E}]$ is regulated, if and only if there exists $f$ and $g$ such that $u=g\circ \tau_N-\tau_M\circ f$, if and only if $u\circ \tau_N^{-1}=g-\tau_M\circ f\circ \tau_N^{-1}$, which happens if and only if $[\underline{E}']$ is itself regulated.
\end{proof}

In the particular situation of $\underline{N}=\mathbbm{1}$, Proposition \ref{prop:explicit-regulation} yields:
\begin{Corollary}\label{cor:regulated-extensions-1-M}
Let $\underline{M}$ be an $A$-motive over $F$. Then, $\iota$ induces an isomorphism of $A$-modules:
\[
\frac{M+\tau_M(\tau^*M)}{(\id-\tau_M)(M)} \stackrel{\sim}{\longrightarrow} \Ext^{1,\operatorname{reg}}_{\cM_F}(\mathbbm{1},\underline{M}).
\]
\end{Corollary}

One may use Corollary \ref{cor:regulated-extensions-1-M} to show that the extension of $\mathbbm{1}$ by itself constructed in Subsection \ref{subsec:counter-example} is not regulated. Indeed, using the notations of Subsection \ref{subsec:counter-example}, the aforementioned extension was represented by an element 
\[
m=\frac{m_{k}}{(t-\theta)^{k}}+\cdots +\frac{m_1}{(t-\theta)} \in M[\fj^{-1}]=E[t]\left[\frac{1}{t-\theta}\right]
\] 
where at least one of the $m_i$'s is non zero. In particular, it does not belong to $M+\tau_M(\tau^*M)$ which is $E[t]$ in this case. 

\begin{Remark}
For Hodge-Pink structures, Pink proved that the dimension of \emph{Hodge additive extension spaces} is finite \cite[Prop. 8.7]{pink}. It seems at first reasonable to expect a similar  result for integral regulated extensions of $A$-motives. However, the condition of \emph{regulation} alone is not sufficient to state a counterpart of conjecture \ref{item:conjectureC5} for $A$-motives. An other condition, that of \emph{analytic reduction at $\infty$}, is required. This is the main subject of our sequel \cite{gazda2}, where we prove the finite generation of the $A$-module of integral regulated extensions having analytic reduction at $\infty$ (see Thm. 4.1 in \emph{loc.\,cit.}). 
\end{Remark}

\subsection{Regulated extensions having good reduction}\label{subsec:regulated-good-reduction}
We now assume that $F$ is a finite field extension of $K$, let $\fp$ be a finite place of $F$ and let $\fm$ be the place of $K$ sitting under $\fp$. Hereafter, $\kappa$ is the inclusion of $A$ into $F$; it factors through $\cO_F$, the integral closure of $A$ in $F$. By the field $L$ (resp. the ring $\cO_L$) we shall mean either $F$ or $F_{\fp}$ (resp. $\cO_F$ or $\cO_{\fp}$). Let $\ell$ be a maximal ideal in $A$ distinct from $\fm$. \\

Let $\underline{M}$ be an $A$-motive over $L$. Recall that we considered two submodules of $\Ext^1_{\cM_L}(\mathbbm{1},\underline{M})$ related to integrality and good reduction respectively, \\

\begin{center}
\begin{tabular}{llc}
$\Ext^1_{\cO_L}(\mathbbm{1},\underline{M})$ & integral extensions & (Definition \ref{def:integral-part-local}) \\
$\Ext^1_{\operatorname{good}}(\mathbbm{1},\underline{M})_{\ell}$ & good reduction extensions w.r.t. $\ell$ & (Definition \ref{def:ext-good-red-with-resp-to-ell})
\end{tabular}
\end{center}

We now introduce their regulated avatar.

\begin{Definition}
We let $\Ext^{1,\text{reg}}_{\cO_L}(\mathbbm{1},\underline{M})$ be the submodule of $\Ext^1_{\cO_L}(\mathbbm{1},\underline{M})$ consisting of regulated extensions. Similarly, by $\Ext^{1,\text{reg}}_{\text{good}}(\mathbbm{1},\underline{M})_{\ell}$ we designate the submodule of $\Ext^{1}_{\text{good}}(\mathbbm{1},\underline{M})_{\ell}$ consisting of regulated extensions in the category $\cM_L$.
\end{Definition}

Assume now that $L=F_{\fp}$. By Theorem \ref{mthm:Integral-inside-Good} (Theorem \ref{thm:main2}), there is an inclusion of $A$-modules:
\begin{equation}\label{eq:property}
\Ext^{1,\text{reg}}_{\cO_{\fp}}(\mathbbm{1},\underline{M})\subseteq \Ext^{1,\text{reg}}_{\text{good}}(\mathbbm{1},\underline{M})_{\ell}.
\end{equation}
We strongly suspect the above to be an equality although we were unable to prove it in generality. The following is our expected analogue of Conjecture \ref{item:conjectureC4}:
\begin{Conjecture}\label{conjecture-reg}
The inclusion \eqref{eq:property} is an equality. In particular, $\Ext^{1,\text{reg}}_{\text{good}}(\mathbbm{1},\underline{M})_{\ell}$ does not depend on $\ell$.
\end{Conjecture}

We stated Conjecture \ref{conjecture-reg} as the only examples we could produce of extensions which belongs to the right-hand side of \eqref{eq:property}, but not in the left-hand one, were not regulated. Still, we owe the reader stronger motivations for it. For the remainder of this section, we present some evidences for Conjecture \ref{conjecture-reg}, in Theorems \ref{thm:true-for-pure0} and \ref{thm:true-for-carlitz} below. \\

Hereafter we assume $A=\bF[t]$. Let $\underline{V}=(V,\varphi)$ be a Frobenius space over $F_\fp$ (Subsection \ref{subsec:integral-models-of-frobenius-modules}). From $\underline{V}$ we obtain an $A$-motive over $F_\fp$, denoted $A\otimes \underline{V}$, whose underlying module is $A\otimes V$ and whose morphisms is $\id_A\otimes \varphi$. Those $A$-motives form a quite restrictive class: observe that $A\otimes \underline{V}$ is effective of weight and Hodge weight zero. 

\begin{Theorem}\label{thm:true-for-pure0}
Assume that $\underline{V}$ has good reduction (Definition \ref{def:good-red-frob-space}). Then Conjecture \ref{conjecture-reg} is true for $A\otimes \underline{V}$.
\end{Theorem}

We will in fact prove a slightly more general version:
\begin{Proposition}\label{prop:thm:true-for-pure0}
Let $\underline{N}=(N,\tau_N)$ be an $A$-motive over $F_\fp$ such that :
\begin{enumerate}[label=$(\roman*)$]
\item it is effective,
\item it has good reduction,
\item\label{item:existence-lattice} there exists an $\cA(F_\fp)$-lattice $\Lambda$ in $N\otimes_{A\otimes F_\fp} \cB(F_\fp)$ such that both
\begin{enumerate}[label=$(\alph*)$]
\item $\tau_N(\tau^*\Lambda)=\Lambda$,
\item As $F_\fp$-vector spaces, $N\otimes_{A\otimes F_\fp}\cB_{\infty}(F_\fp)=N\oplus \Lambda$.
\end{enumerate}
\end{enumerate}
Then, Conjecture \ref{conjecture-reg} is true for $\underline{N}$.
\end{Proposition}

The assumption $A=\bF[t]$ was superfluous so far; we use it next to relate Proposition \ref{prop:thm:true-for-pure0} to Theorem \ref{thm:true-for-pure0} in the next lemma.
\begin{Lemma}
The $A$-motive $A\otimes \underline{V}$ satisfies the condition of Proposition \ref{prop:thm:true-for-pure0}.
\end{Lemma}
\begin{proof}
Write $\underline{N}=(N,\tau_N)$ for the $A$-motive $A\otimes \underline{V}$. That $\underline{N}$ is effective and has good reduction is clear. It remains to find $\Lambda$. Identifying $\cA(F_\fp)$ with $F_\fp[\![t^{-1}]\!]$ and $N\subset N_{A\otimes F_\fp}\cB_{\infty}(F_\fp)$ with $V[t]\subset V(\!(t^{-1})\!)$, we can choose $\Lambda=t^{-1}V[\![t^{-1}]\!]$.
\end{proof}

We are left with the proof of Proposition \ref{prop:thm:true-for-pure0}. First, we show:
\begin{Lemma}\label{lem:ell-adic-closed}
Let $\ell$ be a maximal ideal of $A$ and let $\underline{N}$ be an $A$-motive over $F_{\fp}$ as in Proposition \ref{prop:thm:true-for-pure0}. Then $N_{\cO_{\fp}}+(\id-\tau_N)(N)$ is $\ell$-adically closed in $N$.
\end{Lemma}
\begin{proof}
Let $\Lambda$ be as in Proposition \ref{prop:thm:true-for-pure0} \ref{item:existence-lattice}. For $n\geq 0$, let $N_n$ be the finite dimensional $F_{\fp}$-vector space $\ell^n\Lambda\cap N$. $(N_n)_n$ defines an increasing sequence of subspaces of $N$ and we both have $\bigcup_{n\geq 0}N_n=N$ and $N=N_n\oplus \ell^nN$. We claim that:
\[
N_{\cO_{\fp}}\cap N=(N_{\cO_{\fp}}\cap N_n)\oplus \ell^n(N_{\cO_{\fp}}\cap N).
\]
To see this, note that since $\underline{N}$ has good reduction, the image of $N_{\cO_\fp}\cap N_n$ through $N\to N/\ell^nN$ equals the maximal integral model of $(N/\ell^n N,\tau_M)$. Hence $N_{\cO_{\fp}}\cap N\subset (N_{\cO_{\fp}}\cap N_n)\oplus \ell^nN$ and the claim follows.

Let $m\in N$ be such that there exists a sequence $(m_n)_{n\geq 0}$ in $N_{\cO_{\fp}}+(\id-\tau_N)(N)$ which converges $\ell$-adically to $m$. We can assume without loss of generality that $m_n\in N_n$ for all $n$. Yet, we have $m\in N_d$ for a large enough integer $d$. If $p_d$ denotes the projection onto $N_d$ orthogonally to $\ell^dN$, we obtain $m=p_d(m)=p_d(m_d)\in (N_{\cO_\fp}\cap N_d)+(\id-\tau_N)(N_d)$ as desired.
\end{proof}

\begin{proof}[Proof of Proposition \ref{prop:thm:true-for-pure0}]
Let $[\underline{E}]=\iota(m)$ be an extension in $\Ext^{1,\text{reg}}_{\text{good}}(\mathbbm{1},\underline{N})_{\ell}$. By definition, $m\in N$, and by Proposition \ref{prop:explicit-cocycle}, there exists $\xi\in (N_{F_{\fp}^{\operatorname{ur}}})^{\wedge}_{\ell}$ such that
\begin{equation}\label{eq:x-txi=n}
m=\xi-\tau_N(\tau^*\xi).
\end{equation}
Reducing \eqref{eq:x-txi=n} modulo $\ell^n$ for all $n\geq 1$, we obtain from Proposition \ref{prop:artin-schreier}, applied to the Frobenius space $(N/\ell^n N,\tau_N)$, that
\[
\forall n\geq 1:\quad m\in \tilde{L}_n+(\id-\tau_N)(N)+\ell^nN
\]
where $\tilde{L}_n$ is a lift in $N$ of a maximal integral model for $(N/\ell^n N,\tau_N)$. It follows from Proposition \ref{prop:variation-integral-model:motive-to-frobenius} that $m$ belongs to the $\ell$-adic closure of $N_{\cO_\fp}+(\id-\tau_N)(N)$. But the later is already closed by Lemma \ref{lem:ell-adic-closed}. It follows that $m\in N_{\cO_\fp}+(\id-\tau_N)(N)$, which amounts to $[\underline{E}]\in \Ext^{1,\text{reg}}_{\cO_\fp}(\mathbbm{1},\underline{N})$ as desired.
\end{proof}

Let $\underline{M}=\underline{A}(n)$ be the $n$th twist of the Carlitz motive over $F_{\fp}$ (Example \ref{ex:carlitz-motive}). That is, $M=F_\fp[t]$ and that $\tau_M$ acts by mapping $\tau^*p(t)$ to $(t-\theta)^{-n}p(t)^{(1)}$, for $p(t)\in F_\fp[t]$. Let $(\ell)\neq \fm$ be a maximal ideal of $A=\bF[t]$.
\begin{Theorem}\label{thm:true-for-carlitz}
Let $n$ be a non negative integer and let $\underline{M}=\underline{A}(n)$. Then Conjecture \ref{conjecture-reg} is true for $\underline{M}$.
\end{Theorem}

The original proof of Theorem \ref{thm:true-for-carlitz} only worked for $n$ a power of the characteristic of $\bF$. We are thankful to the referee for suggesting this new argument which fills our previous gap. 

\begin{proof}[Proof of Theorem \ref{thm:true-for-carlitz}]
Let $m\in M+\tau_M(\tau^*M)$ be such that there exists $f\in (M_{F_{\fp}^{\operatorname{ur}}})^{\wedge}_{\ell}$ for which $f-\tau_M(\tau^*f)=m$. That is, $m=(t-\theta)^{-n}p\in (t-\theta)^{-n}\cdot F_{\fp}[t]$ and $f\in F_{\fp}^{\operatorname{ur}}[t]^{\wedge}_{\ell}$. We let $s:=v_{\fp}(p)$ be the minimum of the valuation of the coefficients of the polynomial $p$; we have to show that up to subtracting an element of $(\id-\tau_M)(M)$ to $m$, \emph{i.e.} a polynomial of the form $(t-\theta)^{n}h-h^{(1)}$ to $p$, we can find $p$ with $v_{\fp}(p)\geq 0$. Suppose otherwise, and denote by $s<0$ the maximal valuation of such $p$. 

From the relation
\[
p=(t-\theta)^n f-f^{(1)},
\]
which we consider in $F_{\fp}^{\operatorname{ur}}[t]^{\wedge}_{\ell}=\bF(\!(\varpi)\!)^d[\![\ell]\!]$ where $d=[A/\ell:\bF]$, we deduce that the coefficients $(c_i)_{i\geq 0}$ in $\bF(\!(\varpi)\!)^d$  of $f$ as a formal power series in the variable $\ell$ are bounded. In particular, the value $v_{\fp}(f):=\inf_{i} \{v_{\fp}(c_i)\}$ is well-defined and, since $s<0$, satisfies $s=v_{\fp}(f^{(1)})=qv_{\fp}(f)$. This implies that $q$ divides $s$, from which one deduces the existence of a polynomial $h$ such that $v_\fp(p+(t-\theta)^n h-h^{(1)})>s$. But this contradicts the maximality of $s$.
\end{proof}

\appendix

\section{Exact categories}\label{sec:exact-categories}
Exact categories were introduced by Quillen in \cite{quillen} in order to define their $K$-groups. By definition, exact categories possess a class of short sequences satisfying several axioms which are modeled out of properties of exact sequences in abelian category; in that respect, short sequences in those classes are called short exact.

The \emph{embedding theorem} (\emph{cf}. \cite[\S A]{buhler}) states that any exact category which is small embeds fully and faithfully into an abelian category, in such a way that exact sequences coincide. This is presumably the main reason why the literature on this topic does not abound: many of the properties of exact categories can be deduced from their counterpart in abelian categories.

There are, however, some statements involved at certain key steps in the course of this text that we did not find clearly stated in references. Among those,
\begin{enumerate}[label=$-$]
\item There does not seem to be a consensus on what \emph{long exact sequences} and \emph{higher extension groups} mean in an exact category $\eC$. One could either take Yoneda extensions for definition (as we do below), but then it is not immediate that there are long exact sequences of extension groups attached to a short exact sequence. To palliate to this, one could define extension groups after a choice of an embedding $\eC\to \eA$ into an abelian category by the embedding theorem, and then define extension groups in $\eC$ to be the ones in $\eA$. While this would yield equivalent definitions for degree $1$ extensions, it is not clear what degree $i>1$ extensions are, nor that they are independent of the abelian embedding. 
\item We did not find references for what a \emph{monoidal exact category} means. We define the corresponding notion in this appendix, and show canonical isomorphisms
\[
\Ext^i(A\otimes X,B)\cong \Ext^i(A,B\otimes X^{\vee})
\]
whenever $X$ is a dualizable object. 
\end{enumerate}
Our main reference for what follows is the paper of B\"uhler \cite{buhler}.

\subsection{Definitions}
Let $\eC$ be an additive category. We fix a class $\operatorname{ex}(\eC)$ of sequences $S$ of composable arrows in $\eC$ of the form: 
\begin{equation}\label{eq:typical-exact-sequence}
S: \quad 0\longrightarrow X \stackrel{f}\longrightarrow Z \stackrel{g}{\longrightarrow} Y\longrightarrow 0
\end{equation}
such that $g$ is a cokernel of $f$ and $f$ is a kernel of $g$. Sequences in the class $\operatorname{ex}(\eC)$ will be called \emph{short exact}. Morphisms $f$ (resp. $g$) which are featured in an exact sequence $S$ as in \eqref{eq:typical-exact-sequence} will be called \emph{($\operatorname{ex}(\eC)$)-admissible mono} (resp. \emph{epi}). \\

The next definition, in a reduced form compared to Quillen's, is due to Yoneda.
\begin{Definition}[Exact category]\label{def:exact-category}
We say that $(\eC,\operatorname{ex}(\eC))$--or just $\eC$ for short--is an \emph{exact category} if $\operatorname{ex}(\eC)$ verifies the following list of axioms:
\begin{enumerate}[label=$\operatorname{Ex}\arabic*$]
\item\label{item:identity-admissible} For all objects $X$ of $\eC$, the identity morphism $\id_X$ is an admissible mono (resp. epi).
\item\label{item:composition-admissible} The class of admissible mono (resp. epi) is stable under composition.
\item\label{item:excat-pullback} The pushout (resp. pullback) of an admissible mono along a morphism with the same source (resp. target) is representable in $\eC$ and itself is an admissible mono (resp. epi).
\end{enumerate}
\end{Definition}

It can be shown that any sequence $S$ isomorphic to a sequence in $\operatorname{ex}(\eC)$ is itself in $\eC$. For any pair of objects $(A,B)$ of $\eC$, the sequence $0\to A\to A\oplus B\to B\to 0$ belongs to $\operatorname{ex}(\eC)$ and is called \emph{the canonical split sequence}. We shall call $S$ \emph{split} if $S$ is isomorphic to the canonical split sequence. 

\subsection{Exact functor and the embedding theorem}
Let $(\eC,\operatorname{ex}(\eC))$ and $(\eD,\operatorname{ex}(\eD))$ be exact categories and let $F:\eC\to \eD$ be an additive functor. 
\begin{Definition}
The functor $F$ is called \emph{exact} if $F(S)\in \operatorname{ex}(\eD)$ for any $S\in\operatorname{ex}(\eC)$. $F$ is said to \emph{reflect exactness} if $F$ is exact and if $S\in \operatorname{ex}(\eC)$ whenever $F(S)\in \operatorname{ex}(\eD)$. \\
We say that $F$ \emph{reflects admissible mono} (resp. \emph{epi}) if $f$ is admissible mono (resp. epi) if and only if $F(f)$ is. 
\end{Definition}

Any abelian category is endowed with the structure of an exact category in an evident way. Conversely, we record the Freyd--Mitchel embedding theorem \cite[Thm. A.1]{buhler}.
\begin{Theorem}[Embedding Theorem]\label{thm:embedding-thm}
Let $(\eC,\operatorname{ex}(\eC))$ be a small exact category. 
\begin{enumerate}[label=$(\roman*)$]
\item\label{item:abelian-embedding} There is an abelian category $\eA$ and a fully faithful exact functor $i:\eC\to \eA$ that reflects exactness. Moreover, $\eC$ is closed under extensions in $\eA$.
\item\label{item:abelian-embedding-reflects-epi} Assume moreover that $\eC$ is weakly idempotent complete\footnote{Recall that an additive is called \emph{weakly idempotent complete} if every retraction admits a kernel (equivalently, any coretraction admits a cokernel) \cite[\S 7]{buhler}.}. Then, one can choose $i:\eC\to \eA$ as in \ref{item:abelian-embedding} in such a way that it further reflects admissible epi. 
\end{enumerate}
\end{Theorem}

\subsection{The group of extensions}\label{subsec:the group of extensions}
Let $(A,B)$ be a pair of objects in $\eC$ and let $i$ be a positive integer. By a \emph{degree $i$ extension of $A$ by $B$} we mean a sequence of composable arrows in $\eC$ of the form 
\begin{equation}\label{eq:typical-degreeIextension}
S:0\to B\xrightarrow{f} E_1\xrightarrow{d_1}\cdots \xrightarrow{d_{i-1}} E_i\xrightarrow{g} A\to 0
\end{equation}
such that, for all $j\in\{1,\ldots,i-1\}$, there exists a factorization of $d_j$ into $E_j\to F_{j+1}\to E_{j+1}$ such that the sequences $S_j:0\to F_j\to E_j\to F_{j+1}\to 0$ are short exact. If we do not want to make reference to $A$ and $B$, we shall call \eqref{eq:typical-degreeIextension} a (long) exact sequence in $\eC$.\\
Observe that the short exact sequences $S_j$ are uniquely determined by $S$ up to isomorphism. In particular, $F_{i}=A$, $F_{i-1}=\ker(E_i\to A)$ and, more generally, $F_j=\ker(E_{j+1}\to F_{j+1})$.

\paragraph{(Cup product)}A degree $i$ extension $S$ of $A$ by $B$ and a degree $j$ extension $T$ of $B$ by $C$ \emph{compose} into a degree $i+j$ extension of $A $ by $C$ via the \emph{cup product}:
\[
T\cup S:0\to C\to F_1\to \cdots \to F_j\xrightarrow{u} E_1 \to \cdots \to E_i\to A\to 0 
\]
where we wrote $T:0\to C\to F_1\to \cdots \to F_j\to B\to 0$ and where $u$ is the composition of $F_j\to B\to E_1$. By definition, a degree $i$ extension $S$ can be written uniquely as the iterative cup product of degree $1$ extensions $S=\cup_j S_j$. 

\paragraph{(Equivalence)}Given another extension $S'$ of $A$ by $B$ of the same degree, we write $S\equiv S'$ if there is a commutative diagram in $\eC$
\begin{equation}
\begin{tikzcd}
S: & 0 \arrow[r] & B\arrow[r]\arrow[d,"\id"] & E_1 \arrow[r]\arrow[d] & \cdots \arrow[r]\arrow[d] & E_i\arrow[r]\arrow[d] & A\arrow[r]\arrow[d,"\id"] & 0 \\
S': & 0 \arrow[r] & B\arrow[r] & E'_1 \arrow[r] & \cdots \arrow[r] & E_i\arrow[r] & A\arrow[r] & 0.
\end{tikzcd}\nonumber
\end{equation}
More generally, we shall say that $A$ and $B$ \emph{are equivalent} and write $S\equiv S'$ if they are equivalent for the equivalence relation generated\footnote{That is, $S$ is equivalent to $S'$ if and only if there exists a sequence $S''$ such that $S\equiv S''$ and $S'\equiv S''$. In fact, one could show that, in the case $i=1$, $\equiv$ already is an equivalence relation as a version of the five lemma holds in exact categories \cite[Lem. 8.9]{buhler}.} by $\equiv$. One verifies that cup products are preserved under equivalences: $T'\cup S'\equiv T\cup S$ whenever $T'\equiv T$ and $S'\equiv S$. \\

The extensions $0\to B\to B\oplus A\to A\to 0$ for $i=1$ and $0\to B\stackrel{\id}{\to} B\stackrel{0}{\to} \cdots \stackrel{0}{\to} A \stackrel{\id}{\to} A\to 0$ for $i>1$ are called \emph{canonically split extension of degree $i$}. A degree $i$ extension $S$ is called \emph{split} if it is equivalent to the canonically split extension of degree $i$. One verifies that $S\cup T$ is split whenever $S$ or $T$ is. 

\paragraph{(Pullback and Pushforward)} Consider two morphisms $a:A'\to A$ and $b:B'\to B$ in $\eC$. By \ref{item:excat-pullback}, the pullback $E\times_A A'$ and the pushout $B'\sqcup_B E$ exist and their universal property yield that the canonical maps $B\to E\times_A A'$ and $B'\sqcup_B E\to A$ are a kernel and a cokernel of $E\times_A A'\to A'$ and $B'\to B'\sqcup_B E$ respectively. Therefore, for any short exact sequence $S:0\to B\to E\to A\to 0$, we obtain short exact sequences:
\begin{align*}
a^*S: &\quad 0\longrightarrow B\longrightarrow E\times_A A'\longrightarrow A'\longrightarrow 0, \\
b_*S: &\quad 0\longrightarrow B'\longrightarrow B'\sqcup_B E\longrightarrow A\longrightarrow 0,
\end{align*}
which are respectively degree $1$ extensions of $A'$ by $B$ and $A$ by $B'$. We call $a^*S$ and $b_* S$ the \emph{pullback of $S$ by $a$} and the \emph{pushout of $S$ by $b$}. Pullbacks are immediately extended to degree $i>1$ extensions $S$ by first writing it as a splitting $S'\cup S_i$ where $S_i$ (resp $S'$) is a degree $1$ (resp. degree $i-1$) extension, then by setting $a^*S$ to be $S'\cup a^*S_i$. Dually for pushforwards.\\
We leave the following list of facts without proof.
\begin{enumerate}[label=$(\alph*)$]
\item\label{item:pullback-pushout-split} A pullback or pushout of a split sequence is itself split.
\item We have $\id_A^*S\equiv S$ and $\id_{B*}S\equiv S$.
\item If $S\equiv S'$, then $a^*S\equiv a^*S'$ and $b_*S\equiv b_*S'$. 
\item\label{item:pullback-pushout-composition} If $a':A''\to A'$ and $b':B'\to B''$, then $(a'\circ a)^*S\equiv a'^*(a^*S)$ and $(b\circ b')^*S\equiv b^*(b'^*S)$.
\item\label{item:exact-functor} Given $F:\eC\to \eD$ an exact functor, $F(S')\equiv F(S)$ whenever $S'\equiv S$.
\item\label{item:exact-functor-pullpush} Given $F:\eC\to \eD$ an exact functor, $F(a)^*F(S)\equiv F(a^*S)$ and $F(b)_*F(S)\equiv F(b_* S)$.
\item We have $b_*(a^*S)\equiv a^*(b_*S)$ as extensions of $A'$ by $B'$.
\item\label{item:split-pullback} The sequences $f_*S$ and $g^*S$ are split, where $f$ and $g$ are as in \eqref{eq:typical-degreeIextension}.
\end{enumerate}

\paragraph{(Extension groups)} We assume that $\eC$ is small. As a consequence, equivalence classes of degree $i$ extensions of $A$ by $B$ form a set, pointed by the equivalence class of the split extension, which we denote by:
\[
\Ext^i_{\eC}(A,B).
\]
The cup product defines a map of pointed set $\Ext^i_{\eC}(A,B)\times \Ext^j_{\eC}(B,C)\to \Ext^{i+j}_{\eC}(A,C)$. Proprieties \ref{item:pullback-pushout-split}--\ref{item:pullback-pushout-composition} ensure that $\Ext^i_{\eC}$ defines a bifunctor $\eC^{\operatorname{op}}\times \eC\to \mathbf{Set}_*$ of pointed sets. By \ref{item:exact-functor}, an exact functor $F:\eC\to \eD$ induces a natural transformation of bifunctors of pointed sets $\Ext^i_{\eC}(F):\Ext^i_{\eC}(-,-)\to \Ext^i_{\eC}(F(-),F(-))$.

This construction can be refined into a bifunctor of abelian groups under the \emph{Baer sum}. If $\nabla_A:A\to A\oplus A$ denotes the diagonal embedding and $\Delta_B:B\oplus B\to B$ the addition, the \emph{degree $i$ Baer sum} is defined as the composition:
\[
+:\Ext^i_{\eC}(A,B)\times \Ext^i_{\eC}(A,B)\to \Ext^i_{\eC}(A\oplus A,B\oplus B)\xrightarrow{\nabla_A^* \circ \Delta_{B*}} \Ext^i_{\eC}(A,B). 
\]
The first map corresponds to direct sums of exact sequences, which are also exact \cite[Prop. 2.9]{buhler}. Then, $\Ext^i_{\eC}(A,B)$ becomes an abelian group under $+$ with the class of the split exact sequence as the identity. The cup product, as well as $\Ext^i_{\eC}(F)$, become morphisms of abelian groups in virtue of properties \ref{item:exact-functor}-\ref{item:exact-functor-pullpush}. If the category $\eC$ is $R$-linear for some commutative ring $R$, then $\Ext^i_{\eC}(A,B)$ becomes an $R$-module with $r\in R$ acting as pullback by the multiplication by $r$ on $A$. 

\subsection{Long exact sequence of Extensions groups}\label{subsec:Long exact sequence of Extensions groups}
We again assume that $\eC$ is small. Consider a short exact sequence $S:0\to X\xrightarrow{f} Y \xrightarrow{g} Z\to 0$. Let also $A$ be an object of $\eC$. Then, one can produce two natural complexes of abelian groups.\\
The \emph{pushout sequence associated to $S$ and $A$}:
\begin{equation}
\begin{tikzcd}
0 \arrow[r] & \Hom_{\eC}(A,X)\arrow[r,"f\circ -"] & \Hom_{\eC}(A,Y)\arrow[r,"g\circ -"] & \Hom_{\eC}(A,Z)\arrow[d,"h\mapsto h^* S"] \\
& \Ext^1_{\eC}(A,Z)\arrow[d,"T\mapsto S\cup T"'] & \Ext^1_{\eC}(A,Y)\arrow[l,"g_*"'] & \Ext^1_{\eC}(A,X)\arrow[l,"f_*"'] \\
& \Ext^2_{\eC}(A,X)\arrow[r,"f_*"] & \Ext^2_{\eC}(A,Y)\arrow[r,"g_*"] & \Ext^2_{\eC}(A,Z)\arrow[r,"T\mapsto S\cup T"] & \cdots
\end{tikzcd}\nonumber
\end{equation}
and the \emph{pullback sequence associated to $S$ and $A$}:
\begin{equation}
\begin{tikzcd}
0 \arrow[r] & \Hom_{\eC}(Z,A)\arrow[r,"-\circ g"] & \Hom_{\eC}(Y,A)\arrow[r,"-\circ f"] & \Hom_{\eC}(X,A)\arrow[d,"h\mapsto h_* S"] \\
& \Ext^1_{\eC}(X,A)\arrow[d,"T\mapsto S\cup T"'] & \Ext^1_{\eC}(Y,A)\arrow[l,"f^*"'] & \Ext^1_{\eC}(Z,A)\arrow[l,"g^*"'] \\
& \Ext^2_{\eC}(Z,A)\arrow[r,"g^*"] & \Ext^2_{\eC}(Y,A)\arrow[r,"f^*"] & \Ext^2_{\eC}(X,A)\arrow[r,"T\mapsto S\cup T"] & \cdots
\end{tikzcd}\nonumber
\end{equation}
That two successive arrows in these sequences compose to zero follows easily from properties \ref{item:pullback-pushout-split}--\ref{item:split-pullback} above. 

\begin{Proposition}\label{prop:vanishing-higher-ext}
Assume that the functor $\Ext^i_{\eC}(A,-)$ transforms admissible epi into surjections. Then, the functors $\Ext^n_{\eC}(A,-)$ are identically zero for $n>i$. Dually, if $\Ext^i_{\eC}(-,B)$ transforms admissible mono into surjections, the functors $\Ext^n_{\eC}(-,B)$ are identically zero for $n>i$.
\end{Proposition}
\begin{proof}
The class of any extension $U$ in $\Ext^{i+1}_{\eC}(A,B)$ is decomposed as the cup product of a degree one extension $S$ and a degree $i$ extension $T$ whose class are in $\Ext^1_{\eC}(E,B)$ and $\Ext^{i}_{\eC}(A,E)$ respectively for some object $E$. The pushout sequence associated to $S$ and $A$ together with the surjectivity of $g_*$ imply that the map
\[
S\cup- :\Ext^i_{\eC}(A,E)\longrightarrow \Ext^{i+1}_{\eC}(A,B)
\]
starred in the pushout sequence, is zero. Since $U$ is in its image, its class is zero. Therefore, $\Ext^{i+1}_{\eC}(A,B)=(0)$.

For the general case, observe similarly that a degree $n>i+1$ extension $U$ decomposes as the cup product of a degree $i+1$ and $n-(i+1)$ extensions respectively. The class of the first is zero, hence so is the class of $U$.
\end{proof}

The author is not aware of a reference that states the exactness of the pushout and pullback long sequences at this level of generality. However, we prove under reasonable assumptions that they are. 

\begin{Proposition}\label{prop:pull-push-are-exact}
Assume that $\eC$ is a small exact category that is weakly idempotent complete. Then pushout and pullback sequences are exact.
\end{Proposition}

We begin with a lemma.
\begin{Lemma}\label{lem:conciliate-definitions}
Let $h:\eC\to \eD$ be an embedding of $\eC$ into an exact category $\eD$ which reflects exactness and admissible epi. Let $S:0\to X_n\to X_{n-1}\to \cdots \to X_1\to 0$ be a sequence in $\eC$. The following are equivalent:
\begin{enumerate}[label=$(\roman*)$]
\item\label{item:Grayson-exact} $S$ is exact,
\item\label{item:h-exact} $h(S)$ is exact in $\eD$.
\end{enumerate} 
\end{Lemma}
\begin{proof}
If $S$ is exact, then $S=\cup_i S_i$ for short exact sequences $S_i$, hence $h(S)=\cup_i h(S_i)$. Because $h$ reflects exactness, this proves \ref{item:Grayson-exact}$\Longrightarrow$\ref{item:h-exact}. \\
We prove the converse by induction on the length $n$ of the sequence; for $n\leq 3$, there is nothing to prove. For $n>3$,  suppose $h(S)$ is exact. We splice up $h(S)$ into exact sequences as
\begin{equation}\label{eq:splice-h}
\left(0\to h(X_n)\to \cdots \to h(X_2)\to Z\to 0\right)\cup \left(0\to Z\to h(X_1)\xrightarrow{h(d_1)} h(X_0)\to 0\right).
\end{equation}
In particular $h(d_1)$ is an admissible epi and so is $d_1$ by reflection. Hence $d_1$ admits a kernel $Y_1$ in $\eC$ and the sequence $V:0\to Y_1\to X_1\to X_0\to 0$ is short exact. By universality of the kernel, the splicing \eqref{eq:splice-h} corresponds to $h(S)=h(U)\cup h(V)$ for some sequence $U$ in $\eC$:
\[
U:0\to X_n\xrightarrow{d_n} X_{n-1}\to \cdots \to X_2\to Y_1\to 0.
\]
Since $h(U)$ is exact, so is $U$ by induction hypothesis. Since $V$ is exact, so is their cup product $U\cup V=S$. 
\end{proof}

\begin{proof}[Proof of Proposition \ref{prop:pull-push-are-exact}]
It suffices to apply the statement of Lemma \ref{lem:conciliate-definitions} to the functor $i:\eC\to \eA$ produced by the Freyd--Mitchell embedding in \ref{item:abelian-embedding-reflects-epi}. In particular, the pushout and pullback long sequences become the respective ones for the abelian category $\eA$ for which the statement is known. 
\end{proof}

\subsection{Extensions and adjunctions}
Let $(\eC,\operatorname{ex}(\eC))$ and $(\eD,\operatorname{ex}(\eD))$ be exact categories. Let $F:\eC\to \eD$ be an exact functor that admits a right-adjoint $G:\eD\to \eC$ which we assume to be exact as well (it is a priori unclear that this holds). We denote by $\eta:\id_{\eD}\to GF$ and $\varepsilon:FG\to \id_{\eC}$ the unit and counit of this adjunction. In particular, given any objects $A$, $B$ of $\eC$ and $\eD$ respectively, there are maps of groups:
\begin{align*}
g:& \Ext^i_{\eD}(F(A),B)\xrightarrow{G}\Ext^i_{\eC}(GF(A),G(B))\xrightarrow{\eta^*_A} \Ext^i_{\eC}(A,G(B)), \\
f:& \Ext^i_{\eC}(A,G(B))\xrightarrow{F}\Ext^i_{\eD}(F(A),FG(B))\xrightarrow{(\varepsilon_B)_*} \Ext^i_{\eD}(F(A),B).
\end{align*}
\begin{Proposition}\label{prop:ext-adjunction}
The maps $f$ and $g$ are mutually inverse. In particular, for all $i> 0$, there are isomorphisms of groups:
\[
\Ext^i_{\eD}(F(A),B)\cong \Ext^i_{\eC}(A,G(B)).
\]
\end{Proposition}
\begin{proof}
We only prove that $f\circ g$ is equivalent to the identity as the dual statement is proven along the same lines. We consider an extension in $\Ext^i_{\eD}(F(A),B)$ represented by a long exact sequence $S:0\to B\to E_1 \to \cdots \to E_i \to F(A)\to 0$. We have the following commutative diagram, whose rows are exact by assumption:
\begin{equation}
\begin{tikzcd}
0 \arrow[r] & B  \arrow[dr, phantom, "\square" description]\arrow[r] & E_1 \arrow[r] & \cdots \arrow[r] & E_i \arrow[r] & F(A) \arrow[r] & 0  \\
0 \arrow[r] & FG(B) \arrow[r]\arrow[u,"\varepsilon_B"] & FG(E_1) \arrow[r]\arrow[u,"\varepsilon_{E_1}"] & \cdots \arrow[r]\arrow[u,"\varepsilon"] & FG(E_i) \arrow[r]\arrow[u,"\varepsilon_{E_i}"] & FGF(A) \arrow[r]\arrow[u,"\varepsilon_{F(A)}"] & 0 \\
0 \arrow[r] & FG(B) \arrow[r]\arrow[u,equal]\arrow[d,"\varepsilon_B"] & FG(E_1) \arrow[r]\arrow[u,equal]\arrow[d] & \cdots \arrow[r]\arrow[u,equal]\arrow[d,equal] & F(*) \arrow[ur, phantom, "\urcorner", very near start]\arrow[r]\arrow[u,"F(t)"]\arrow[d,equal] & F(A) \arrow[r]\arrow[u,"F(\eta_A)"]\arrow[d,equal] & 0 \\
0 \arrow[r] & B \arrow[r] & **  \arrow[ul, phantom, "\ulcorner", very near start]\arrow[r] & \cdots \arrow[r] & F(*) \arrow[r] & F(A) \arrow[r] & 0
\end{tikzcd}
\nonumber
\end{equation}
Above, we denoted by $*$ the pullback of $G(E_i)\to GF(A)\leftarrow A$ and by $**$ the pushforward of $B\leftarrow FG(B)\rightarrow FG(E_1)$, as indicated by the symbols $\urcorner$ and $\ulcorner$.\\
This diagram corresponds to $S\stackrel{\varepsilon}{\leftarrow} FG(S)\leftarrow F(g(S)) \rightarrow f(g(S))$ vertically. Therefore, it suffices to show that the bottom row is equivalent to the first one. 

By commutativity of the square $\square$ and universal property of the pushforward, there exists a map $s:**\to E_1$ whose composition with $FG(E_1)\to **$ is the map $\varepsilon_{E_1}$ and whose composition with $B\to **$ is $B\to E_1$ starred in $S$. In particular, we get a commutative diagram 
\begin{equation}
\begin{tikzcd}
0 \arrow[r] & B \arrow[r]\arrow[d,equal] & ** \arrow[r]\arrow[d,"s"] & \cdots \arrow[r]\arrow[d,"\varepsilon"] & FG(E_i) \arrow[r]\arrow[d,"\varepsilon_{E_i}\circ F(t)"] & F(A) \arrow[r]\arrow[d,"\varepsilon_{F(A)}\circ F(\eta_A)"] & 0\\
0 \arrow[r] & B  \arrow[r] & E_1 \arrow[r] & \cdots \arrow[r] & E_i \arrow[r] & F(A) \arrow[r] & 0 
\end{tikzcd}
\nonumber
\end{equation}
To conclude that the above diagram is an equivalence among $f(g(S))$ and $S$, it remains to show that $\varepsilon_{F(A)}\circ F(\eta_A)=\id_{F(A)}$. This results from diagram:
\begin{equation}
\begin{tikzcd}
\Hom_{\eD}(F(A),FGF(A)) \times \Hom_{\eD}(FGF(A),F(A))  \arrow[r] & \Hom_{\eD}(F(A),F(A)) \\
\Hom_{\eC}(A,GF(A)) \arrow[u,"F\times \operatorname{ad}" description]\times \Hom_{\eC}(GF(A),GF(A)) \arrow[r] &\Hom_{\eC}(A,GF(A))\arrow[u,"\operatorname{ad}"]
\end{tikzcd}
\nonumber
\end{equation}
which commutes, as the adjunction maps $\operatorname{ad}$ are obtained by composing $F$ with the natural transformation $FG\to \id_{\eD}$. It remains to apply the commutativity of the above square to the couple $(\eta_A,\id_{GF(A)})$ sitting in the lower left corner. Via the lower path, it is mapped to $\eta_A$  and then to $\id_{F(A)}$. Via the upper path, to $(F(\eta_A),\varepsilon_{F(A)})$, then to $\varepsilon_{F(A)}\circ F(\eta_A)$. Hence $\varepsilon_{F(A)}\circ F(\eta_A)=\id_{F(A)}$ as announced.
\end{proof}

\subsection{Construction of exact categories}
We present a way of constructing exact categories out of certain functors. This procedure will be used in the sequel to endow the category of $A$-motives with the structure of an exact category (\emph{cf}. Definition \ref{def:exact-sequence}). Let $(\eD,\operatorname{ex}(\eD))$ be an exact category and let $F:\eC\to \eD$ be an additive functor. 

\begin{Proposition}\label{prop:Functor-to-exactCat}
Assume the following:
\begin{enumerate}[label=$(\roman*)$]
\item\label{item:admits-kernel-of-epi} Given $g:Y\to Z$ be a morphism in $\eC$, if $F(g)$ is an admissible epi of $\eD$, then $g$ admits a kernel $\ker g\to Y$ in $\eC$ and $F(\ker g\to Y)$ is a kernel of $F(g)$.
\item\label{item:admits-cokernel-of-mono} Given $f:X\to Y$ be a morphism in $\eC$, if $F(f)$ is an admissible mono of $\eD$, then $f$ admits a cokernel $Y\to\coker f$ in $\eC$ and $F(Y\to\coker f)$ is a cokernel of $F(f)$.
\end{enumerate}
Then there exists a unique structure of exact category on $\eC$ making $F$ a functor which reflects exactness. In fact, $F$ further reflects admissible epi and mono.
\end{Proposition}
\begin{proof}
Since $F$ has to reflect exactness, the candidate class is uniquely determined to be the class $\operatorname{ex}(\eC)$ of short sequences $S:0\to X\stackrel{f}{\to} Y\stackrel{g}{\to} Z \to 0$ such that $f$ is a kernel of $g$, $g$ is a cokernel of $f$, and $F(S)$ is exact in $\eD$.

Assumptions \ref{item:admits-kernel-of-epi} and \ref{item:admits-cokernel-of-mono} ensure that for short sequences of $\operatorname{ex}(\eC)$ the former non trivial arrow is a kernel of the latter, and dually. It remains to show the axioms of Definition \ref{def:exact-sequence}. Axioms \ref{item:identity-admissible} and \ref{item:composition-admissible} are clear as $F$ is a functor. To prove \ref{item:excat-pullback}, we fix a short sequence $S:0\to X\to Y\to Z\to 0$ in $\operatorname{ex}(\eC)$ and a morphism $Z'\to Z$. We claim that the pullback $Y\times_Z Z'$ is representable in $\eC$; indeed, $F(Y)\times_{F(Z)}F(Z')$ is representable in $\eD$ and hence $F(Y)\oplus F(Z')=F(Y\oplus Z')\to F(Z)$ admits a kernel. By the \emph{obscure axioms} in $\eD$ \cite[Prop. 2.16]{buhler}, it is an admissible epi since the composition $F(Y)\to F(Y\oplus Z')\to F(Z)$ is. Hence, $Y\oplus Z'\to Z$ admits a kernel by \ref{item:admits-kernel-of-epi}--this kernel represents the pullback $Y\times_Z Z'$--and we have $F(Y\times_Z Z')\cong F(Y)\times_{F(Z)}F(Z')$. \\
In particular, by \ref{item:excat-pullback} for $\eD$, $F(Y\times_Z Z'\to Z)$ is an admissible epi, hence $Y\times_Z Z'\to Z$ admits a kernel in $\eC$ by \ref{item:admits-kernel-of-epi} and $\ker F(Y\times_Z Z'\to Z)\cong F(\ker(Y\times_Z Z'\to Z))$. That is, the sequence 
\[
F(0\to \ker(Y\times_Z Z'\to Z)\to Y\times_Z Z'\to Z\to 0)
\]
belongs to $\operatorname{ex}(\eD)$, hence $0\to \ker(Y\times_Z Z'\to Z)\to Y\times_Z Z'\to Z\to 0$ is in $\operatorname{ex}(\eC)$. The dual statement follows similarly.

That $F$ further reflects admissible mono and epi is clear from the class $\operatorname{ex}(\eC)$.
\end{proof}

\subsection{Monoidal exact categories}\label{subsec:monoidal-exact-cat}
Let $(\eC,\operatorname{ex}(\eC))$ be an exact category and let $\otimes:\eC\times \eC\to \eC$ be a bifunctor making $(\eC,\otimes)$ an additive monoidal category \cite[0FNA]{stack}. We denote by $\mathbbm{1}$ a neutral object for $\otimes$. 

\begin{Definition}\label{def:monoidal-exact}
The datum of $(\eC,\operatorname{ex}(\eC),\otimes)$ is called a \emph{monoidal exact category} if, given any object $X$ of $\eC$, the endofunctors $-\otimes X$ and $X\otimes -$ are exact.
\end{Definition}

It is immediate that, given any couple of objects $(A,B)$ in $\eC$, the functor $-\otimes X$ determines group morphisms
\begin{equation}\label{eq:monoidal-structure-ext}
-\otimes X: \Ext^i_{\eC}(A,B)\longrightarrow \Ext^i_{\eC}(A\otimes X,B\otimes X), \quad [S]\longrightarrow [S\otimes X].
\end{equation}
Recall that $X$ is called \emph{dualizable} if the functor $-\otimes X$ admits a right-adjoint, denoted by $\cHom(X,-)$. We denote by $X^{\vee}$ the object $\cHom(X,\mathbbm{1})$, called \emph{the dual of $X$}. Since $\cHom(X,-)\cong X^{\vee}\otimes -$, the functor $\cHom(X,-)$ is exact. In particular, as a corollary of Proposition \ref{prop:ext-adjunction}, we obtain:
\begin{Corollary}\label{cor:dualizable-tensor}
Let $X$ be a dualizable object of $\eC$. Then,  there are canonical isomorphisms of groups
\[
\Ext^i_{\eC}(A\otimes X,B)\cong \Ext^i_{\eC}(A,B\otimes X^{\vee}).
\]
\end{Corollary}


\begin{thebibliography}{111111}

\bibitem[Stack]{stack} \emph{The Stacks Project}, Website, available at \href{http://stacks.math.columbia.edu/}{http://stacks.math.columbia.edu/}

\bibitem[GAnd]{anderson} G. W. Anderson, \emph{$t$-motives}, Duke Math. J. 53, no. 2 (1986).

\bibitem[YAnd]{andre} Y. Andr\'e, \emph{Slope filtrations}, Confluentes Mathematici Vol. 1 (2009).

\bibitem[Bei1]{beilinson} A. A. Beilinson, \emph{Notes on absolute Hodge cohomology}, Applications of algebraic $K$-theory to Algebraic geometry and number theory (1986).

\bibitem[Bei2]{beilinson-L} A. A. Beilinson, \emph{Higher regulators and values of $L$-functions}, Journal of Soviet Mathematics 30 (1985).

\bibitem[BoHa]{hartl-bornhofen} M. Bornhofen, U. Hartl, \emph{Pure Anderson motives and abelian $\tau$-sheaves}, Math. Z. 268 (2011).

\bibitem[Bou]{bourbaki} N. Bourbaki, \emph{Alg\`ebre (A), Alg\`ebre commutative (AC)}.

\bibitem[B\"uh10]{buhler} T. Bühler, \emph{Exact categories}, Expositiones Mathematicae, Vol. 28, 1, 1--69 (2010)

\bibitem[BrPa]{brownawell-papanikolas} W.D. Brownawell, M. A Papanikolas, \emph{A rapid introduction to Drinfeld modules, $t$-modules and $t$-motives}, $t$-Motives: Hodge Structures, Transcendence, and Other Motivic Aspects, European Mathematical Society, Z\"urich (2020).

\bibitem[Del]{deligne-droite} P. Deligne, \emph{Le Groupe Fondamental de la Droite Projective Moins Trois Points}, Galois Groups over $\bQ$ (1989).

\bibitem[Eis]{eisenbud} D. Eisenbud, \emph{Commutative Algebra with a View Toward Algebraic Geometry}, Graduate Texts in Mathematics (1995).

\bibitem[Gar1]{gardeyn1} F. Gardeyn, \emph{A Galois criterion for good reduction of $\tau$-sheaves}, J. Number Theory Vol. 97, 447--471 (2002).

\bibitem[Gar2]{gardeyn2} F. Gardeyn, \emph{The structure of analytic $\tau$-sheaves}, J. Number Theory Vol. 100, 332--362 (2003).

\bibitem[Gaz]{gazda2} Q. Gazda, \emph{Regulators in the arithmetic of function fields}, Preprint, available at \href{https://arxiv.org/abs/2207.03461}{arXiv:2207.03461} (2022).

\bibitem[GaMa]{gazda-maurischat-ext} Q. Gazda, A. Maurischat, \emph{Carlitz twists: their motivic cohomology, regulators, zeta values and polylogarithms}, Preprint, available at \href{https://arxiv.org/abs/2212.02972}{arXiv:2212.02972} (2023).

\bibitem[Har1]{hartl-dico} U. Hartl, \emph{A dictionary between Fontaine-theory and its analogue in equal characteristic}, J. Number Theory Vol. 129, 7, 1734--1757 (2009).

\bibitem[Har2]{hartl} U. Hartl, \emph{Period Spaces for Hodge Structures in Equal Characteristic}, Ann. of Math 173, n. 3 (2011).

\bibitem[Har3]{hartl-isogeny} U. Hartl, \emph{Isogenies of abelian Anderson $A$-modules and $A$-motives}, Annali della Scuola Normale Superiore di Pisa, Classe di Scienze XIX (2019).

\bibitem[HaJu]{hartl-juschka} U. Hartl, A.-K. Juschka, \emph{Pink's Theory of Hodge Structures and the Hodge Conjecture over Function Fields}, $t$-motives: Hodge structures, transcendence and other motivic aspects, EMS Congress Reports, European Mathematical Society (2020).

\bibitem[Jan]{jannsen} U. Jannsen, \emph{Mixed Motives and Algebraic $K$-Theory}, Lecture Notes in Math. 1400 (1990).

\bibitem[Kat1]{katz} N. Katz, \emph{Une nouvelle formule de congruence pour la fonction $\zeta$}, \'Expos\'e XXII SGA7 t. II, Lecture Notes in Math. 340 (1973).

\bibitem[Kat2]{katz-modular} N. Katz, \emph{$p$-adic Properties of Modular Schemes and Modular Forms}, Lecture Notes in Math. 350 (1973).

\bibitem[Kat3]{katz-slope} N. Katz, \emph{Slope filtration of $F$-crystals}, Ast\'erisque, tome 63, 113--163 (1979).

\bibitem[Lau]{laumon} G. Laumon, \emph{Cohomology of Drinfeld modular varieties. Part I.}, Cambridge Studies in Advanced Mathematics, 41. Cambridge University Press, Cambridge (1996).

\bibitem[Mor1]{mornev-shtuka} M. Mornev, \emph{Shtuka cohomology and Goss $L$-values}, PhD thesis, available at \href{https://arxiv.org/pdf/1808.00839.pdf}{arXiv:1808.00839} (2018).

\bibitem[Mor2]{mornev-isocrystal} M. Mornev, \emph{Tate modules of isocrystals and good reduction of Drinfeld modules}, Algebra and Number Theory Vol. 15 no. 4 (2021).

\bibitem[Pap]{papanikolas} M. A. Papanikolas, \emph{Tannakian duality for Anderson-Drinfeld motives and algebraic independence of Carlitz logarithms}, Invent. Math. 171 no. 1 (2008).
 
\bibitem[PaRa]{papanikolas-ramachandran} M. A. Papanikolas, N. Ramachandran, \emph{A Weil--Barsotti formula for Drinfeld modules}, J. Number Theory Vol. 98, 2, 407--431 (2003).

\bibitem[Pin]{pink} R. Pink, \emph{Hodge Structures over Function Fields}, Preprint, available at \href{https://people.math.ethz.ch/~pink/ftp/HS.pdf}{https://people.math.ethz.ch/~pink/ftp/HS.pdf} (1997).

\bibitem[Qui1]{quillen} D. Quillen, \emph{Higher algebraic $K$-theory. I}, in Algebraic $K$-theory, I: Higher $K$-Theories (Proc. Conf., Battelle Memorial Inst., Seattle, Wash., 1972), Lecture Notes in Math. 341 (1973).

\bibitem[Qui2]{quillen-projective} D. Quillen, \emph{Projective modules over polynomial rings}. Invent. Math. 36 (1976).

\bibitem[Ros]{rosen} M. Rosen, \emph{Number theory in function fields}, Graduate Texts in Mathematics 210 (2002).

\bibitem[Sch]{scholl} A. J. Scholl, \emph{Remarks on special values of $L$-functions}, $L$-functions and Arithmetic, Proceedings of the Durham Symposium, July, 1989 (J. Coates and M. J. Taylor, eds.), London Mathematics Society Lecture Notes Series, vol. 153 (1991).

\bibitem[Ser]{serre} J.-P. Serre, \emph{Corps locaux}, Publications de l'Universit\'{e} de Nancago, No. VIII (1968).

\bibitem[Tae1]{taelman} L. Taelman, \emph{Artin $t$-motifs}, J. Number Theory Vol. 129, 1, 142--157 (2009).

\bibitem[Tae2]{taelman-dirichlet} L. Taelman, \emph{A Dirichlet unit theorem for Drinfeld modules}, Math. Ann. 348 (2010).

\bibitem[Tae3]{taelman-woods} L. Taelman, \emph{Sheaves and functions modulo $p$, lectures on the Woods-Hole trace formula}, LMS Lecture Note Series Vol. 429 (2015).

\bibitem[Tae4]{taelman-t-motif} L. Taelman, \emph{$1$-$t$-motifs}, $t$-motives: Hodge structures, transcendence and other motivic aspects, EMS Congress Reports, European Mathematical Society (2020).

\bibitem[TaWa]{taguchi-wan} Y. Tagushi, D. Wan, \emph{$L$-functions of $\varphi$-sheaves and Drinfeld modules}, Journal of the American Mathematical Society, vol. 9, no. 3 (1996).

\bibitem[Tag]{taguchi} Y. Taguchi, \emph{The Tate conjecture for $t$-motives}, Proc. Amer. Math. Soc. 123, no. 11 (1995).

\bibitem[Tam]{tamagawa} A. Tamagawa, \emph{Generalization of Anderson's $t$-motives and Tate conjecture}, in Moduli Spaces, Galois Representations and $L$-Functions, S\"urikaisekikenky\"uho K\"okyu\"oroku, no. 884 (1994).

\bibitem[tMo]{motif} EMS Series of Congress Reports, \emph{$t$-Motives: Hodge Structures, Transcendence and Other Motivic Aspects}, Editors: G. B\"ockle, D. Goss, U. Hartl, M. A. Papanikolas (2020).

\bibitem[Wei]{weibel} C. Weibel, \emph{An Introduction to Homological Algebra}, Cambridge Studies in Advanced Mathematics (1994).

\end{thebibliography}
\end{document}